\documentclass[11pt]{article}

\usepackage[a4paper]{geometry}
\usepackage{lmodern}        
\usepackage{graphicx,color}
\usepackage{xcolor}
\usepackage{enumitem}       
\usepackage{multirow}       
\usepackage{float}
\usepackage{amsmath,amssymb,amsfonts,amsthm,mathtools,bm}
\usepackage{dsfont}
\usepackage{bbm}
\usepackage{comment}
\usepackage{a4wide}

\usepackage{hyperref}
\usepackage{cleveref}

\newtheorem{proposition}{Proposition}[section]
\newtheorem{theorem}[proposition]{Theorem}
\newtheorem{lemma}[proposition]{Lemma}
\newtheorem{corollary}[proposition]{Corollary}
\newtheorem{assumptions}[proposition]{Assumptions}
\theoremstyle{definition}
\newtheorem{remark}[proposition]{Remark}

\newtheorem{definition}[proposition]{Definition}
\theoremstyle{remark}

\newenvironment{proofof}[1]{\smallskip\noindent{\textbf{Proof~of~#1.}}%
  \hspace{1pt}}{\hspace{-5pt}{\nobreak\quad\nobreak\hfill\nobreak%
    $\square$\vspace{2pt}\par}\smallskip\goodbreak}

\numberwithin{equation}{section}
\numberwithin{figure}{section}
\numberwithin{table}{section}

\crefname{proposition}{proposition}{propositions}
\Crefname{proposition}{Proposition}{Propositions}

\crefname{theorem}{theorem}{theorems}
\Crefname{theorem}{Theorem}{Theorems}

\crefname{lemma}{lemma}{lemmas}
\Crefname{lemma}{Lemma}{Lemmas}

\crefname{corollary}{corollary}{corollaries}
\Crefname{corollary}{Corollary}{Corollaries}

\crefname{definition}{definition}{definitions}
\Crefname{definition}{Definition}{Definitions}

\crefname{remark}{remark}{remarks}
\Crefname{remark}{Remark}{Remarks}

\crefname{example}{example}{examples}
\Crefname{example}{Example}{Examples}

\Crefname{assumptions}{Assumptions}{assumptions}

\allowdisplaybreaks[4]
\newcommand{\RR}{\mathbb{R}}
\newcommand{\ve}{\varepsilon}

\begin{document}
\title{A well-posedness theory for the 2D inhomogeneous incompressible Euler equations and a class of global-in-time solutions}
\author{Ibrahim Suleiman\footnote{Courant Institute School of Mathematics, Computing and Data Science, New York University, 251 Mercer Street, New York, NY 10012, USA, \texttt{is1647@nyu.edu}}}
\maketitle

 \begin{abstract}
     We consider the Cauchy problem for the 2D inhomogeneous incompressible Euler equations. We show that the Cauchy problem is locally well-posed in a class of $m$-fold symmetric functions at critical regularity and allowing the velocity field to grow linearly at spatial infinity. This was a problem posed by P.L. Lions in the homogeneous, not necessarily symmetric case. A key ingredient is a variable-coefficient elliptic theory developed using quasiconformal mappings and Schauder estimates in certain scale-invariant H\"{o}lder spaces. Local well-posedness allows us to derive the one-dimensional system on $\mathbb{S}^1$ satisfied by solutions with $0$-homogeneous vorticity and density. For this system, we can obtain global well-posedness for a class of initial data away from equilibrium. We are unaware of any previous nonperturbative class of time-dependent solutions to the two-dimensional inhomogeneous Euler equations for which global well-posedness has been established. Finally, for some signed global solutions, we show a dichotomy between density homogenization and averaged relaxation in infinite time.
 \end{abstract}

\section{Introduction}
In this paper, we consider the inhomogeneous incompressible Euler equations in $\RR^2$
\begin{equation}
    \label{eqn:inhomogeneous-2deuler}
    \begin{rcases}
        \partial_t \rho + u \cdot \nabla \rho  = 0 \\
        \partial_t(\rho u) + u \cdot \nabla(\rho u) + \nabla p = 0\\
        \nabla \cdot u = 0
    \end{rcases}
\end{equation}
The system governs the motion of an incompressible inhomogeneous ideal fluid with density $\rho: [0, \infty) \times  \mathbb{R}^2 \to \RR_+$, velocity field  $u: [0, \infty) \times  \mathbb{R}^2 \to \RR^2$,  and scalar pressure $p: [0, \infty) \times  \mathbb{R}^2 \to \RR$. Here $u = u(t, x), \rho = \rho(t, x)$, and $p = p(t,x)$, where $t \ge 0$ and $x = (x_1, x_2)  \in \mathbb{R}^2$. We are interested in the Cauchy problem for \eqref{eqn:inhomogeneous-2deuler}. That is, given divergence-free initial data $u_0:\mathbb{R}^2 \to \RR^2$ and $\rho_0:\mathbb{R}^2 \to \RR_+$ with $0 < c_0 \le \rho_0 \le c_1$, we seek $\rho, u$ and $p$ satisfying \eqref{eqn:inhomogeneous-2deuler} in $(0, \infty) \times \RR^2$    and $\rho(0, x) = \rho_0(x), u(0,x) = u_0(x)$. It is known in various settings that the Cauchy problem for \eqref{eqn:inhomogeneous-2deuler} is well-posed locally in time. We refer the reader to the references \cite{beirao1980euler}, \cite{beir1980euler}, \cite{chae2003local},\cite{danchin2010well}, \cite{danchin2011well}, \cite{wei2013local}, \cite{itoh1994cauchy}. There are long time existence results in \cite{danchin2011well} and \cite{bae2025long}. In the perturbative regime, there are only a few results on global well-posedness: the work \cite{chen2023nonlinear}, set in the periodic channel $\mathbb{T} \times \mathbb{R}$,  which establishes global well-posedness when the data is a small Gevrey perturbation of the Couette flow $u = (x_2, 0)$ with constant density $\rho \equiv 1$; and the related work \cite{zhao2025inviscid}, albeit set in the finite channel $\mathbb{T} \times [-1, 1]$,  which treats Gevrey perturbations of certain monotone shear flows $u=(f(x_2),0)$ with density $\rho = \theta(x_2)$; and the work \cite{zhao2026traveling} constructs traveling wave solutions with Sobolev regularity near monotone shear flows $u = (f(x_2), 0)$ with density $\rho = \theta(x_2)$, also set in $\mathbb{T} \times [-1, 1]$.  The works on local well-posedness often feature continuation criteria for the local solution, but it is an outstanding problem to decide in general whether or not the Cauchy problem is globally well-posed in time. The recent result of \cite{córdoba2026vorticityblowup2dincompressible} shows finite time singularity formation for the forced system. 

In contrast, in the case where $\rho$ is identically constant, i.e. the homogeneous incompressible Euler equations, global well-posedness in $\RR^2$ has been known since the work of Wolibner \cite{wolibner1933theoreme} from the 1930s. The result relies on the fact that the vorticity $\omega := \nabla^\perp \cdot u = \partial_{x_1}u_2 - \partial_{x_2}u_1$ is transported by the flow $u$ in 2 dimensions:
\begin{equation}
\label{eqn:hom-vort}
    \partial_t \omega + u\cdot\nabla \omega = 0
\end{equation}
and thus its $L^\infty$ norm is conserved. In fact, it is well known that given $\omega$ having sufficient decay at infinity (for instance $\omega \in L^1 \cap L^\infty$), one can uniquely determine the velocity field $u$ decaying at infinity via the Biot-Savart law
\begin{equation}
\label{eqn:hom-euler-bS-form}
    u(x) = \frac{1}{2\pi}\int_{\RR^2} \frac{(x-y)^\perp}{|x-y|^2}\omega(y)\ \mathrm{d}y.
\end{equation}

Thus one can formulate the Cauchy problem for the homogeneous Euler equations using the equations \eqref{eqn:hom-vort} and \eqref{eqn:hom-euler-bS-form} supplemented with initial data $\omega_0  = \nabla^\perp \cdot u_0$, where $\nabla \cdot u_0 = 0$. This is the vorticity formulation of the 2 dimensional Euler equations, from which a global existence and uniqueness theory can be developed. The celebrated theory of Yudovich \cite{yudovich1963non} indeed shows global existence and uniqueness of appropriately defined weak solutions in $L^1 \cap L^\infty$. A drawback of the foregoing formulation is that the initial vorticity and velocity must decay at spatial infinity to make sense of \eqref{eqn:hom-euler-bS-form}, which is physically not reasonable.  As stated for instance in \cite{ambrose2015serfati}, ``\dots full plane flow is an approximate model for flow far from boundaries, where no decay of distant vorticity should be expected".  Other well-posedness theories that have been developed do address this drawback. The first of these is the work of Benedetto, Marchioro and Pulvirenti\cite{benedetto1993euler}, who considered flows with vorticity in $L^p \cap L^\infty$, $1\le p < \infty$ and velocity satisfying $|u(t,x)| \le C(1+|x|^\beta)$, $\beta < 1$. See also the related work of \cite{brunelli2010euler}. In \cite[p.~146]{lions1996mathematical}, Lions asks whether one can obtain well-posedness for vorticity in $L^\infty$ and velocity satisfying $|u(t,x)| \le C(1+|x|)$.  We also mention the work of Serfati \cite{serfati1995solutions}, who discovered a replacement for the Biot-Savart law, permitting the construction of solutions with merely bounded initial velocity and vorticity. Similar results can be found in \cite{kelliher2015characterization}, \cite{ambrose2015serfati} and references therein. The most relevant theory to the current paper however, is the work of Elgindi and Jeong \cite{elgindi2020symmetries}, who present a well-posedness theory in the class of bounded vorticities possessing a discrete rotational symmetry. We describe the relevant results below to motivate the current paper. 
\subsection{Homogeneous Euler equations with discrete rotational symmetry}
We begin with some important definitions. 
\begin{definition}
Let $m \in \mathbb{Z}^+$. A function $f:\RR^2 \to \RR$ is said to be $m$-fold symmetric if $f(\mathcal{O}_mx) = f(x)$ for all $x \in \RR^2$, where $\mathcal{O}_m \in SO_2$ is the matrix corresponding to counterclockwise rotation by $\frac{2\pi}{m}$. A vector field $v:\RR^2 \to \RR^2$ is said to be $m$-fold symmetric if $v(\mathcal{O}_mx) = \mathcal{O}_mv(x)$ for all $x \in \RR^2$.
\end{definition}
In \cite{elgindi2020symmetries}, it was shown that given $L^\infty$ and $m$-fold symmetric $\omega$ with $m \ge 3$, one can uniquely recover $u$ in the class of $m$-fold symmetric velocity fields satisfying $|u(x)| \le C|x|$ via the formula
\begin{equation}
\label{eqn:pv-bs-law-hom}
    u(x) = \lim_{R\to\infty} \frac{1}{2\pi}\int_{|y| \le R} \frac{(x-y)^\perp}{|x-y|^2}\omega(y) \ \mathrm{d}y,
\end{equation}
the key observation being that with $m$-fold symmetry, one can replace the Biot-Savart kernel $K(x-y) := \frac{1}{2\pi}\frac{(x-y)^\perp}{|x-y|^2}$ with its $m$-fold symmetrization, $K^{(m)}(x-y) =\frac{1}{m}\sum_{j=1}^m K(x-\mathcal{O}^j_{2\pi/m}y)$ which exhibits integrable decay in $y$ when $m\ge 3$. Thus the Biot-Savart law holds in the principal value sense, and $u$ still satisfies a log-Lipschitz bound. This permits an extension of the Yudovich theory to the class of essentially bounded and $m$-fold symmetric vorticities when $m \ge 3$. Furthermore, the well-posedness theory under symmetry was extended to certain scale-invariant weighted H\"{o}lder spaces $\mathring{C}^{k,\gamma}(\RR^2)$ which encode angular regularity. We will define precisely the spaces $\mathring{C}^{k,\gamma}(\RR^2)$ in \eqref{def:angular-Holder}-\eqref{def:angular-k-Holder} below.   

\subsubsection{Radially homogeneous solutions}
For $x \in \mathbb{R}^2$, recall the associated polar coordinate variables $r = |x|$, $x_1 = r\cos\theta, x_2 = r\sin\theta$. The polar coordinate basis vectors are denoted $ \mathbf{e}_r = (\cos\theta, \sin\theta)$ and $\mathbf{e}_\theta = (-\sin\theta, \cos\theta)$. We also recall the polar coordinate expressions $\nabla = \mathbf{e}_r \partial_r + \mathbf{e}_\theta \frac{1}{r}\partial_\theta$ and $\nabla^\perp = \mathbf{e}_\theta \partial_r - \mathbf{e}_r \frac{1}{r}\partial_\theta$.

The theory in \cite{elgindi2020symmetries} includes the class of radially homogeneous solutions. These are solutions with the property that
\begin{equation}
\label{def:rad-hom-hom}
\omega(t,r,\theta) = \omega(t, 1, \theta), \quad u(t,r,\theta) = ru(t, 1, \theta), \quad \psi(t, r, \theta) = r^2\psi(t, 1, \theta)
\end{equation}
for all times $t$, where $\psi$ denotes the stream function of $u$, i.e. $u = \nabla^\perp \psi$. Note in particular that the vorticity is scale invariant, of the form $\omega(t,x) = h(t, x/|x|)$. Radial homogeneity is propagated by the evolution \eqref{eqn:hom-vort}-\eqref{eqn:hom-euler-bS-form}, and by uniqueness, the equations reduce to the following one-dimensional system on $\mathbb{S}^1$:
\begin{equation}
\label{eqn:1d-hom-sys}
    \begin{rcases}
        \partial_t h(t, \theta) + 2H(t, \theta)\partial_\theta h(t, \theta) = 0,\\
        4H + H'' = h.
    \end{rcases}
\end{equation}
As is well-known, scale invariant solutions are important in the study of the long time behavior of solutions to evolution equations. In \cite{elgindi2022long}, the long time behavior of solutions to \eqref{eqn:1d-hom-sys} possessing $m$-fold symmetry with $m \ge 4$ was described. This is significant because the setting is away from equilibrium.

\subsection{Notational conventions and functional setting}
Before presenting our main results, we describe our functional setting and establish some notational conventions.
Define the norm
\begin{align}
    \|f\|_{\mathring{C}^{0, \gamma}(\RR^2)} &:= \|f\|_{L^\infty(\mathbb{R}^2)} + [|\cdot|^\gamma f]_{C^\gamma(\mathbb{R}^2)} \label{def:angular-Holder}\\
    &:= \|f\|_{L^\infty(\mathbb{R}^2)}+ \sup_{x\ne y} \frac{||x|^\gamma f(x) - |y|^\gamma f(y)|}{|x-y|^\gamma} \nonumber
\end{align}
where $[\ \cdot\  ]_{C^\gamma(\mathbb{R}^2)}$ is the $\gamma-$H\"{o}lder seminorm for $\gamma \in (0,1)$. Note that $\mathring{C}^{0,\gamma}(\mathbb{R}^2)$ scales like $L^\infty(\mathbb{R}^2)$. For $\gamma = 1$, we define
\begin{equation}
\label{def:angular-Lip}
    \|f\|_{\mathring{C}^{0, 1}(\mathbb{R}^2)} := \|f\|_{L^\infty(\mathbb{R}^2)} + \||x|\nabla f\|_{L^\infty(\mathbb{R}^2)}
\end{equation}
We also define higher order spaces for $k \ge 1$ via the norms
\begin{equation}
\label{def:angular-k-Lip}
    \|f\|_{\mathring{C}^{k, 1}(\mathbb{R}^2)} := \|f\|_{L^\infty(\mathbb{R}^2)} + \||x|^{k}\nabla^k f\|_{L^\infty(\mathbb{R}^2)}
\end{equation}
\begin{equation}
\label{def:angular-k-Holder}
    \|f\|_{\mathring{C}^{k, \gamma}(\mathbb{R}^2)} := \|f\|_{\mathring{C}^{k-1, 1}(\mathbb{R}^2)} + [| \cdot |^{\gamma + k} \nabla^{k}f]_{C^\gamma(\mathbb{R}^2)},  \quad \gamma \in (0,1).
\end{equation}
Thus $\mathring{C}^{k,\gamma}(\mathbb{R}^2)$ is the subspace of $L^\infty(\mathbb{R}^2)$ functions with finite $\mathring{C}^{k,\gamma}(\mathbb{R}^2)$ norm. The usual H\"{o}lder spaces will be denoted by $C^{k,\gamma}(\mathbb{R}^2)$; on the circle we use $C^{k,\gamma}(\mathbb{S}^1)$. We will often omit the dependence of the spaces and norms on the domain $\RR^2$ and just write for example $\mathring{C}^{k,\gamma}$  where no confusion arises. We remark that the spaces $\mathring{C}^{k,\gamma}(\mathbb{R}^2)$ encode regularity in the angular variable $\theta$; in fact, $g(\theta) \in C^{k,\gamma}(\mathbb{S}^1)$ if and only if $g \in \mathring{C}^{k,\gamma}(\mathbb{R}^2)$ (see \Cref{app:a} for useful properties of $\mathring{C}^{k,\gamma}(\mathbb{R}^2)$). Additionally, $\mathring{C}^{k,\gamma}(\mathbb{R}^2)$ functions are ${C}^{k,\gamma}$ away from the origin.

For a function space $X$, we will use $X_m$ to denote the subspace of $m$-fold symmetric functions in $X$. In addition $C([0, T]; X)$ denotes the Banach space of $X$-valued continuous functions on $[0,T]$, where $X$ is a Banach space.  For a function on the circle, being $m$-fold symmetric is the same as being periodic with period $\frac{2\pi}{m}$. Thus we will often identify an $m$-fold symmetric function on $\mathbb{S}^1$ with its restriction to $[-\frac{\pi}{m}, \frac{\pi}{m}]$. 
    The letter $C$ will often be used to denote a constant whose value can change from instance to instance. We will also occasionally write $C = C(p_1, p_2,\dots, p_\ell)$ when $C$ depends on parameters $p_1, p_2,\dots, p_\ell$ and where inessential, some dependencies will be suppressed. We may also write $A \lesssim B$ to mean $A \le CB$ for some $C$, and $A \approx B$ to denote $A \lesssim B$ and $B \lesssim A$.
\subsection{Problem formulation and main results}
The first aim of the present article is to extend the well-posedness theory in $\mathring{C}^{k,\gamma}(\RR^2)$ to the inhomogeneous Euler equations under $m$-fold symmetry, at least locally in time. As in the homogeneous setting above, this theory will cover a special class of radially homogeneous solutions for the inhomogeneous Euler equations, for which the dynamics reduce to a new one-dimensional system \eqref{eqn:1dsys} on $\mathbb{S}^1$. Our second aim is to study this one-dimensional system; we will show that the corresponding Cauchy problem is globally well-posed in time for a certain class of initial data. While the literature contains continuation criteria for various local theories, we are unaware of any non-trivial \textit{large data} setting before our study for which global existence is shown for the inhomogeneous Euler equations. Hence our results also offer a new nonperturbative setting in which the long-time dynamics of the inhomogeneous Euler equations can be studied. We make some observations in that direction in \Cref{sec:long-time} below in a special case within this setting.

We first reformulate the inhomogeneous problem as follows. If we denote $\Omega := \nabla^\perp \cdot (\rho u)$, apply formally $\nabla^\perp\cdot$ to the second equation in \eqref{eqn:inhomogeneous-2deuler} and use $\nabla \cdot u = 0$, we obtain after a short calculation
\begin{equation}
\label{eqn:vorticityform}
    \partial_t \Omega + u\cdot \nabla \Omega = \frac{1}{2}\nabla|u|^2 \cdot \nabla^\perp \rho.
\end{equation}
Since $\nabla \cdot u = 0$, we may introduce the stream function $\psi$ of $u$, so that $u = \nabla^\perp \psi$ and hence
\begin{equation*}
    \Omega = \nabla^\perp\cdot(\rho\nabla^\perp \psi) = \nabla \cdot (\rho \nabla \psi).
\end{equation*}
Thus, if we can recover $u = \nabla^\perp \psi$ uniquely from $\Omega = \nabla \cdot (\rho \nabla \psi)$, then we can reformulate the Cauchy problem as follows: given (sufficiently regular) $u_0$ divergence-free and $\rho_0$ with $0 < c_0 \le \rho_0 \le c_1$, define $\Omega_0 = \nabla^\perp\cdot(\rho_0u_0)$ and seek $(\rho, \Omega)$ satisfying $\rho(0,x) = \rho_0(x)$ and $\Omega(0, x) = \Omega_0(x)$ and
\begin{equation}
    \label{eqn:inho-2deuler:vort}
    \begin{rcases}
        \partial_t \rho + u \cdot \nabla \rho  = 0 \\
        \partial_t \Omega + u\cdot \nabla \Omega = \frac{1}{2}\nabla|u|^2 \cdot \nabla^\perp \rho\\
    \end{rcases}
\end{equation}
for $t > 0$, where for each $t$, $u = \nabla^\perp \psi$, and $\psi$ is the unique solution of the elliptic equation
\begin{equation}
\label{eqn:elliptic-bslaw}
    \nabla\cdot (\rho \nabla \psi) = \Omega.
\end{equation}
Thereafter, the pressure $p$ can be recovered by solving the same kind of equation as \eqref{eqn:elliptic-bslaw}, namely
\begin{equation}
\label{eqn:pressure}
    -\nabla \cdot(\rho^{-1}\nabla p) = \nabla \cdot ((u \cdot\nabla) u),
\end{equation}
obtained formally by taking the divergence of the second equation in \eqref{eqn:inhomogeneous-2deuler} in the simplified form $\partial_t u + u\cdot\nabla u  = -\frac{1}{\rho}\nabla p$.

\begin{remark}
    One could have obtained a formulation of the inhomogeneous Euler equations in terms of the actual vorticity $\omega = \nabla^\perp\cdot u$ which satisfies
    \begin{equation*}
        \partial_t \omega + u\cdot \nabla \omega = - \frac{\nabla^\perp \rho}{\rho} \cdot \nabla p
    \end{equation*}
    Thus to obtain estimates on $\omega$, one must estimate the pressure $p$ and this still amounts to studying \eqref{eqn:pressure}. In this paper, we have elected to study the system \eqref{eqn:inho-2deuler:vort} on $\Omega$ and $\rho$ above. The unknown $\Omega$ has been utilized previously by Fanelli \cite{fanelli2025geometric} for establishing blow-up criteria. 
\end{remark}
Our first major task will be to show that in our functional setting, the equation \eqref{eqn:elliptic-bslaw} defines $\psi$ (and hence $u$) uniquely. We remark that this is already non-trivial even when $\rho$ is constant (i.e. the Poisson problem), since the source term is not required to decay at infinity. In \cite{elgindi2020symmetries}, the following theorem was shown.
\begin{theorem}
[{\cite[Theorem~B]{elgindi2020symmetries}}]
\label{thm:poisson}
    Let $\gamma \in (0,1)$. Assume that $f \in \mathring{C}_m^{0,\gamma}(\mathbb{R}^2)$ with $m \ge 3$. Then there exists a unique solution of the equation
    \begin{equation*}
        \Delta \psi = f
    \end{equation*}
    satisfying $|\psi(x)| \le C\|f\|_{L^\infty}|x|^2$,  $|\nabla \psi(x)| \le C\|f\|_{L^\infty}|x|$ and $m$-fold symmetric. Furthermore, $\nabla^2  \psi \in \mathring{C}^{0,\gamma}(\mathbb{R}^2)$ and 
    \begin{equation}
    \label{ineq:poisson:circ-schaud}
        \|\nabla^2\psi\|_{\mathring{C}^{0,\gamma}} \le C_\gamma \|f\|_{\mathring{C}^{0,\gamma}}
    \end{equation}
\end{theorem}
The condition $m \ge 3$ is sharp. The existence part of \Cref{thm:poisson} as shown in \cite{elgindi2020symmetries} relies on the explicit Biot-Savart kernel $K$ and formula \eqref{eqn:pv-bs-law-hom}. For uniqueness, the $m$-fold symmetry with $m > 2$ ensures that $f$ is ``orthogonal" to degree two harmonic polynomials $\{x_1x_2, x_1^2-x_2^2\}$ in the kernel of $\Delta$. The estimate \eqref{ineq:poisson:circ-schaud} is a consequence of the $\mathring{C}^{0,\gamma}$-boundedness of singular integral operators given by a $-2$ degree homogeneous kernel which is smooth on $\mathbb{R}^2\setminus \{0\}$, has zero mean on circles, and whose $m$-fold symmetrization has integrable decay when $m \ge 3$. 

For our inhomogeneous setting, an explicit kernel is unavailable, and hence the method in \cite{elgindi2020symmetries} does not apply. We will instead prove uniqueness by appealing to the theory of quasi-conformal mappings of the complex plane available to us in the two dimensional setting. The main fact we exploit is that given an $m$-fold symmetric $\rho$, a solution $\psi$ of $\nabla \cdot(\rho \nabla \psi) = 0$ with $|\psi(x)| \le C|x|^2$ is topologically equivalent (see \Cref{thm:stoilow} in \Cref{sec:elliptic}) to a harmonic function $\varphi$, $m$-fold symmetric with $|\varphi(x)| \le C|x|^{2\alpha(\rho)}$, where  $\alpha(\rho)$ measures the contrast of $\rho$ and is defined by 
\begin{equation*}
    \alpha(\rho) := \frac{\sup \rho}{ \inf \rho}.
\end{equation*}
By a Liouville theorem, $\varphi$ is an $m$-fold symmetric polynomial of degree not exceeding $2\alpha(\rho)$. Since $m  > 2\alpha(\rho)$, $\varphi$ must be a constant and thus $\varphi\equiv 0$ and consequently $\psi \equiv 0$. Note that $m  > 2\alpha(\rho)$ specialized to the homogeneous case is exactly the optimal $m >2$. For existence, we shall develop Schauder \textit{a priori} estimates in the $\mathring{C}^{0,\gamma}$ framework. Then we will use \Cref{thm:poisson}, the uniqueness statement, and a compactness--contradiction approach in the spirit of \cite{simon1997schauder} to deduce the full \textit{a priori} estimates required to invoke the method of continuity. Thus, we shall generalize the result of \eqref{thm:poisson} to the equation \eqref{eqn:elliptic-bslaw} in the following theorem.
\begin{theorem}
\label{thm:elliptic-existence}
     Let $\gamma \in (0,1)$. Suppose that $0 < \inf \rho\le \rho \in \mathring{C}_m^{1,\gamma}(\mathbb{R}^2)$ with $\|\rho\|_{\mathring{C}^{1,\gamma}} \le M$ and $f \in \mathring{C}_m^{0,\gamma}(\mathbb{R}^2)$ with $m > 2 \alpha(\rho)$ and $M > 0$.  Then there exists a unique solution of the equation
    \begin{equation*}
        \nabla \cdot(\rho\nabla \psi) = f
    \end{equation*}
    satisfying $|\psi(x)|\le C\|f\|_{\mathring{C}^{0,\gamma}}|x|^2$, $|\nabla\psi(x)|\le C\|f\|_{\mathring{C}^{0,\gamma}}|x|$ and $m$-fold symmetric. Furthermore,
    $\nabla^2  \psi \in \mathring{C}^{0,\gamma}(\mathbb{R}^2)$ and 
    \begin{equation*}
        \|\nabla^2\psi\|_{\mathring{C}^{0,\gamma}} \le C \|f\|_{\mathring{C}^{0,\gamma}},
    \end{equation*}
    where $C = C (\gamma,m,\inf \rho, \sup\rho, M)$.
\end{theorem}

With \Cref{thm:elliptic-existence} at our disposal, we shall prove the following local well-posedness theorem. 
\begin{theorem}
\label{thm:lwp}
    Let $\Omega_0\in\mathring{C}_m^{1,\gamma}(\mathbb{R}^2)$ and let $0<\inf\rho_0\le\rho_0\in\mathring{C}_m^{2,\gamma}(\mathbb{R}^2)$, with $\gamma\in(0,1)$ and $m>2\alpha(\rho_0)$. Then there exists
    \[
    T=T\!\left(\gamma,m,\inf\rho_0,\sup\rho_0,\|\rho_0\|_{\mathring{C}^{2,\gamma}},\|\Omega_0\|_{\mathring{C}^{1,\gamma}}\right)>0
    \]
    and a unique solution of \eqref{eqn:inho-2deuler:vort} satisfying, for every $0<\beta<\gamma$,
    \[
    \Omega\in L^\infty(0,T;\mathring{C}_m^{1,\gamma})\cap C([0,T];\mathring{C}_m^{1,\beta}),
    \qquad
    \rho\in L^\infty(0,T;\mathring{C}_m^{2,\gamma})\cap C([0,T];\mathring{C}_m^{2,\beta}),
    \]
    and $(\Omega(0,x),\rho(0,x))=(\Omega_0(x),\rho_0(x))$.
\end{theorem}

\begin{remark}
    The regularity of $\Omega_0$ and $\rho_0$ displayed in \Cref{thm:lwp} above is likely not optimal, but we do not pursue this matter here.
\end{remark}

Assuming \Cref{thm:lwp}, we will deduce the following corollary (see \Cref{sec:rad-hom:inhom} below). 
\begin{corollary}
\label{cor:rad-hom-pres}
    If $\Omega_0$ and $\rho_0$ satisfy the hypotheses of \Cref{thm:lwp} and in addition are 0-homogeneous, i.e. $\Omega_0(\lambda x) = \Omega_0(x), \rho_0(\lambda x) = \rho_0(x)$ for all $\lambda$, then the corresponding solution $\Omega(t,x), \rho(t,x)$ is 0-homogeneous for all $t \in [0, T]$.
\end{corollary}

The foregoing corollary allows us to reduce \eqref{eqn:inhomogeneous-2deuler} to a one-dimensional system \eqref{eqn:1dsys} which we then use to produce global in time solutions for \eqref{eqn:inhomogeneous-2deuler}. Unlike other global-well-posedness results in the literature, we do not assume that the data is small. We state this result informally, with the precise details to follow momentarily (\Cref{thm:gwp}). 
\begin{theorem}[Global existence, informal statement]
    There is a class of 0-homogeneous data so that the maximal existence time of \Cref{thm:lwp} is $T_* = +\infty$. 
\end{theorem}

\subsubsection{Radially homogeneous solutions for inhomogeneous Euler}
\label{sec:rad-hom:inhom}
The local well-posedness statement of \Cref{thm:lwp} covers a distinguished class of radially homogeneous solutions. These solutions have $0$-homogeneous vorticity and density and hence satisfy for each $t \ge 0$,
\begin{equation}
\label{eqn:ansatz}
    \rho = P(t,\theta), \quad \Omega = g(t,\theta), \quad \psi = r^2G(t,\theta),
\end{equation}
for some functions $P, g$ and $G$, where $\psi$ is the stream function. Thus \eqref{def:rad-hom-hom} is the special case where $\rho$ is constant.

If a solution satisfies \eqref{eqn:ansatz}, then
\begin{equation*}
    u = \nabla^\perp (r^2 G) = r(2G \mathbf{e}_\theta - \partial_\theta G \mathbf{e}_r),
\end{equation*}
\begin{equation*}
    \frac{1}{2}\nabla|u|^2 \cdot \nabla^\perp \rho = -\partial_\theta P(4G^2 + (\partial_\theta G)^2)
\end{equation*}
The first equation in \eqref{eqn:inhomogeneous-2deuler} becomes
\begin{equation}
    \label{eqn:P}
            \partial_t P + 2G\partial_\theta P = 0
\end{equation}
and \eqref{eqn:vorticityform} becomes
\begin{equation}
\label{eqn:g}
        \partial_t g + 2G\partial_\theta g   = -(4G^2 + (\partial_\theta G)^2)\partial_\theta P.
\end{equation}
In addition, from \eqref{eqn:elliptic-bslaw} we have
\begin{align*}
    g =\Omega &= \nabla\cdot(P \nabla (r^2G)) \\
    &=  \nabla\cdot(P(2x G+ x^\perp \partial_\theta G)) \\
    &=  P (\partial_{\theta\theta} G + 4G) + \partial_\theta G\partial_\theta P\\
    &= \partial_\theta(P\partial_\theta G) + 4PG
\end{align*}
The equation
\begin{equation}
    \label{eqn:BSlaw}
    g = \partial_\theta(P\partial_\theta G) + 4PG
\end{equation}
allows us to recover $G$ from $g$; hence, it is the analog of the Biot-Savart law in the current setting. Putting together \eqref{eqn:P}, \eqref{eqn:g} and \eqref{eqn:BSlaw}, we obtain the following one dimensional system on $\mathbb{S}^1 = \{ \theta: -\pi \le \theta < \pi \}$
\begin{equation}
\label{eqn:1dsys}
    \begin{rcases}
        \partial_t g + 2G\partial_\theta g   = -(4G^2 + (\partial_\theta G)^2)\partial_\theta P \\
        \partial_t P + 2G\partial_\theta P = 0 \\
        \partial_\theta(P\partial_\theta G) + 4PG = g
    \end{rcases}
\end{equation}
for $t \ge 0$. 

Conversely, a solution $(g,P,G)$ of \eqref{eqn:1dsys} gives a solution of \eqref{eqn:inho-2deuler:vort} via \eqref{eqn:ansatz}. Indeed, given a solution with $g\in L^\infty(0,T;C_m^{1,\gamma})\cap C([0,T];C_m^{1,\beta})$ and $P\in L^\infty(0,T;C_m^{2,\gamma})\cap C([0,T];C_m^{2,\beta})$ for every $0<\beta<\gamma$, and with $m>2\alpha(P)$, we may set $\Omega_0(x)=g(0,\theta)$ and $\rho_0(x)=P(0,\theta)$ in \Cref{thm:lwp} to obtain a local solution of \eqref{eqn:inho-2deuler:vort}. By uniqueness, the obtained solution must be radially homogeneous. Thus radial homogeneity is propagated by the system \eqref{eqn:inho-2deuler:vort}, as stated in \Cref{cor:rad-hom-pres}. 

The system \eqref{eqn:1d-hom-sys} is the special case when $P$ is identically constant. The non-linear non-local term $-(4G^2 + (\partial_\theta G)^2)\partial_\theta P$ which is absent in the homogeneous case is the main source of difficulty in the analysis of \eqref{eqn:1dsys}. In fact, for a smooth solution of \eqref{eqn:1dsys}, one can derive formally that
\begin{equation*}
     \partial_t(\partial_\theta G) + 2G\partial_\theta(\partial_\theta G) = (\partial_\theta G)^2 + \mathrm{term},
\end{equation*}
where ``term'' is specified in \eqref{eqn:ricatti-structure} below.
However, our next theorem asserts that the Cauchy problem for \eqref{eqn:1dsys} is globally well-posed in time  under certain symmetry and structural assumptions on initial data $(g_0, P_0) \in C^{1,\gamma}_m(\mathbb{S}^1) \times C_m^{2,\gamma}(\mathbb{S}^1)$. Hence, for appropriate $m$, the maximal existence time given by \Cref{thm:lwp} for such radially homogeneous data is $+\infty$. The precise assumptions are as follows.

\begin{assumptions}[initial data class]
\label{assumps}
      Suppose $g_0: \mathbb{S}^1 \to \RR$, $P_0: \mathbb{S}^1 \to \RR_+$ with $P_0\in C^{2,\gamma}(\mathbb{S}^1), g_0 \in C^{1,\gamma}(\mathbb{S}^1)$ and $\inf_{\mathbb{S}^1} P_0 > 0$. Denote
    \begin{equation*}
    \alpha =\alpha(P_0) := \frac{\sup_{\mathbb{S}^1} P_0}{\inf_{\mathbb{S}^1} P_0}
    \end{equation*}
and assume the following. 
\begin{enumerate}[label=\normalfont H\arabic*]
    \item \label{H1}The data $g_0$ and $P_0$ are $m$-fold symmetric with $m > 4\sqrt{\alpha}$ and furthermore $g_0(\theta)$ has odd symmetry about $0$, while $P_0(\theta)$ has even symmetry about $0$. 

   \item \label{H2}Either  $P_0'(\theta) \le 0$ on $[0, \frac{\pi}{m}]$, or $P_0'(\theta) \ge 0$ on $[0, \frac{\pi}{m}]$.
    \item \label{H3} $\mathrm{supp}(P_0') \subset [-\frac{\pi}{m}+\ve, \frac{\pi}{m} -\ve]$ for some $\ve > 0$.

\end{enumerate}

\end{assumptions}

\begin{remark}
In the sequel, we shall always assume that the first alternative in the hypothesis \eqref{H2} holds. The ensuing proofs are the same with minor modifications in case $P_0'(\theta) \ge 0$ holds instead. 
\end{remark}

\begin{figure}[H]
    \centering
    \includegraphics[width=0.32\textwidth]{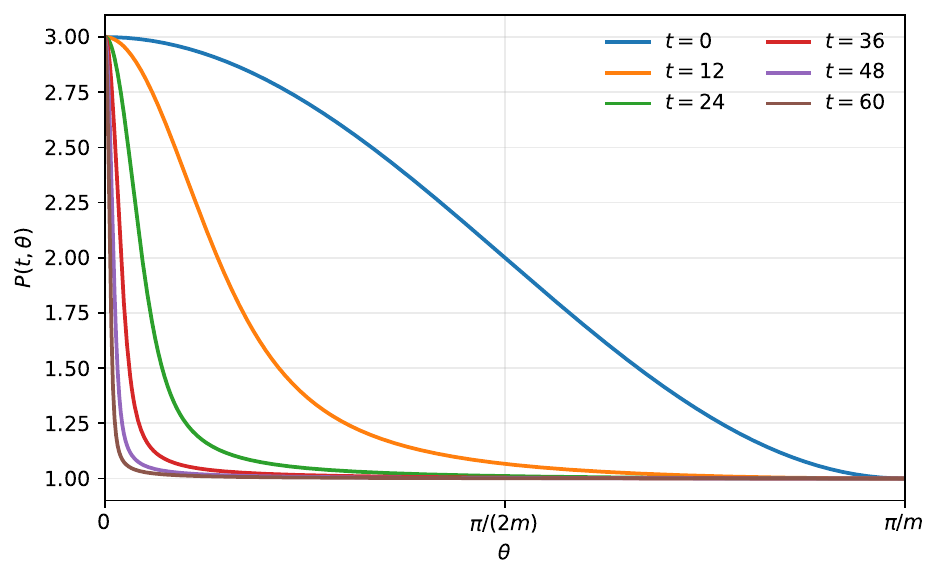}\hfill
    \includegraphics[width=0.32\textwidth]{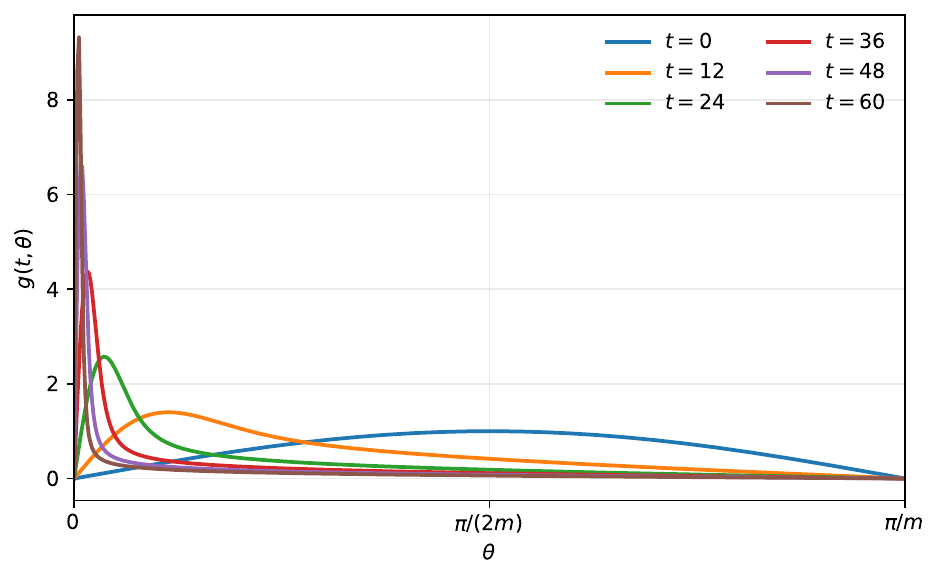}\hfill
    \includegraphics[width=0.32\textwidth]{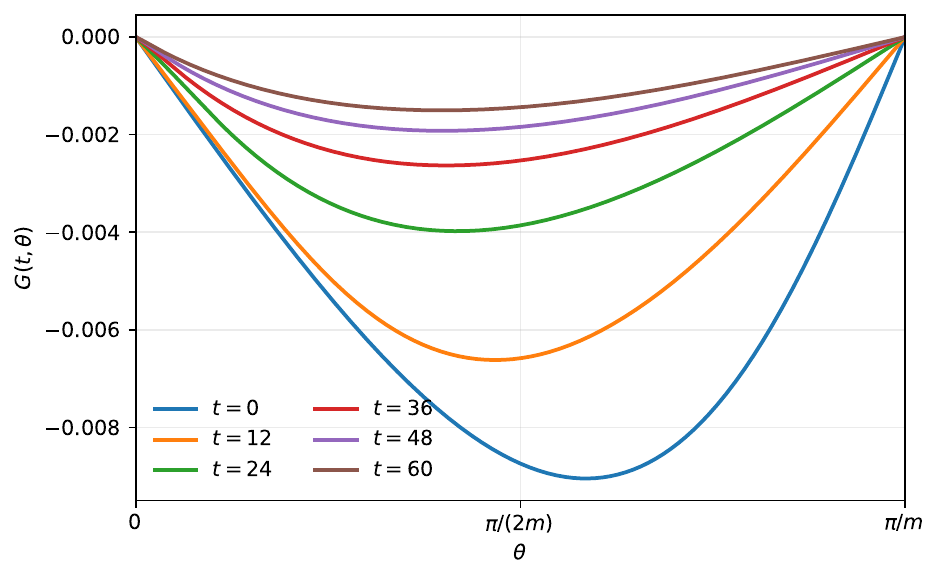}
    \caption{Numerical evolution of \eqref{eqn:1dsys} for $m=8$ with $P_0(\theta)=2+\cos(8\theta)$ and $g_0(\theta)=\sin(8\theta)$, which satisfy \eqref{H1}--\eqref{H2} but not \eqref{H3}. From left to right are $P$, $g$, and $G$. The computation shows concentration and growth in $g$ near $\theta=0$ while $P$ remains positive. We conjecture that global well-posedness holds without \eqref{H3}.}
    \label{fig:1d-nonH3-numerics}
\end{figure}

\begin{figure}[H]
    \centering
    \includegraphics[width=0.32\textwidth]{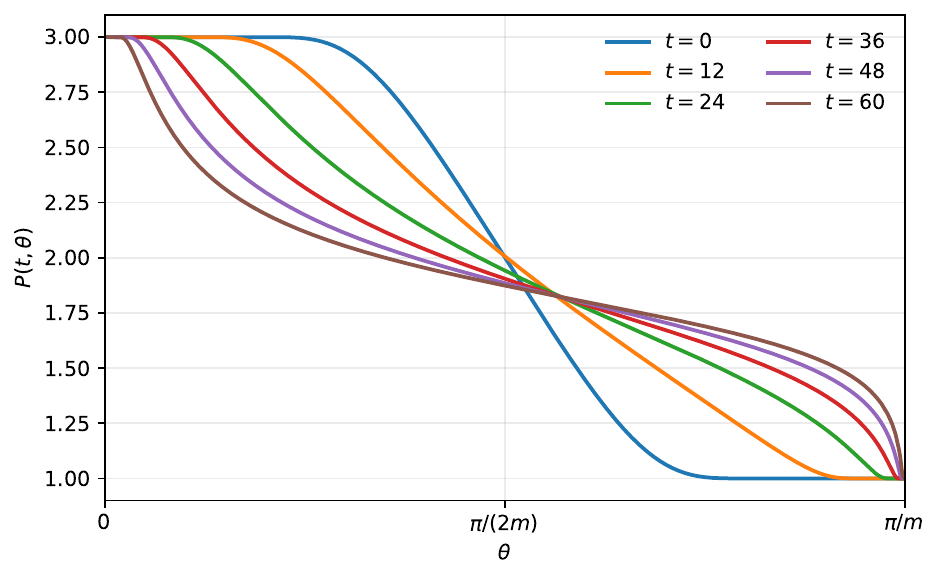}\hfill
    \includegraphics[width=0.32\textwidth]{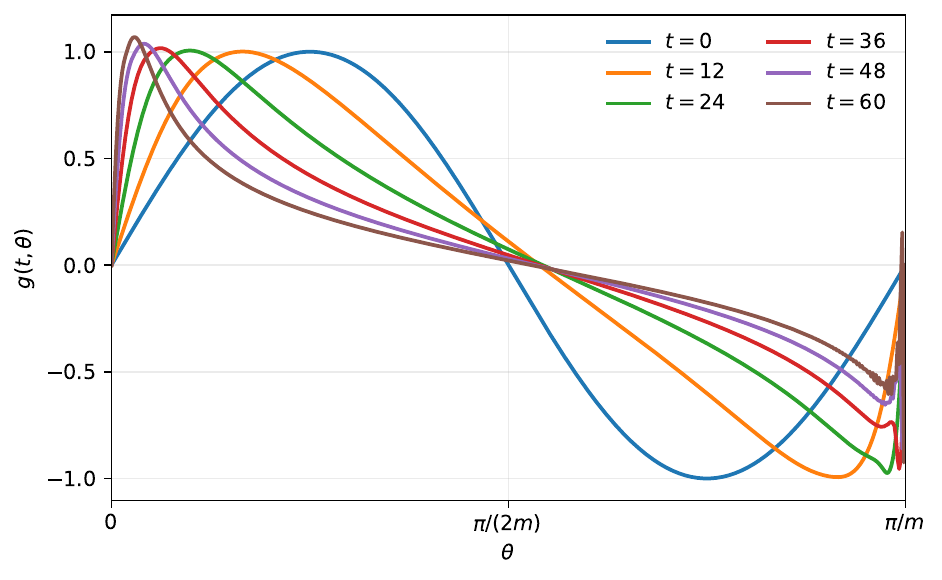}\hfill
    \includegraphics[width=0.32\textwidth]{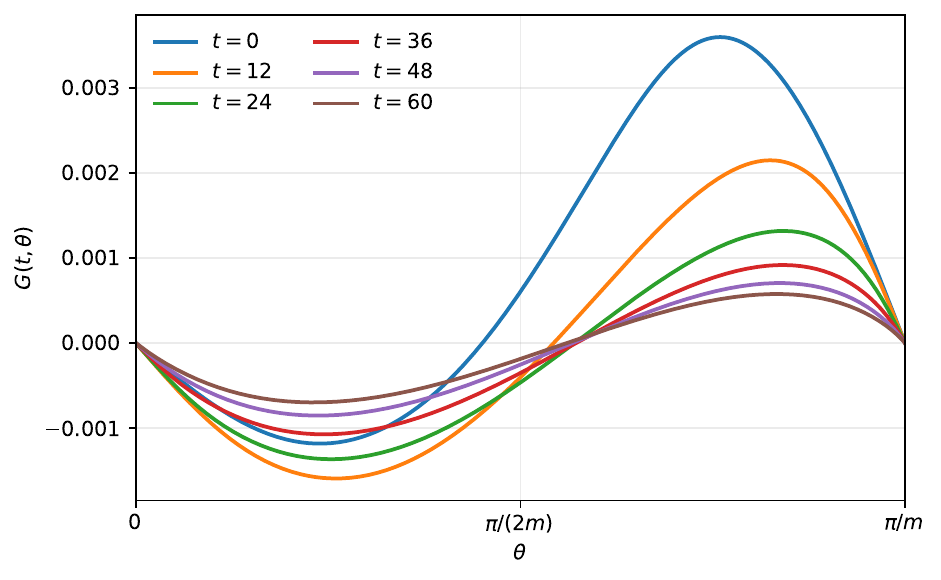}
    \caption{Numerical evolution of \eqref{eqn:1dsys} for admissible data satisfying \eqref{H1}--\eqref{H3}, with $P_0$ constant near the sector endpoints and sign-changing $g_0(\theta)=\sin(2m\theta)$. From left to right are $P$, $g$, and $G$. The density stays positive while $g$ becomes increasingly localized near the sector boundaries.}
    \label{fig:1d-H3-nonsigned-numerics}
\end{figure}
\begin{remark}
 Note that $\alpha(P(t,\cdot)) = \alpha(P_0)$ for smooth solutions of \eqref{eqn:1dsys}, since $P$ is transported. Due to the uniqueness part of \Cref{thm:lwp}, a smooth solution of \eqref{eqn:1dsys} attaining given data $(g_0, P_0)$ is unique. Now by uniqueness, it is easy to check that the assumption \eqref{H1} is propagated by \eqref{eqn:1dsys}. The same is true of \eqref{H2}: since $P$ is transported, $\partial_\theta P(t, \theta)$ preserves its sign ($\le 0$ on $[0, \frac{\pi}{m}]$) and thus $g$ is a super-solution of a homogeneous transport equation, so that the estimate 
 \begin{equation}
 \label{ineq:g:inf}
     g(t,\theta) \ge - \|g_0\|_{L^\infty}
 \end{equation}
on $[0, \frac{\pi}{m}]$
holds by the comparison principle.

\end{remark}

We hereafter refer to data $(g_0, P_0)$ satisfying \Cref{assumps} as \textit{admissible}. Our result is the following.

\begin{theorem}
\label{thm:gwp-1dsys}
    Let $(g_0,P_0)$ be admissible. Then the Cauchy problem for \eqref{eqn:1dsys} has a unique global solution such that, for every $T>0$ and every $0<\beta<\gamma$,
    \[
    g\in L^\infty(0,T;C_m^{1,\gamma}(\mathbb{S}^1))\cap C([0,T];C_m^{1,\beta}(\mathbb{S}^1)),
    \qquad
    P\in L^\infty(0,T;C_m^{2,\gamma}(\mathbb{S}^1))\cap C([0,T];C_m^{2,\beta}(\mathbb{S}^1)).
    \]
\end{theorem}

In view of \Cref{thm:lwp} and \Cref{thm:gwp-1dsys}, we have the following.  
\begin{theorem}[Global existence, formal statement]
\label{thm:gwp}
     Let $\Omega_0\in\mathring{C}_m^{1,\gamma}(\mathbb{R}^2)$ and let $0<\inf\rho_0\le\rho_0\in\mathring{C}_m^{2,\gamma}(\mathbb{R}^2)$, with $\gamma\in(0,1)$ and $m>\max\{2\alpha(\rho_0),4\sqrt{\alpha(\rho_0)}\}$. If $\Omega_0=g_0(\theta)$ and $\rho_0=P_0(\theta)$ for some admissible pair $(g_0,P_0)$, then there exists a unique global solution of \eqref{eqn:inho-2deuler:vort} such that, for every $T>0$ and every $0<\beta<\gamma$,
     \[
     \Omega\in L^\infty(0,T;\mathring{C}_m^{1,\gamma})\cap C([0,T];\mathring{C}_m^{1,\beta}),
     \qquad
     \rho\in L^\infty(0,T;\mathring{C}_m^{2,\gamma})\cap C([0,T];\mathring{C}_m^{2,\beta}),
     \]
     with $(\Omega(0,x),\rho(0,x))=(\Omega_0(x),\rho_0(x))$.
\end{theorem}

\begin{remark}[Role of hypothesis \eqref{H3}]
    The hypothesis \eqref{H3} may not be necessary for global well-posedness, but the non-local non-linear term $-\partial_\theta P(4G^2 + (\partial_\theta G)^2)$ in \eqref{eqn:1dsys} poses difficulties in the rigorous analysis. Our proof below works because the odd symmetry of $G$ traps characteristics between $[0, \frac{\pi}{m}]$, and this observation can be used to verify the continuation criterion we shall develop in \Cref{sec:continuation-crit} provided that we have control near the boundaries, which is what \eqref{H3} ensures. 
\end{remark}

\subsection{Outline of the paper}
The rest of this paper is organized as follows. We begin in \Cref{sec:elliptic} with the existence and uniqueness theory for the elliptic equation \eqref{eqn:elliptic-bslaw} in the $\mathring{C}^{k,\gamma}(\RR^2)$ framework. The relevant definitions and theorems from the theory of quasi-conformal mappings are presented in a brief overview before proving \Cref{thm:elliptic-existence}. In \Cref{sec:lwp}, we outline some estimates for linear transport equations in $\mathring{C}^{k,\gamma}$ and use these estimates, the results of \Cref{sec:elliptic} and an iteration scheme to prove \Cref{thm:lwp}. \Cref{sec:rad-hom} is devoted to radially homogeneous solutions of the 2D inhomogeneous Euler equations. It begins in \Cref{sec:bvp} with the study of the boundary value problem (BVP) for $\partial_\theta(P\partial_\theta G) + 4PG = g$, i.e. the reduction of \eqref{eqn:elliptic-bslaw} under radial homogeneity. Green's functions for the BVP are constructed with attendant estimates in \Cref{sec:greensbounds}. This is followed by the local well-posedness theory for the one dimensional system \eqref{eqn:1dsys}. We also present a continuation criterion for the one dimensional system in \Cref{sec:continuation-crit}. We prove \Cref{thm:gwp-1dsys} in \Cref{sec:gwppf}. Lastly, we record in \Cref{sec:long-time} an observation about the long-time behavior of solutions in the special case where $g$ is signed.

\section{Existence and uniqueness for the elliptic problem}
\label{sec:elliptic}
We turn first to the task of recovering $\psi$ (and hence $u = \nabla^\perp \psi$) from $\Omega$ via the elliptic equation
\begin{equation*}
    \nabla\cdot (\rho \nabla \psi) = \Omega.
\end{equation*}
The equation $\nabla\cdot (\rho \nabla \psi) = \Omega$ should be viewed as the ``inhomogeneous Biot-Savart law" in the present context. Unlike the homogeneous case, one does not have an explicit kernel and furthermore the operator defined by $\mathcal{L}_\rho[\psi] := \nabla \cdot(\rho \nabla \psi)$ is evolving with $\rho$ in the broader context of \eqref{eqn:inho-2deuler:vort}. We address these problems in this section with the aid of tools from complex analysis and elliptic theory. 
\subsection{Uniqueness}
The goal here is to show that given an $m$-fold symmetric density $\rho$, the homogeneous problem 
\begin{equation}
\label{eqn:elliptic-bslaw-hom}
    \nabla\cdot (\rho \nabla \psi) = 0
\end{equation}
has no non-trivial solutions in the $m$-fold symmetric setting, for appropriate $m > 2\alpha(\rho)$. When $\rho$ is a constant (i.e. in the case of the Laplacian), this was shown in \cite{elgindi2020symmetries}. We shall assume in this section that $0 < \inf \rho < \sup\rho = 1$, without loss of generality. The following uniqueness statement takes advantage of the two dimensional setting of the problem by invoking the theory of quasi-conformal mappings of the complex plane. We begin with some relevant definitions and theorems, for which our chief reference is \cite{astala2008elliptic}. Denote $z = x_1 + ix_2$. We will identify a function $g$ on $\mathbb{R}^2$ with a function on $\mathbb{C}$ by writing $g(z)$ to mean $g(x_1, x_2)$ for $z = x_1 + i x_2$. For a complex function $f: \mathbb{C} \to \mathbb{C}$, we define the directional (distributional) derivatives
\begin{equation}
\label{def:complex-partials}
    \partial_\vartheta f(z) := \cos(\vartheta)\partial_{x_1}f(z) + \sin(\vartheta)\partial_{x_2} f(z), \quad \vartheta \in [0, 2\pi].
\end{equation}

\begin{definition}
    Let $1 \le K < \infty$ and let $\Omega, \Omega' \subset \mathbb{C}$. An orientation preserving homeomorphism $f: \Omega \to \Omega'$ is said to be $K$-quasiconformal if $f \in H^{1}_{\mathrm{loc}}(\Omega)$ and the directional derivatives satisfy
    \begin{equation}
    \label{ineq:distortion}
        \max_{\vartheta} |\partial_\vartheta f(z)| \le K\min_{\vartheta}|\partial_\vartheta f(z)|
    \end{equation}
    for a.e. $z \in \Omega$.
\end{definition}

\begin{definition}
    Let $1 \le K < \infty$ and let $\Omega, \Omega' \subset \mathbb{C}$. An orientation preserving map $f: \Omega \to \Omega'$ is said to be $K$-quasiregular if $f \in H^{1}_{\mathrm{loc}}(\Omega)$ and the directional derivatives satisfy
    \begin{equation*}
        \max_{\vartheta} |\partial_\vartheta f(z)| \le K\min_{\vartheta}|\partial_\vartheta f(z)|
    \end{equation*}
    for a.e. $z \in \Omega$.
\end{definition}

Thus quasiconformal mappings are homeomorphic quasiregular mappings. The condition \eqref{ineq:distortion} says that quasi-conformal mappings may distort the plane, but only boundedly so. If $K = 1$, then in fact $f$ is a conformal mapping. The connection to our problem is that one can relate quasiconformal mappings to solutions of  \eqref{eqn:elliptic-bslaw-hom} in the same manner that one can relate a harmonic function of the plane with a holomorphic function. Precisely, if $\nabla \cdot (\rho \nabla \psi) = 0$ for $\psi \in H^{1}_{\mathrm{loc}}$, then there exists a stream function $\varphi \in H^1_{\mathrm{loc}}$ such that
    \begin{equation*}
        \varphi_{x_1} = \rho\psi_{x_2}, \quad \varphi_{x_2} = -\rho\psi_{x_1}.
    \end{equation*}
    Define $f:\mathbb{C} \to \mathbb{C}$ by
    \begin{equation}
        \label{def:f}
        f(z) = \psi - i\varphi
    \end{equation}
Now we recall the Wirtinger derivatives
    \begin{equation}
    \label{def:wirtinger-der}
        \partial_{\bar{z}} f := \frac{1}{2}(\partial_{x_1} + i\partial_{x_2} ) f, \qquad \partial_{z} f := \frac{1}{2}(\partial_{x_1} - i\partial_{x_2} ) f
    \end{equation}
    By a direct calculation, we find
    \begin{equation*}
       \partial_{\bar{z}} f = \frac{1-\rho}{2}(\psi_{x_1} + i\psi_{x_2}) 
    \end{equation*}
    and 
        \begin{equation*}
       \partial_{z} f = \frac{1+\rho}{2}(\psi_{x_1} - i\psi_{x_2})
    \end{equation*}
    Hence $f$ satisfies the \textit{Beltrami equation}:
    \begin{equation}
    \label{eqn:beltrami}
        \partial_{\bar{z}} f = \mu_f(z)\partial_{z} f,
    \end{equation}
    where
    \begin{equation*}
        \mu_f(z) = \frac{1 - \rho(z)}{1+\rho(z)}e^{-2i\ \mathrm{arg}(\partial_z f)}
    \end{equation*}
Since $\sup \rho =1$, it follows that
    \begin{equation*}
        |\mu_f(z)| \le \frac{1-\inf \rho}{1+\inf \rho} = \frac{\alpha-1}{\alpha+1} =: k'  < 1. 
    \end{equation*}
By a simple computation using \eqref{def:complex-partials}, \eqref{eqn:beltrami} and \eqref{def:wirtinger-der}, we see that
\begin{equation}
\label{eqn:quasi-reg-est}
    \max_{\vartheta} |\partial_\vartheta f(z)| \le \frac{1+k'}{1-k'}\min_{\vartheta}|\partial_\vartheta f(z)| = \alpha\min_{\vartheta}|\partial_\vartheta f(z)|.
\end{equation}
Furthermore, \eqref{eqn:beltrami} implies that the Jacobian $J(z, f) = |\partial_z f|^2 - |\partial_{\bar{z}}f|^2 \ge 0$, so that $f$ is orientation preserving and thus $\alpha$-quasiregular.

We state below the measurable Riemann mapping theorem, which asserts the existence of a homeomorphic solution of \eqref{eqn:beltrami}. In the following, $\hat{\mathbb{C}}$ denotes the Riemann sphere. 
\begin{theorem}
[{\cite[Theorem~5.3.4]{astala2008elliptic}}]
\label{thm:riemann}
    Let $|\mu(z)| \le k < 1$ for a.e. $z \in \mathbb{C}$ and let $K = \frac{1+k}{1-k}$. There exists a unique solution $f: \hat{\mathbb{C}} \to \hat{\mathbb{C}}$ of the Beltrami equation
    \begin{equation*}
        \partial_{\bar{z}} f = \mu(z)\partial_z f \quad \text{for a.e. } z \in \mathbb{C}
    \end{equation*}
which is a $K$-quasiconformal homeomorphism satisfying
\begin{equation*}
    f(0) = 0, \quad f(1) = 1, \quad f(\infty) = \infty.
\end{equation*}
\end{theorem}
The next theorem we state is the St\"{o}ilow factorization theorem, which classifies all solutions of \eqref{eqn:beltrami} in terms of any homeomorphic one (a homeomorphic solution exists by \Cref{thm:riemann}). In short, solutions of \eqref{eqn:beltrami} are topologically equivalent to holomorphic functions. This equivalence will be key for us in extending the uniqueness statement of \cite{elgindi2020symmetries} for the Laplacian to \eqref{eqn:elliptic-bslaw}. 

\begin{theorem}
[{\cite[Theorem~5.5.1]{astala2008elliptic}}]
\label{thm:stoilow}
    Let $f: \mathbb{C} \to \mathbb{C}$ be a homeomorphic solution to the Beltrami equation
        \begin{equation}
        \label{eqn:beltrami-stoilow}
        \partial_{\bar{z}} f = \mu(z)\partial_z f, \quad f \in W^{1,1}_{\mathrm{loc}}(\mathbb{C}),
    \end{equation}
    with $|\mu(z)| \le k < 1$ for a.e. $z \in \mathbb{C}$. Assume that $g$ is any other solution of \eqref{eqn:beltrami-stoilow} on $\mathbb{C}$. Then there exists a holomorphic function $\Phi: \mathbb{C} \to \mathbb{C}$ such that
    \begin{equation*}
        g(z) = \Phi(f(z)), \quad z \in \mathbb{C}.
    \end{equation*}
    Conversely, if $\Phi$ is a holomorphic function on $\mathbb{C}$, then $\Phi \circ f$ is a $H^1_{\mathrm{loc}}$ solution to \eqref{eqn:beltrami-stoilow}. 
\end{theorem}
 
Lastly, we shall also need the following result.
\begin{lemma}
[{\cite[Corollary~3.10.4]{astala2008elliptic}}]
\label{lem:quasiconformal-holder}
If $f: \mathbb{C} \to \mathbb{C}$ is a $K$-quasiconformal mapping and $z_0, z_1, z_2 \in \mathbb{C}$, then
\begin{equation*}
    \frac{|f(z_0) - f(z_1)|}{|f(z_0) - f(z_2)|} \le \eta\left(\frac{|z_0-z_1|}{|z_0-z_2|}\right)
\end{equation*}
where $\eta(t) = C_K\max\{t^K, t^{1/K}\}\quad t \in [0, \infty)$, for a constant $C_K > 0$.
\end{lemma}
Observe that if $f$ is $K$-quasiconformal and $f(0) = 0, f(1) = 1$, then by \Cref{lem:quasiconformal-holder} for $z_0 = 0, z_1  =z, z_2=1$ and $z_0 = 0, z_1 = 1, z_2 = z$, we obtain
\begin{equation*}
   \frac{1}{\eta(\frac{1}{|z|})}\le |f(z)| \le \eta(|z|) 
\end{equation*}
or
\begin{equation}
\label{ineq:quasi-holder-est}
 \frac{1}{C_K} \min\{|z|^K, |z|^{\frac{1}{K}}\} \le |f(z)| \le C_K\max\{|z|^K, |z|^{\frac{1}{K}}\}
\end{equation}

For the proofs of the foregoing theorems and a detailed treatment of the theory and applications of quasiconformal mappings of the plane, we refer the reader to \cite{astala2008elliptic}. We are now ready to prove the uniqueness statement.

\begin{proposition}
\label{prop:uniqness-elliptic}
    Let $\alpha := \frac{\sup \rho}{\inf \rho} = \frac{1}{\inf \rho}$. Assume that $\psi:\mathbb{R}^2 \to \mathbb{R}$, $\psi \in {W}_{\mathrm{loc}}^{1, \infty}(\mathbb{R}^2)$ is $m$-fold symmetric with $m > 2\alpha$ and satisfies
    \begin{itemize}
        \item $\nabla\cdot (\rho \nabla \psi) = 0$ in the sense of distributions
        \item $|\psi(x)| \le C|x|^2$ for each $x \in \mathbb{R}^2$
    \end{itemize}
    Then $\psi = 0$.
    
\end{proposition}

\begin{proof}
Assuming $\rho, \psi$ satisfy the stated hypotheses, define $f:\mathbb{C} \to \mathbb{C}$ by
    \begin{equation*}
        f(z) = \psi - i\varphi
    \end{equation*}
where $\varphi \in H^1_{\mathrm{loc}}$ is a stream function for $\psi$ as in \eqref{def:f}. Hence by the preceding discussion, $f$ satisfies the Beltrami equation 
    \begin{equation*}
        \partial_{\bar{z}} f = \mu_f(z)\partial_{z} f
    \end{equation*}
and is $\alpha$-quasiregular. 
    Since $\rho, \psi$ are $m$-fold symmetric, so is $f$; thus $f(z) = f(e^{\frac{2\pi i}{m}}z)$. In addition, we may assume $f(0) = 0$, since $\varphi$ is defined up to a constant and $\psi(0) = 0$.     Note that the symmetry of $\rho$ and $f$ imply 
    \begin{equation}
    \label{eqn:mu-symm}
        e^{\frac{4\pi i}{m}}\mu_f(z) =\mu_f( e^{\frac{2\pi i}{m}}z)
    \end{equation}

    By \Cref{thm:stoilow}, we may write
    \begin{equation*}
        f(z) = \Phi(\chi(z))
    \end{equation*}
    where $\Phi$ is a holomorphic function and $\chi(z) \in H^1_{\mathrm{loc}}(\mathbb{C})$ is the unique normalized (i.e.  $\chi(0) = 0$, $\chi(1) = 1$, and  $\chi(\infty) = \infty$) $K$-quasiconformal homeomorphism given by the measurable Riemann mapping theorem. 

    Consider now the homeomorphism $\frac{\chi(e^{\frac{2\pi i}{m}}z)}{\chi(e^{\frac{2\pi i}{m}})}$ which is also normalized, and by \eqref{eqn:mu-symm} solves the same equation as $\chi(z)$. Hence by uniqueness, 
    \begin{equation*}
       \chi(e^{\frac{2\pi i}{m}}z)= \lambda\chi(z)
    \end{equation*}
    with $\lambda = \chi(e^{\frac{2\pi i}{m}})$. Iterating $m$ times, we deduce
    \begin{equation*}
        \chi(z) = \chi(e^{2\pi i}z)  = \lambda^m \chi(z)
    \end{equation*}
    so that $\lambda$  is necessarily an $m$-th root of unity---a primitive one in fact, since $\chi$ is a homeomorphism. Hence $\lambda = e^{\frac{2\pi i n}{m}}$ for some $n$ with $\gcd{(n, m)} = 1$. 

    Now $f$ is $m-$fold symmetric, so that
    \begin{equation*}
       \Phi(\chi(z)) =  \Phi(\chi(e^{\frac{2\pi i}{m}}z)) = \Phi(e^{\frac{2\pi i n}{m}}\chi(z)).
    \end{equation*}
    Thus $\Phi(w) = \Phi(e^{\frac{2\pi i n}{m}}w)$ for $w \in \chi(\mathbb{C}) = \mathbb{C}$.

    By assumption, we have
    \begin{equation*}
        |\mathrm{Re}(\Phi(\chi(z))|= |\psi(z)| \le C|z|^2.
    \end{equation*}
    Setting $w = \chi(z)$, we have from \Cref{lem:quasiconformal-holder} and \eqref{ineq:quasi-holder-est} that
    \begin{equation*}
        |\mathrm{Re}(\Phi(w))|\le C|\chi^{-1}(w)|^2 \le C_{\alpha} \max\{|w|^{\alpha}, |w|^{1/\alpha}\}^2 \le C_\alpha|w|^{2\alpha}, \quad |w| \ge 1
    \end{equation*}
    Now since $\Phi$ is holomorphic, the above estimate implies that it is a polynomial of degree not exceeding $\lfloor2\alpha\rfloor$. Moreover, the invariance $\Phi(w) = \Phi(e^{\frac{2\pi i n}{m}}w)$ and $\gcd(n,m)=1$ imply that every non-constant monomial appearing in $\Phi$ has degree a positive multiple of $m$. Since $\deg \Phi \le 2\alpha < m$, no such monomial can occur. Hence $\Phi$ is constant, and $f(0) = 0$ implies $\Phi \equiv 0$. Thus $f \equiv 0$ and hence $\psi \equiv 0$.
    This concludes the proof of the proposition. 
\end{proof}
\subsection{Existence}
Our approach is to solve the existence problem for
\begin{equation}
\label{eqn:elliptic-inhom}
    \nabla\cdot(\rho \nabla \psi) = f
\end{equation}
by the method of continuity. Define
\begin{equation}
\label{def:X_m}
    X_m := \left\{ f \in \dot{W}^{2, \infty}(\mathbb{R}^2)  : f\ m\text{-fold symmetric and }\|f\|_{X} :=\left\|\frac{ f}{|x|^2}\right\|_{L^\infty}  + \|\nabla^2 f\|_{\mathring{C}^{0, \gamma}}< \infty. \right \}
\end{equation}
The main issue is to establish \textit{a priori} estimates for \eqref{eqn:elliptic-inhom}. We will achieve this in two steps. First we use a scaling argument to prove Schauder estimates in the $\mathring{C}^{k, \gamma}$ framework as a consequence of the classical Schauder estimates in the $C^{k, \gamma}$ setting (see \cite{gilbarg1998elliptic} for instance). Secondly, using  a ``blow up'' argument, we prove an $L^\infty$ \textit{a priori} estimate  on $|x|^{-2}\psi$ as a consequence of the uniqueness assertion of \Cref{prop:uniqness-elliptic}. It is in this step that we use $m$-fold symmetry. We are forced to argue indirectly here because in closing the \textit{a priori} estimate, a naive use of the maximum principle does not seem to work without a smallness assumption on $r|\partial_r \rho|$ (or $m$ large relative to the size of $r|\partial_r \rho|$). Since $r|\partial_r \rho(t)|$ can grow exponentially in time for a transported density $\rho(t)$, we cannot afford such an assumption. 

For the following proof, it is convenient to use the equivalent norm  (see \eqref{eqn:C-ring-norm-eq})
\begin{equation*}
    \|f\|_{L^\infty} + \sup_{x\ne y} \left\{\min\{|x|^\gamma, |y|^\gamma\}\frac{|f(x)-f(y)|}{|x-y|^\gamma}\right\}
\end{equation*}
on $\mathring{C}^{0, \gamma}$.

\begin{proofof}{\Cref{thm:elliptic-existence}}
   Let $s \in [0, 1]$ and define $L_{\rho, s}:  X_m\to\mathring{C}_m^{0,\gamma}$ by
\begin{equation*}
   L_{\rho, s}[\psi] := (1-s)\Delta\psi + s(\rho\Delta\psi + \nabla\psi \cdot \nabla \rho) = \nabla\cdot (\rho_s\nabla \psi), 
\end{equation*}
where $\rho_s = 1-s + s\rho$.

Thus $L_0 = \Delta$ and $L_1 = \nabla\cdot(\rho \nabla)$. By \Cref{thm:poisson}, $L_{\rho,0}=\Delta$ maps $X_m$ onto $\mathring{C}_m^{0,\gamma}$. To invoke the method of continuity, we need to prove the \textit{a priori} estimate
\begin{equation}
\label{ineq:aprori-X_m}
    \|\psi\|_{X_m}\le C\|L_{\rho, s}\psi\|_{\mathring{C}^{0,\gamma}}
\end{equation}
for some $C >0$, and all $s \in  [0,1]$, $\psi \in X_m$. This together with \Cref{prop:uniqness-elliptic} will establish the theorem. 

\underline{Step 1.} First, we will prove the following \textit{a priori} estimate as a consequence of classical Schauder estimates and a scaling argument: there is a constant $C = C(\|\rho\|_{\mathring{C}^{1,\gamma}}, \inf \rho, \sup \rho)$ so that
\begin{equation}
\label{ineq:apriori-schauder-2+}
   \|\nabla^2 \psi\|_{\mathring{C}^{0,\gamma}}\le C\left(\left\|\frac{\psi}{|x|^2}\right\|_{L^\infty} + \| L_{\rho, s} \psi\|_{\mathring{C}^{0,\gamma}}\right)
\end{equation}
for all $\psi \in X_m$.
Indeed, let $A_R = \{x: \frac{R}{4} < |x| < 2R\}$ and $A_R' =  \{x: \frac{R}{2} < |x| < \frac{3R}{2} \}$, with $R > 0$. Now let $f_s := L_{\rho, s} \psi$ and define $\rho_R(x) = \rho_s(Rx)$,  $\psi_R(x) = \psi(Rx)$ and  $f_R(x) = R^2f_s(Rx)$. Then we have
\begin{equation*}
    \rho_R\Delta \psi_R + \nabla \rho_R \cdot \nabla \psi_R = f_R.
\end{equation*}
Thus by classical Schauder estimates on $A_1$, there is a constant $C = C(\|\rho_R\|_{C^{1,\gamma}(A_1)}, \inf \rho, \sup \rho)$ (we may take $C$ to be polynomial in the first argument) such that--omitting the dependence of $C$ on $\inf \rho, \sup \rho$--
\begin{align*}
       \|\nabla^{2} \psi_R\|_{L^\infty(A_1')} + [\nabla^{2}\psi_R]_{C^\gamma(A_1')} &\le C\left(\|\rho_R\|_{C^{1,\gamma}(A_1)}\right)\left(\|\psi_R\|_{L^\infty(A_1)} + \| f_R\|_{C^{0,\gamma}(A_1)}\right) \\
       &= C\left(\|\rho_s\|_{\mathring{C}^{1,\gamma}(A_R)}\right)\left(\|\psi\|_{L^\infty(A_R)} + R^2\| f_s\|_{\mathring{C}^{0,\gamma}(A_R)}\right)\\
        &\lesssim C\left(\|\rho\|_{\mathring{C}^{1,\gamma}(A_R)}\right)\left(\|\psi\|_{L^\infty(A_R)} + R^2\| f_s\|_{\mathring{C}^{0,\gamma}(A_R)}\right)
\end{align*}
This implies that
\begin{equation*}
   R^{2}\|\nabla^{2} \psi\|_{L^\infty(A_R')} + R^{2+\gamma}[\nabla^{2}\psi]_{C^\gamma(A_R')} \le C(\|\rho\|_{\mathring{C}^{1,\gamma}(A_R)})(\|\psi\|_{L^\infty(A_R)} + R^2\| f_s\|_{\mathring{C}^{0,\gamma}(A_R)})
\end{equation*}
and upon dividing by $R^2$,
\begin{align*}
   \|\nabla^{2} \psi\|_{L^\infty(A_R')} + R^{\gamma}[\nabla^{2}\psi]_{C^\gamma(A_R')} &\le C\left(\|\rho\|_{\mathring{C}^{1,\gamma}(A_R)}\right)\left(R^{-2}\|\psi\|_{L^\infty(A_R)} + \| f_s\|_{\mathring{C}^{0,\gamma}(A_R)}\right) \\
   &\le C\left(\|\rho\|_{\mathring{C}^{1,\gamma}(A_R)}\right)\left(\left\|\frac{\psi}{|x|^2}\right\|_{L^\infty(A_R)} + \| f_s\|_{\mathring{C}^{0,\gamma}(A_R)}\right)\\
   &\le C\left(\|\rho\|_{\mathring{C}^{1,\gamma}(\mathbb{R}^2)}\right)\left(\left\|\frac{\psi}{|x|^2}\right\|_{L^\infty(\mathbb{R}^2)} + \| f_s\|_{\mathring{C}^{0,\gamma}(\mathbb{R}^2)}\right)\\
   &\le C\left(M\right)\left(\left\|\frac{\psi}{|x|^2}\right\|_{L^\infty(\mathbb{R}^2)} + \| f_s\|_{\mathring{C}^{0,\gamma}(\mathbb{R}^2)}\right)
\end{align*}

The right hand side above is independent of $R$, hence it follows in particular that, for $x \in \mathbb{R}^2\setminus \{0\}$, 
\begin{equation*}
    |\nabla^{2}\psi(x)|  \le  C\left(M\right)\left(\left\|\frac{\psi}{|x|^2}\right\|_{L^\infty(\mathbb{R}^2)} + \| f_s\|_{\mathring{C}^{0,\gamma}(\mathbb{R}^2)}\right)
\end{equation*}
Next, let $x,y \in \mathbb{R}^2\setminus \{0\}$. If $|x-y| <\frac{1}{2}|x|$, then $x, y \in A_{|x|}$ and hence by the preceding,
\begin{multline*}
    \min\{|x|^\gamma, |y|^\gamma\}\frac{|\nabla^2\psi(x)-\nabla^2\psi(y)|}{|x-y|^\gamma}\le |x|^\gamma\frac{|\nabla^2\psi(x)-\nabla^2\psi(y)|}{|x-y|^\gamma} \\ \le |x|^\gamma[\nabla^2\psi]_{C^\gamma(A'_{|x|})} \le C\left(M\right)\left(\left\|\frac{\psi}{|x|^2}\right\|_{L^\infty(\mathbb{R}^2)} + \| f_s\|_{\mathring{C}^{0,\gamma}(\mathbb{R}^2)}\right)
\end{multline*}
On the other hand, if $|x-y| \ge \frac{1}{2}|x|$, then
\begin{align*}
    \min\{|x|^\gamma, |y|^\gamma\}\frac{|\nabla^2\psi(x)-\nabla^2\psi(y)|}{|x-y|^\gamma}\le 2^{1+\gamma}\|\nabla^2 \psi\|_{L^\infty}\frac{ \min\{|x|^\gamma, |y|^\gamma\}}{|x|^\gamma} \le 2^{1+\gamma} \|\nabla^2 \psi\|_{L^\infty}.
\end{align*}
In view of the $L^\infty$ estimate above, this completes the proof of \eqref{ineq:apriori-schauder-2+}. 

\underline{Step 2.} Next, we will prove that there exists $C = C (m, \inf \rho, \sup\rho, M)$ such that
\begin{equation*}
    \left\|\frac{\psi}{|x|^2}\right\|_{L^\infty(\mathbb{R}^2)} \le C\|L_{\rho, s}\psi\|_{\mathring{C}^{0,\gamma}}
\end{equation*}
for all $\psi \in X_m$. The $m$-fold symmetry assumption will play a role here.

For $m > 2c_1/c_0$, define the class
\begin{equation}
\label{def:rho-const-class}
   \mathcal{D}(c_0, c_1, m, M) := \{\rho \in \mathring{C}_m^{1,\gamma}(\mathbb{R}^2): \inf \rho = c_0, \sup \rho = c_1, \|\rho\|_{\mathring{C}^{1,\gamma}} \le M\}. 
\end{equation}

We shall show in fact that there exists a constant  $C = C(c_0, c_1, m, M) > 0$ such that for all  $\rho \in \mathcal{D}(c_0, c_1, m, M)$, the following inequality holds for all $\psi \in X_m \text{ and } s \in [0,1]$
\begin{equation}
    \label{ineq:apriori:Linfty-schauder}
    \left\|\frac{\psi}{|x|^2}\right\|_{L^\infty(\mathbb{R}^2)} \le C\|L_{\rho, s}\psi\|_{\mathring{C}^{0,\gamma}}
\end{equation}
We remark that  the constant $C$ is uniform for all $\rho$ in $\mathcal{D}$.

Now suppose towards a contradiction that \eqref{ineq:apriori:Linfty-schauder} is not true. Then there are sequences of $\rho^j \in \mathcal{D}$,  $\psi_j \in X_m$ and $s_j \in [0, 1]$, $j \ge 1$, with $\left\|\frac{\psi_j}{|x|^2}\right\|_{L^\infty} = 1$ such that
\begin{equation*}
    \|L_{\rho^j, s_j} \psi_j \|_{\mathring{C}^{0,\gamma}} < \frac{1}{j}.
\end{equation*}
We hereafter denote $f_j := L_{\rho^j, s_j} \psi_j$ and $\rho_j := 1 - s_j + s_j\rho^j$. Thus $\nabla \cdot(\rho_j \nabla \psi_j) = f_j$ and $\|f_j \|_{\mathring{C}^{0,\gamma}} < \frac{1}{j}$. In particular,
\begin{equation}
    \label{conv:f_j}
    f_j \to 0 \quad \mathrm{in} \quad L^\infty(\mathbb{R}^2).
\end{equation}

Since $\left\|\frac{\psi_j}{|x|^2}\right\|_{L^\infty} = 1$, there exists a point $x_j \in \mathbb{R}^2 \setminus \{0\}$ so that
\begin{equation}
    \label{ineq:counter:case1:ext}
    \frac{|\psi_j(x_j)|}{|x_j|^2}  \ge \frac{1}{2}
\end{equation}
Set $x_j = R_j y_j$ where $|y_j| =1$ and $R_j = |x_j|$ and define
\begin{equation*}
    \tilde{\psi}_j(x) = \frac{\psi_j(R_j x)}{R_j^2}, \quad \tilde{\rho}_j(x) = \rho_j(R_j x). 
\end{equation*}

Note that the sequences $\tilde{{\psi}}_j(x)$ and $\tilde{\rho}_j(x)$ constructed above satisfy the following properties:
\begin{enumerate}
    \item $\|\tilde{\psi}_j\|_{X_m} = \|\psi_j\|_{X_m}$ and $\left\|\frac{\tilde{\psi}_j}{|x|^2}\right\|_{L^\infty} = \left\|\frac{\psi_j}{|x|^2}\right\|_{L^\infty} = 1$ for all $j$.
    \item $\|\tilde{\rho}_j\|_{\mathring{C}^{1,\gamma}} = \|\rho_j\|_{\mathring{C}^{1,\gamma}}$ for all $j$.
    \item for each $j \ge 1$, $|\tilde{\psi}_j(y_j)|  \ge \frac{1}{2}$ for some $y_j \in \mathbb{S}^1$.
    \item $\nabla \cdot(\tilde{\rho}_j(x) \nabla \tilde{\psi}_j(x)) = f_j(R_jx)$ for all $j$.
    \item $\tilde{\psi}_j$ and $\tilde{\rho}_j$ are $m$-fold symmetric for all $j$.
\end{enumerate}

In view of \eqref{ineq:apriori-schauder-2+} and the scale invariance of the $\mathring{C}^{0,\gamma}$ norm, we have
\begin{equation*}
   \|\nabla^2 \tilde{\psi}_j\|_{\mathring{C}^{0,\gamma}} = \|\nabla^2 \psi_j\|_{\mathring{C}^{0,\gamma}}\le C\left(1 + \|f_j \|_{\mathring{C}^{0,\gamma}}\right) \le 2C.
\end{equation*}
Thus by the Arzela-Ascoli theorem, we deduce that there exist $s_\infty \in [0, 1]$ and $m$-fold symmetric functions $\psi_\infty$ and $\rho_\infty$ such that, up to a subsequence,

\begin{align}
        s_j &\to s_\infty\\
       \label{conv:C1top} \tilde{\psi}_j &\to \psi_\infty \quad \mathrm{in} \quad C^1_{\mathrm{loc}}(\mathbb{R}^2)\\
        \label{conv:C2top}  \tilde{\psi}_j &\to \psi_\infty \quad \mathrm{in} \quad C^2_{\mathrm{loc}}(\mathbb{R}^2\setminus \{0\})\\
        \label{conv:rhoC0top}\tilde{\rho}_j &\to \rho_\infty \quad \mathrm{in} \quad L^\infty_{\mathrm{loc}}(\mathbb{R}^2\setminus \{0\})
\end{align}
Let $\chi \in C^\infty_0(\mathbb{R}^2)$ be a test function. Then from property 4, we have
\begin{equation*}
    -\int \tilde{\rho}_{j} \nabla \tilde{\psi}_j \cdot \nabla \chi = \int \chi f_j(R_j \ \cdot)
\end{equation*}
Since $\|f_j\|_{L^\infty}\to 0$, the right hand side converges to $0$ as $j \to \infty$, while the left hand side converges in view of \eqref{conv:C1top} and \eqref{conv:rhoC0top} (and the dominated convergence theorem for instance) to
\begin{equation*}
        -\int \rho_{\infty} \nabla \psi_\infty \cdot \nabla \chi = 0
\end{equation*}

Next, by pointwise convergence and $\tilde{\psi}_j(0) = 0$, we have $|\psi_\infty(x)| \le |x|^2$. Furthermore, we note that 
\begin{equation*}
2\frac{\sup \rho_\infty}{\inf \rho_\infty} \le 2\frac{1-s_\infty + s_\infty c_1}{1-s_\infty + s_\infty c_0} \le 2\frac{c_1}{c_0} < m.
\end{equation*}
Hence by \Cref{prop:uniqness-elliptic}, $\psi_\infty \equiv 0$. On the other hand, after passing to a further subsequence we may assume $y_j \to y_\infty \in \mathbb{S}^1$. Since $\tilde{\psi}_j \to \psi_\infty$ uniformly on $\mathbb{S}^1$, property 3 gives
\begin{equation*}
    |\psi_\infty(y_\infty)| = \lim_{j\to\infty}|\tilde{\psi}_j(y_j)| \ge \frac{1}{2},
\end{equation*}
a contradiction.

This completes the proof of \eqref{ineq:apriori:Linfty-schauder} and together with \eqref{ineq:apriori-schauder-2+}, this proves \eqref{ineq:aprori-X_m}. Hence by the continuity method $L_{\rho, s}$ maps $X_m$ onto $\mathring{C}_m^{0,\gamma}$ for all $s \in [0,1]$, and the desired conclusion follows. 
\end{proofof}
\begin{remark}
     We emphasize that  the constant $C$ on the right hand side of \eqref{ineq:aprori-X_m} is uniform for all $\rho$ in $\mathcal{D}$. Furthermore $C$ (i.e. the best possible constant in \eqref{ineq:aprori-X_m} for $\rho$ in $\mathcal{D}$) is non-decreasing in $M$, since $\mathcal{D}(c_0, c_1, m, M_1) \subset \mathcal{D}(c_0, c_1, m, M_2)$ for any $M_2 > M_1$.
\end{remark}
\begin{remark}
    Note also that since $\psi(0)=\nabla\psi(0) = 0$ for $\psi \in X_m$, we have by the Hardy inequality that $\left\|\frac{\nabla\psi}{|x|}\right\|_{L^\infty} \le \|\nabla^2\psi\|_{L^\infty}$; hence all the claimed estimates on $\psi$ in  \Cref{thm:elliptic-existence} are proved.
\end{remark}
We now record higher regularity estimates.
\begin{corollary}
\label{cor:schauder-higher-ord}
    Assume that $\gamma \in (0,1), k\ge 0$ and  that $0 < \inf \rho\le \rho \in \mathring{C}_m^{k+1,\gamma}(\mathbb{R}^2)$ with $\|\rho\|_{\mathring{C}_m^{k+1,\gamma}} \le M$ and $f \in \mathring{C}^{k,\gamma}_m(\mathbb{R}^2)$, $m > 2 \alpha$. If $\psi \in X_m$ satisfies
    \begin{equation*}
        \nabla \cdot(\rho\nabla \psi) = f,
    \end{equation*}
    then $\nabla^2\psi \in  \mathring{C}^{k,\gamma}(\mathbb{R}^2)$ and
    \begin{equation}
    \label{ineq:schauder-kth}
        \|\nabla^2 \psi\|_{\mathring{C}^{k,\gamma}} \le C_k(m, \inf\rho, \sup \rho,M) \|f\|_{\mathring{C}_m^{k,\gamma}},
    \end{equation}
\end{corollary}

\begin{proof}
The proof of the case $k = 0$ is contained in the proof of \eqref{ineq:apriori-schauder-2+} above. The case $k  > 0$ is almost identical---the $k$-th order Schauder estimates on annuli together with the rescaling argument yields \eqref{ineq:schauder-kth}. We omit the details.
\end{proof}

 \section{Local well-posedness for inhomogeneous Euler with discrete rotational symmetry}
 \label{sec:lwp}
In this section, we present a proof of \Cref{thm:lwp}. Recall that the system under consideration is
\begin{equation*}
    \begin{rcases}
        \partial_t \rho + u \cdot \nabla \rho  = 0 \\
        \partial_t \Omega + u\cdot \nabla \Omega = \frac{1}{2}\nabla|u|^2 \cdot \nabla^\perp \rho\\
    \end{rcases}
\end{equation*}
for $t > 0$, where for each $t$, $u = \nabla^\perp \psi$, and $\psi$ is the unique solution of the elliptic equation
\begin{equation*}
    \nabla\cdot (\rho \nabla \psi) = \Omega.
\end{equation*}
The definition of $\psi$ now makes sense in light of \Cref{thm:elliptic-existence}.

As preparation for the non-linear problem, we begin with linear transport equations in the $\mathring{C}^{k, \gamma}(\mathbb{R}^2)$ framework.
\subsection{Linear transport equations in $\mathring{C}^{k, \gamma}$}
\label{sec:transport}
We consider the linear transport equation 
\begin{equation}
\label{eqn:linear-transport}
    \begin{rcases}
        \partial_t f + v \cdot \nabla f = g\\
        f(0, x) = f_0(x)
    \end{rcases}
\end{equation}
in the class $\mathring{C}^{k, \gamma}(\mathbb{R}^2)$. The main conclusion is the following. 
\begin{proposition}
    \label{prop:linear-theory}
    Let $k\ge1$ and $0<\gamma\le1$. Assume $v(t,0)=0$ for a.e. $t$,
    \[
       \nabla v\in L^1(0,T;\mathring{C}^{k-1,\gamma}(\mathbb{R}^2)),
       \qquad
       g\in L^1(0,T;\mathring{C}^{k,\gamma}(\mathbb{R}^2)),
    \]
    and let $f_0\in\mathring{C}^{k,\gamma}(\mathbb{R}^2)$. Then \eqref{eqn:linear-transport} has a unique solution satisfying
    \[
       f\in L^\infty(0,T;\mathring{C}^{k,\gamma}(\mathbb{R}^2))
       \cap C([0,T];\mathring{C}^{k,\beta}(\mathbb{R}^2))
       \qquad\text{for every }0<\beta<\gamma.
    \]
    For almost every $t$,
    \begin{equation}
    \label{ineq:linear-differential}
        \frac{\mathrm{d}}{\mathrm{d}t}\|f(t)\|_{\mathring{C}^{k,\gamma}}
        \le C_{k,\gamma}\|\nabla v(t)\|_{\mathring{C}^{k-1,\gamma}}\|f(t)\|_{\mathring{C}^{k,\gamma}}
        +\|g(t)\|_{\mathring{C}^{k,\gamma}}.
    \end{equation}
    Consequently,
    \begin{equation}
    \label{ineq:linear-gronwall}
        \|f(t)\|_{\mathring{C}^{k,\gamma}}
        \le \left(\|f_0\|_{\mathring{C}^{k,\gamma}}+\int_0^t\|g(s)\|_{\mathring{C}^{k,\gamma}}\,\mathrm{d}s\right)
        \exp\left(C_{k,\gamma}\int_0^t\|\nabla v(s)\|_{\mathring{C}^{k-1,\gamma}}\,\mathrm{d}s\right).
    \end{equation}
    For $k=0$, it suffices to assume $v(t,0)=0$, $\nabla v\in L^1(0,T;L^\infty)$, and $g\in L^1(0,T;\mathring{C}^{0,\gamma})$. Then the solution belongs to $L^\infty(0,T;\mathring{C}^{0,\gamma})\cap C([0,T];\mathring{C}^{0,\beta})$ for every $0<\beta<\gamma$, and for almost every $t$,
    \begin{equation}
    \label{ineq:linear-differential-k0}
        \frac{\mathrm{d}}{\mathrm{d}t}\|f(t)\|_{\mathring{C}^{0,\gamma}}
        \le C_\gamma\|\nabla v(t)\|_{L^\infty}\|f(t)\|_{\mathring{C}^{0,\gamma}}
        +\|g(t)\|_{\mathring{C}^{0,\gamma}}.
    \end{equation}
\end{proposition}

\begin{proof}
Since $\nabla v\in L^1_tL^\infty_x$, the flow $\Phi(s;t,x)$ is well defined and bi-Lipschitz, with
\begin{equation}
\label{ineq:lin:chord-arc-est}
        \exp\left(-\int_s^t\|\nabla v(\tau)\|_{L^\infty}\,\mathrm{d}\tau\right)
        \le \frac{|\Phi(s;t,x)-\Phi(s;t,y)|}{|x-y|}
        \le
        \exp\left(\int_s^t\|\nabla v(\tau)\|_{L^\infty}\,\mathrm{d}\tau\right).
\end{equation}
The solution is given by
\[
    f(t,x)=f_0(\Phi(t;0,x))+\int_0^t g(s,\Phi(t;s,x))\,\mathrm{d}s.
\]
The differential estimate \eqref{ineq:linear-differential} is the standard H\"older transport estimate: commute spatial derivatives of order at most $k$ with $\partial_t+v\cdot\nabla$, use the weighted product estimates from \Cref{app:a}, and estimate the top-order H\"older seminorm along pairs of characteristics. The case $k=0$ is identical without differentiating. Gr\"onwall gives \eqref{ineq:linear-gronwall} and its zeroth-order counterpart. Finally, the characteristic formula, \eqref{ineq:lin:chord-arc-est}, and the composition estimates of Appendix A show strong continuity in every lower H\"older exponent $0<\beta<\gamma$.
\end{proof}
\subsection{Local existence and uniqueness}
\label{sec:lwp:exist}
In the estimation of $\mathring{C}^{k, \gamma}$ norms, we will often use without explicit mention the calculus lemmas in \Cref{app:a}. We also note the following lemma for use below.
\begin{lemma}
    For $\psi$ in the space $X_m$ defined in \eqref{def:X_m}, we have 
\begin{equation}
\label{ineq:C-ring-frac-est}
    \left\|\frac{\nabla\psi}{|x|}\right\|_{\mathring{C}^{0,\gamma}} \lesssim \|\nabla^2\psi\|_{L^\infty} + \left\|\frac{\nabla\psi}{|x|}\right\|_{L^\infty}
\end{equation}
with the implicit constant depending on $\gamma$.
\end{lemma}
\begin{proof}
     Indeed, since $\nabla\psi(0) = 0$, we have $\left\|\frac{\nabla\psi}{|x|}\right\|_{L^\infty} \le \|\nabla^2\psi\|_{L^\infty}$ and so for the $\mathring{C}^{0, \gamma}$ semi-norm, it suffices consider $x,y$ such that $|x-y| < \frac{1}{2}|x|$. Then 
\begin{align*}
  |x|^\gamma \left( \frac{\nabla\psi(x)}{|x|}- \frac{\nabla\psi(y)}{|y|}\right) =  \frac{\nabla\psi(x)-\nabla \psi(y)}{|x|}|x|^\gamma + |x|^\gamma\frac{\nabla\psi(y)}{|y|}\frac{|y|-|x|}{|x|}
\end{align*}
and so
\begin{align*}
    \frac{|x|^\gamma \left|\frac{\nabla\psi(x)}{|x|}- \frac{\nabla\psi(y)}{|y|}\right|}{|x-y|^\gamma} &\le \left(\|\nabla^2\psi\|_{L^\infty} +\left\|\frac{\nabla\psi}{|x|}\right\|_{L^\infty}\right)\left(\frac{|x-y|}{|x|}\right)^{1-\gamma}\\
    &\le 2^{\gamma-1}\left(\|\nabla^2\psi\|_{L^\infty} +\left\|\frac{\nabla\psi}{|x|}\right\|_{L^\infty}\right)
\end{align*}
\end{proof}

\begin{proofof}{\Cref{thm:lwp}}
    We prove uniqueness first. Suppose $(\rho_i, \Omega_i)$ for $i = 1,2$ is a pair of distinct solutions on $[0, T]$ with the same initial data $(\rho_0, \Omega_0)$. Then
    \begin{equation*}
        \partial_t (\rho_1-\rho_2) + u_1 \cdot \nabla(\rho_1 - \rho_2) = (u_2 - u_1) \cdot \nabla \rho_2
    \end{equation*}
    and
    \begin{equation*}
               \partial_t (\Omega_1-\Omega_2) + u_1 \cdot \nabla(\Omega_1 - \Omega_2) = (u_2 - u_1) \cdot \nabla \Omega_2 + \frac{1}{2}(\nabla |u_1|^2 \cdot \nabla^\perp\rho_1 - \nabla |u_2|^2 \cdot \nabla^\perp\rho_2)
    \end{equation*}
    For $\nabla^\perp \psi_i = u_i$, we have $\nabla \cdot (\rho_i \nabla \psi_i) = \Omega_i$ and hence for each $t$,
    \begin{equation}
    \label{eqn:uniq:elliptic-diff}
        \nabla \cdot(\rho_1\nabla (\psi_1 - \psi_2)) = \nabla\cdot((\rho_2 - \rho_1)\nabla \psi_2) + \Omega_1 - \Omega_2
    \end{equation}
    Let $M_T$ be such that
    \begin{equation*}
        \sup_{0\le t \le T}\|\rho_i\|_{ \mathring{C}^{2, \gamma}}, \sup_{0\le t \le T}\|\Omega_i\|_{ \mathring{C}^{1, \gamma}} \le M_T
    \end{equation*}
     In the following, $C(M_T)$ is a constant depending on $M_T$ but its value may change from line to line. By \eqref{eqn:uniq:elliptic-diff}, \Cref{thm:elliptic-existence} and \eqref{ineq:C-ring-frac-est}, we have

     \begin{align}
         \left\|\frac{u_1-u_2}{|x|}\right\|_{\mathring{C}^{0,\gamma}} +   \|\nabla^2 (\psi_1-\psi_2)\|_{\mathring{C}^{0,\gamma}} &\le C(M_T)(\|\Omega_1 - \Omega_2\|_{\mathring{C}^{0,\gamma}} + \|\rho_1 - \rho_2\|_{\mathring{C}^{1, \gamma}} \left\|\frac{\nabla \psi_2}{|x|}\right\|_{\mathring{C}^{0,\gamma}}) \nonumber \\
         &\le C(M_T)(\|\Omega_1 - \Omega_2\|_{\mathring{C}^{0, \gamma}} + \|\rho_1 - \rho_2\|_{\mathring{C}^{1, \gamma}} )     \label{ineq:elliptic-diff}
     \end{align}
     We shall also use the corresponding first-order bound. Indeed,
     \begin{equation*}
     |x|\nabla\left(\frac{u_1-u_2}{|x|}\right)
     = \nabla(u_1-u_2)-\frac{u_1-u_2}{|x|}\otimes\frac{x}{|x|},
     \end{equation*}
     and the weighted product estimate therefore gives
     \begin{equation}
     \label{ineq:uniq-quotient-C1}
     \left\|\frac{u_1-u_2}{|x|}\right\|_{\mathring{C}^{1,\gamma}}
     \lesssim \|\nabla(u_1-u_2)\|_{\mathring{C}^{0,\gamma}}
     +\left\|\frac{u_1-u_2}{|x|}\right\|_{\mathring{C}^{0,\gamma}}
     \le C(M_T)\left(\|\Omega_1-\Omega_2\|_{\mathring{C}^{0,\gamma}}+\|\rho_1-\rho_2\|_{\mathring{C}^{1,\gamma}}\right).
     \end{equation}
     Denote $F = (u_2 - u_1) \cdot \nabla \Omega_2 + \frac{1}{2}(\nabla |u_1|^2 \cdot \nabla^\perp\rho_1 - \nabla |u_2|^2 \cdot \nabla^\perp\rho_2)$. Thus
    \begin{multline}
\|F\|_{\mathring{C}^{0,\gamma}} \le \left\|\frac{u_1-u_2}{|x|}\right\|_{\mathring{C}^{0,\gamma}}\||x| \nabla \Omega_2\|_{\mathring{C}^{0,\gamma}} + \left\|\frac{u_1-u_2}{|x|}\right\|_{\mathring{C}^{0,\gamma}}\||x| \nabla \rho_1\|_{\mathring{C}^{0,\gamma}}\|\nabla u\|_{\mathring{C}^{0,\gamma}}\\ + \left\|\frac{u_2}{|x|}\right\|_{\mathring{C}^{0,\gamma}}\||x| \nabla \rho_1\|_{\mathring{C}^{0,\gamma}}\|\nabla (u_1-u_2)\|_{\mathring{C}^{0,\gamma}} +  \left\|\frac{u_2}{|x|}\right\|_{\mathring{C}^{0,\gamma}}\||x| \nabla (\rho_1-\rho_2)\|_{\mathring{C}^{0,\gamma}}\|\nabla u_2\|_{\mathring{C}^{0,\gamma}}
    \end{multline}
    and in view of \eqref{ineq:elliptic-diff}, 
    \begin{equation*}
        \|F\|_{\mathring{C}^{0,\gamma}} \le C(M_T) (\|\Omega_1 - \Omega_2\|_{\mathring{C}^{0,\gamma}} + \|\rho_1 - \rho_2\|_{\mathring{C}^{1, \gamma}} )   
    \end{equation*}
Next, denote $G = (u_2 - u_1) \cdot \nabla \rho_2$ and note using \eqref{ineq:uniq-quotient-C1} that 
\begin{equation*}
    \|G\|_{\mathring{C}^{1, \gamma}} \le  \left\|\frac{u_1-u_2}{|x|}\right\|_{\mathring{C}^{1, \gamma}}\||x| \nabla \rho_2\|_{\mathring{C}^{1, \gamma}} \lesssim \|\rho_2\|_{\mathring{C}^{2, \gamma}} C(M_T)(\|\Omega_1 - \Omega_2\|_{\mathring{C}^{0, \gamma}} + \|\rho_1 - \rho_2\|_{\mathring{C}^{1, \gamma}} )
\end{equation*}
Thus $Q :=\Omega_1- \Omega_2$ and $R := \rho_1  - \rho_2$ respectively solve
\begin{equation*}
    \partial_t Q + u_1 \cdot\nabla Q = F,
\end{equation*}
\begin{equation*}
    \partial_t R + u_1 \cdot\nabla R = G.
\end{equation*}
Hence from \Cref{prop:linear-theory} and estimates on $F, G$ above, we have
\begin{align*}
        \|Q(t)\|_{\mathring{C}^{0, \gamma}}\le C(M_T)\int_0^t (\|Q\|_{\mathring{C}^{0, \gamma}} + \|R\|_{\mathring{C}^{1, \gamma}} )
\end{align*}
and
\begin{align*}
        \|R(t)\|_{\mathring{C}^{1, \gamma}} \le C(M_T)\int_0^t (\|Q\|_{\mathring{C}^{0, \gamma}} + \|R\|_{\mathring{C}^{1, \gamma}} )
\end{align*}
Since $R(0) = Q(0) = 0$, Gr\"{o}nwall's inequality implies that $Q(t) = R(t) = 0$ for $t \in [0, T]$. This contradiction concludes the proof of uniqueness. 

Next, we construct a solution via the following scheme.    Define $\Omega^0 = \Omega_0$ and $\rho^0 = \rho_0$. Given $\Omega^n$ and $\rho^n$, define $\Omega^{n+1}$ and $\rho^{n+1}$ inductively via
\begin{equation}
    \label{eqn:inho-2deuler:vort:iter-rho}
    \begin{rcases}
        \partial_t \rho^{n+1} + u^{n+1} \cdot \nabla \rho^{n+1}  = 0 \\
        \rho^{n+1}(0,x) = \rho_0 \\
    \end{rcases}
\end{equation}
and
\begin{equation}
    \label{eqn:inho-2deuler:vort:iter-Omega}
    \begin{rcases}
        \partial_t \Omega^{n+1} + u^{n+1}\cdot \nabla \Omega^{n+1} = \frac{1}{2}\nabla|u^{n+1}|^2 \cdot \nabla^\perp \rho^n =: F^n\\
        \Omega^{n+1}(0, x) = \Omega_0
    \end{rcases}
\end{equation}
where for each $t$, $u^n = \nabla^\perp \psi^n$, and $\psi^n$ satisfies the elliptic equation
\begin{equation}
\label{eqn:elliptic-bslaw-n}
    \nabla\cdot (\rho^{n} \nabla \psi^{n+1}) = \Omega^{n}.
\end{equation}

Set
\[
M_0:=\max\left\{\|\rho_0\|_{\mathring{C}^{2,\gamma}},\|\Omega_0\|_{\mathring{C}^{1,\gamma}}\right\},
\qquad M:=2M_0.
\]
Fix $0<\beta<\gamma$. Suppose inductively that
\[
\rho^n\in L^\infty(0,T;\mathring{C}_m^{2,\gamma})\cap C([0,T];\mathring{C}_m^{2,\beta}),
\qquad
\Omega^n\in L^\infty(0,T;\mathring{C}_m^{1,\gamma})\cap C([0,T];\mathring{C}_m^{1,\beta}),
\]
and
\begin{equation*}
    \sup_{0\le t \le T}\|\rho^n(t)\|_{\mathring{C}^{2, \gamma}},  \sup_{0\le t \le T}\|\Omega^n(t)\|_{\mathring{C}^{1, \gamma}} \le M.
\end{equation*}

The base case $n=0$ clearly holds.
Since $\Omega^n(t) \in \mathring{C}_m^{1, \gamma}$ and $\rho^n(t) \in \mathring{C}_m^{2, \gamma}$, \Cref{thm:elliptic-existence} and \Cref{cor:schauder-higher-ord} imply that $\psi^{n+1}(t)$ satisfies
\begin{equation}
\label{ineq:psi-n+1}
    \left\|\frac{\psi^{n+1}(t)}{|x|^2}\right\|_{L^\infty} + \left\|\frac{\nabla \psi^{n+1}(t)}{|x|}\right\|_{L^\infty} + \|\nabla^2 \psi^{n+1}(t)\|_{\mathring{C}^{1, \gamma}} \le C(M)\|\Omega^n(t)\|_{\mathring{C}^{1, \gamma}} \le MC(M)
\end{equation}
for $0 \le t \le T$. 
The constant $C$ also depends on $m, \inf \rho_0, \sup \rho_0$, but we have suppressed this dependence, since these parameters remain fixed for the current discussion.
Note that \eqref{ineq:psi-n+1} in particular shows that $\psi^{n+1}(t)|x|^{-2}, u^{n+1}(t)|x|^{-1} \in L^\infty(0, T; L^\infty)$ and $\nabla u^{n+1}(t) \in L^\infty(0, T;\mathring{C}_m^{1, \gamma})$. For $t,s\in[0,T]$ we have
\begin{equation}
\label{eqn:elliptic-diff}
\nabla\cdot\big(\rho^n(t)\nabla(\psi^{n+1}(t)-\psi^{n+1}(s))\big)
=\nabla\cdot\big((\rho^n(s)-\rho^n(t))\nabla\psi^{n+1}(s)\big)+\Omega^n(t)-\Omega^n(s).
\end{equation}
Applying the elliptic difference estimate to \eqref{eqn:elliptic-diff} at the lower exponent $\beta$ gives
\[
\left\|\frac{\psi^{n+1}(t)-\psi^{n+1}(s)}{|x|^2}\right\|_{L^\infty}
+\left\|\frac{u^{n+1}(t)-u^{n+1}(s)}{|x|}\right\|_{L^\infty}
+\|\nabla^2(\psi^{n+1}(t)-\psi^{n+1}(s))\|_{\mathring{C}^{1,\beta}}
\]
\[
\le C(M,\beta)\left(\|\Omega^n(t)-\Omega^n(s)\|_{\mathring{C}^{1,\beta}}
+\|\rho^n(t)-\rho^n(s)\|_{\mathring{C}^{2,\beta}}\right).
\]
Consequently $u^{n+1}|x|^{-1}\in C([0,T];L^\infty)$ and $\nabla u^{n+1}\in C([0,T];\mathring{C}^{1,\beta})$, while the endpoint $\mathring{C}^{1,\gamma}$ norm remains uniformly bounded by \eqref{ineq:psi-n+1}.

By \Cref{prop:linear-theory}, $\rho^{n+1}\in L^\infty(0,T;\mathring{C}_m^{2,\gamma})\cap C([0,T];\mathring{C}_m^{2,\beta})$. Its differential H\"older estimate \eqref{ineq:linear-differential}, together with \eqref{ineq:psi-n+1}, gives a constant $C(M)$, depending also on the fixed parameters $m,\gamma,\inf\rho_0,\sup\rho_0$, such that
\begin{equation*}
    \|\rho^{n+1}(t)\|_{\mathring{C}^{2,\gamma}}
    \le e^{tC(M)}\|\rho_0\|_{\mathring{C}^{2,\gamma}}.
\end{equation*}
Similarly, \Cref{prop:linear-theory} gives $\Omega^{n+1}\in L^\infty(0,T;\mathring{C}_m^{1,\gamma})\cap C([0,T];\mathring{C}_m^{1,\beta})$, and the source term in \eqref{eqn:inho-2deuler:vort:iter-Omega} satisfies
\begin{align*}
    \|F^n(t)\|_{\mathring{C}^{1,\gamma}}
    &\le C\||x|\nabla\rho^n(t)\|_{\mathring{C}^{1,\gamma}}
    \|\nabla u^{n+1}(t)\|_{\mathring{C}^{1,\gamma}}
    \left\|\frac{u^{n+1}(t)}{|x|}\right\|_{\mathring{C}^{1,\gamma}} \\
    &\le C(M).
\end{align*}
Hence the inhomogeneous transport estimate gives
\begin{equation*}
    \|\Omega^{n+1}(t)\|_{\mathring{C}^{1,\gamma}}
    \le e^{tC(M)}\left(\|\Omega_0\|_{\mathring{C}^{1,\gamma}}+tC(M)\right).
\end{equation*}
Since $M=2M_0$ is already fixed, we may choose $T=T(M)$ sufficiently small that
\[
e^{TC(M)}\le \frac32,
\qquad
TC(M)\le \frac{M_0}{3}.
\]
Then, for every $0\le t\le T$,
\[
\|\rho^{n+1}(t)\|_{\mathring{C}^{2,\gamma}}
\le \frac32 M_0<M,
\]
and
\[
\|\Omega^{n+1}(t)\|_{\mathring{C}^{1,\gamma}}
\le \frac32\left(M_0+\frac{M_0}{3}\right)=M.
\]
Thus the induction closes and we obtain the uniform bounds
\begin{equation}
\label{ineq:Omega_n-and-rho_n}
    \|\Omega^n(t)\|_{\mathring{C}^{1,\gamma}} \le M, \quad \|\rho^{n}(t)\|_{\mathring{C}^{2,\gamma}}  \le M
\end{equation}
for all $n \ge 0$ and $0\le t\le T$.

In addition, we have

\begin{equation}
\label{ineq:psi_n}
    \left\|\frac{\psi^{n}}{|x|^2}\right\|_{L^\infty} + \left\|\frac{\nabla \psi^{n}}{|x|}\right\|_{L^\infty} + \|\nabla^2 \psi^{n}\|_{\mathring{C}^{1, \gamma}} \le C(M)M
\end{equation}
for all $n \ge 0$. Since $\partial_t\rho^n = - u^n \cdot \nabla \rho^n$ and $\partial_t \Omega^n = -u^{n}\cdot \nabla \Omega^n + F^{n-1}$, we also have
\begin{equation}
\label{ineq:dtrho_n}
    \|\partial_t\rho^n\|_{L^\infty} \le \left\|\frac{u^n}{|x|}\right\|_{L^\infty}\||x|\nabla \rho^n\|_{L^\infty} \le M^2C(M)
\end{equation}
and
\begin{align}
\label{ineq:dtOmega_n}
        \|\partial_t\Omega^n\|_{L^\infty} &\le \left\|\frac{u^n}{|x|}\right\|_{L^\infty}\||x|\nabla \Omega^n\|_{L^\infty} + \|F^{n-1}\|_{L^\infty}\\ 
        & \nonumber\lesssim M^2C(M) + \||x|\nabla\rho^{n-1}(s)\|_{L^\infty} \|\nabla u^{n}(s)\|_{L^\infty} \left\|\frac{u^{n}(s)}{|x|}\right\|_{ L^\infty}\\
        & \nonumber\lesssim M^2C(M) + M^3C(M)^2
\end{align}
The time derivative bounds \eqref{ineq:dtrho_n}--\eqref{ineq:dtOmega_n}, together with the uniform spatial H\"older bounds \eqref{ineq:Omega_n-and-rho_n}, imply by interpolation that $\{\Omega^n\}$ is equicontinuous in $C^1_{\mathrm{loc}}(\mathbb{R}^2\setminus\{0\})$ and $\{\rho^n\}$ is equicontinuous in $C^2_{\mathrm{loc}}(\mathbb{R}^2\setminus\{0\})$. Hence Arzel\`a--Ascoli yields, after passing to a subsequence,
\begin{equation*}
   \Omega^n \to \Omega \quad \text{in } C([0,T];C^1_{\mathrm{loc}}(\mathbb{R}^2\setminus\{0\})),
   \qquad
   \rho^n \to \rho \quad \text{in } C([0,T];C^2_{\mathrm{loc}}(\mathbb{R}^2\setminus\{0\})).
\end{equation*}
Interpolating also in the weighted H\"older scale, in any exponent $0<\gamma'<\gamma$, and applying the elliptic difference estimate to \eqref{eqn:elliptic-bslaw-n} gives equicontinuity of $\{\psi^n\}$ in $C^2_{\mathrm{loc}}$; thus, after extracting the same subsequence,
\begin{equation*}
    \psi^n \to \psi \quad \text{in } C([0,T];C^2_{\mathrm{loc}}(\mathbb{R}^2\setminus\{0\})).
\end{equation*}

We may then pass to the limit in the equations
\begin{equation*}
    \begin{cases}
            \Omega^n(t,x) = \Omega_0(x) - \int_0^t u^n \cdot \nabla\Omega^n + \frac{1}{2}\nabla|u^{n}|^2 \cdot \nabla^\perp \rho^{n-1} \ \mathrm{d}s \\
    \rho^n(t,x) = \rho_0(x) - \int_0^t u^n \cdot \nabla \rho^n \ \mathrm{d}s\\
    \rho^{n-1} \Delta\psi^n + \nabla \rho^{n-1} \cdot \nabla \psi^n = \Omega^{n-1}
    \end{cases}
\end{equation*}
to obtain
\begin{equation}
\label{eqn:limiting-int-form}
    \begin{cases}
    \Omega(t,x) = \Omega_0(x) - \int_0^t u \cdot \nabla\Omega + \frac{1}{2}\nabla|u|^2 \cdot \nabla^\perp \rho \ \mathrm{d}s \\
    \rho(t,x) = \rho_0(x) - \int_0^t u \cdot \nabla \rho\ \mathrm{d}s\\
    \rho\Delta\psi + \nabla \rho \cdot \nabla \psi = \Omega
    \end{cases}
\end{equation}
for $(t,x) \in [0, T] \times \mathbb{R}^2\setminus \{0\}$. 

Next, since $|\Omega^n (t,x)|, ||x|\nabla\Omega^n(t,x)| \le M$ for each $(t,x) \in [0, T] \times \mathbb{R}^2\setminus \{0\}$, the pointwise convergence of $\Omega^n, \nabla \Omega^n$ implies that $|\Omega(t,x)|, ||x|\nabla\Omega(t,x)| \le M$ for $(t,x) \in [0, T] \times \mathbb{R}^2\setminus \{0\}$. In addition for $x,y \in \mathbb{R}^2\setminus \{0\},x \ne y$ we have
\begin{equation*}
    \frac{||x|^{1+\gamma}\nabla\Omega^n(t,x) -|y|^{1+\gamma}\nabla\Omega^n(t,y)|}{|x-y|^\gamma} \le M
\end{equation*}
so that the pointwise convergence of $\nabla\Omega^n$ shows that
\begin{equation*}
    \frac{||x|^{1+\gamma}\nabla\Omega(t,x) -|y|^{1+\gamma}\nabla\Omega(t,y)|}{|x-y|^\gamma} \le M
\end{equation*}
Hence $\Omega \in L^\infty(0, T; \mathring{C}^{1, \gamma}(\mathbb{R}^2))$. Likewise, $\rho \in L^\infty(0, T; \mathring{C}^{2, \gamma}(\mathbb{R}^2))$. We have similarly $|\psi(t,x)| \le MC(M)|x|^2$, $|\nabla\psi(t,x)| \le C(M)M|x|$ and $\nabla^2 \psi \in L^\infty(0, T; \mathring{C}^{1, \gamma})$. Thus for each $t$, $\psi(t)$ is actually continuously differentiable at $0$ and $\psi(t,0) = \nabla\psi(t, 0) = 0$. Note that the equations \eqref{eqn:limiting-int-form} may not be defined at the origin, but the equalities hold in $L^\infty(\mathbb{R}^2)$ for each $t$ (i.e. the left and right hand side agree as $L^\infty$ functions). 

The uniform endpoint bounds pass to the limit, so $\Omega\in L^\infty(0,T;\mathring{C}^{1,\gamma})$ and $\rho\in L^\infty(0,T;\mathring{C}^{2,\gamma})$. Moreover, the compactness argument above, applied at any exponent $0<\beta<\gamma$, gives $\Omega\in C([0,T];\mathring{C}_m^{1,\beta})$ and $\rho\in C([0,T];\mathring{C}_m^{2,\beta})$. Since $\beta<\gamma$ was arbitrary, this is the regularity asserted in \Cref{thm:lwp} and concludes the proof of existence.

\end{proofof}

\section{Radially homogeneous solutions of inhomogeneous Euler}
\label{sec:rad-hom}
This section is devoted to the study of radially homogeneous solutions of \eqref{eqn:inho-2deuler:vort}. As illustrated in the introduction, the Euler system in this case reduces to the one dimensional system \eqref{eqn:1dsys} on $\mathbb{S}^1$:
\begin{equation*}
    \begin{rcases}
        \partial_t g + 2G\partial_\theta g   = -(4G^2 + (\partial_\theta G)^2)\partial_\theta P \\
        \partial_t P + 2G\partial_\theta P = 0 \\
        \partial_\theta(P\partial_\theta G) + 4PG = g
    \end{rcases}
\end{equation*}

Consider the inhomogeneous Biot-Savart law (or rather its reduction under radial homogeneity) \[\partial_\theta(P\partial_\theta G) + 4PG = g.\] When $P$ is a constant, one may study the equation \[\partial_{\theta\theta} G + 4G = g\] on the Fourier side, and $m$-fold symmetry with $m \ge 3$ ensures orthogonality to $\sin(2\theta)$ and $\cos(2\theta)$ in the kernel of $\partial_{\theta\theta}  + 4$. 

In the non-constant case, the matter is not so simple. Fortunately, the second order ODE  $\partial_\theta(P\partial_\theta G) + 4PG = g$ is of Sturm-Liouville type, and ideas from the Sturm-Liouville theory allow us to study this system under the assumption of $m$-fold symmetry with $m > 4\sqrt{\alpha}$, where $\alpha = \frac{\sup P}{\inf P}$. The purpose of the assumption $m > 4\sqrt{\alpha}$ is to ensure that the kernel of $\partial_\theta(P\partial_\theta \cdot ) + 4P$ with periodic boundary conditions on $[-\frac{\pi}{m}, \frac{\pi}{m}]$ is trivial, so that the operator is invertible. Since $P$ is transported, it is obvious--although crucial--that $\alpha$ is invariant under the evolution of $P$.

In the following section, we shall study the operator $\partial_\theta(P\partial_\theta \cdot ) + 4P$  and use our findings to derive \textit{a priori} estimates for the system \eqref{eqn:1dsys} leading to global existence under \Cref{assumps}, i.e. we will prove \Cref{thm:gwp-1dsys}.

\subsection{Boundary value problem}
\label{sec:bvp}
In this section, we study the boundary value problem on $[-\frac{\pi}{m}, \frac{\pi}{m}]$ for $m \in \mathbb{Z}_+$
\begin{equation}
    \label{eqn:bvp}
    \begin{rcases}
    (P G')' + 4PG = g\\
    G(-\frac{\pi}{m}) = G(\frac{\pi}{m}) \\
    G'(-\frac{\pi}{m}) =  G'(\frac{\pi}{m})
\end{rcases}
\end{equation}
Here the unknown is $G:[-\frac{\pi}{m}, \frac{\pi}{m}] \to \mathbb{R}$, $G = G(\theta)$. The function $P:[-\frac{\pi}{m}, \frac{\pi}{m}] \to \mathbb{R}_+, P \in C^1([-\frac{\pi}{m}, \frac{\pi}{m}] )$ is a given even function that is monotonically decreasing on $[0, \frac{\pi}{m}]$  and the source term is $g:[-\frac{\pi}{m}, \frac{\pi}{m}] \to \mathbb{R}, g \in C^0([-\frac{\pi}{m}, \frac{\pi}{m}] )$. 

Denote 
\begin{equation*}
    \alpha =\alpha(P) := \frac{\sup P}{\inf P} = \frac{P(0)}{P(\frac{\pi}{m})},
\end{equation*}
where the `sup' and `inf' are taken over $[-\frac{\pi}{m}, \frac{\pi}{m}]$. Observe that $\alpha > 1$ for non-constant $P$.

\subsubsection{Homogeneous problem}

We begin with some key lemmas pertaining to the homogeneous problem
\begin{equation}
    \label{eqn:hom-bvp}
    \begin{rcases}
    (P G')' + 4PG = 0\\
    G(-\frac{\pi}{m}) = G(\frac{\pi}{m}) \\
    G'(-\frac{\pi}{m}) =  G'(\frac{\pi}{m})
\end{rcases}
\end{equation}

\begin{lemma}
\label{lemma:zeros}
    Suppose that $m > 4\sqrt{\alpha} $ and $(P G')' + 4PG = 0$ on $[-\frac{\pi}{m}, \frac{\pi}{m}]$. Then $G$ has at most one zero in $[-\frac{\pi}{m}, \frac{\pi}{m}]$. 
\end{lemma}

\begin{proof}
Suppose to the contrary that $[a, b] \subset [-\frac{\pi}{m}, \frac{\pi}{m}]$ and $G(a) = G(b) = 0$, with $G$ satisfying $(PG')' + 4PG = 0$. Then multiplying this equation by $G$ and integrating by parts, one finds 
    \begin{equation}
    \label{hom-energy-id}
        \int_a^b P(G')^2 - \int_a^b4PG^2 = 0
    \end{equation}
where the boundary terms vanish because $G(a) = G(b) = 0$. On the other hand, by Wirtinger's inequality and $b-a \le \frac{2\pi}{m}$, we have
\begin{align*}
    4\int_a^b PG^2 &\le 4 \sup P \int_a^b G^2 \\
    &\le 4(\sup P )\Big(\frac{b-a}{\pi}\Big)^2 \int_a^b (G')^2 \\
    &\le \frac{16\alpha}{m^2}  \inf P \int_a^b (G')^2 \\ 
    &<\int_a^b P(G')^2 \\
\end{align*}
which contradicts \eqref{hom-energy-id}.
\end{proof}
By arguing identically to the preceding with $[a, b] = [-\frac{\pi}{m}, \frac{\pi}{m}]$, we deduce
\begin{lemma}
\label{lemma:trivial-kernel}
If $m > 4\sqrt{\alpha} $, then the only solution to \eqref{eqn:hom-bvp} is the trivial one $G \equiv 0$.    
\end{lemma}
Next we note several useful properties of some particular solutions of the homogeneous ODE $(PG')' + 4PG = 0$ which will be used to build a Green's kernel for the boundary value problem \eqref{eqn:bvp}. 

On $[-\frac{\pi}{m}, \frac{\pi}{m}]$, let $a = a(\theta)$ to be the unique solution to the initial value problem
\begin{equation}
    \label{eqn:a-ivp}
    \begin{rcases}
    (P G')' + 4PG = 0\\
    G(0) = 1\\
    G'(0) = 0
\end{rcases},
\end{equation}
let $b = b(\theta)$ be the unique solution to 
\begin{equation}
    \label{eqn:b-ivp}
    \begin{rcases}
    (P G')' + 4PG = 0\\
    G(0) = 0\\
    G'(0) = 1
\end{rcases},
\end{equation}
and finally let $c_m = c_m(\theta)$ be the unique solution to 
\begin{equation}
    \label{eqn:cm-ivp}
    \begin{rcases}
    (P G')' + 4PG = 0\\
    G(\frac{\pi}{m}) = 0\\
    G'(\frac{\pi}{m}) = 1
\end{rcases}
\end{equation}

\begin{lemma}
\label{lemma:abd-props}
Assume that $m > 4\sqrt{\alpha} $. Then
    \begin{enumerate}[label=\normalfont(\roman*)]
        \item $a$ is an even function and $b$ is an odd function.
        \item  $a > 0$ on $[-\frac{\pi}{m}, \frac{\pi}{m}]$ and $b > 0$ on $(0, \frac{\pi}{m}]$.
        \item $a' < 0$ on $(0, \frac{\pi}{m}]$.
        \item $c_m < 0$ on  $[0, \frac{\pi}{m})$.
        \item $ c_m(\theta) = \frac{1}{\alpha}(a(\textstyle{\frac{\pi}{m}}) b(\theta) - b(\textstyle{\frac{\pi}{m}})a(\theta))$.
    \end{enumerate}
\end{lemma}

\begin{proof}
$ $\newline
    \begin{enumerate}[label=\normalfont(\roman*)]
    \item Since $P$ is even, $a(-\theta)$ solves \eqref{eqn:a-ivp} and thus by uniqueness, $a$ is even. Similarly, $-b(-\theta)$ solves \eqref{eqn:b-ivp} and so $b$ is odd by uniqueness. 
    \item That $a > 0$ on $[-\frac{\pi}{m}, \frac{\pi}{m}]$ follows at once from \Cref{lemma:zeros} and the evenness of $a$; likewise $b(0) = 0, b'(0)=1$ and \Cref{lemma:zeros} imply that $b > 0$ on $(0, \frac{\pi}{m}]$.
    \item Since $a > 0$, we have 
    \begin{equation*}
        (Pa')' = -4aP < 0.
    \end{equation*}
    Thus $Pa'$ is decreasing; hence $P(\theta)a'(\theta) < P(0)a'(0) = 0$ for $\theta > 0$ and so $a'(\theta) < 0$ as desired.
    \item That $c_m < 0$ on  $[0, \frac{\pi}{m})$ is due also to \Cref{lemma:zeros} and $c_m(\frac{\pi}{m}) = 0, c_m'(\frac{\pi}{m})=1$.
    \item Since  $a$ and $b$ are linearly independent with Wronskian $W(a,b)(\theta) = \frac{P(0)}{P(\theta)}$, $c_m(\theta) = r_1a(\theta) + r_2b(\theta)$. Solving for $r_1, r_2$ gives the conclusion. 
    \end{enumerate}
\end{proof}

We will use the following elementary variation of parameters formula.
\begin{lemma}
\label{lem:elementary-comparison}
If $u''+4u=F$ on $[0,L]$, with $u(0)=0$ and $u'(0)=1$, then
\[
u(\theta)=\frac12\sin(2\theta)+\frac12\int_0^\theta \sin(2(\theta-\zeta))F(\zeta)\,\mathrm d\zeta.
\]
In particular, if $L<\frac\pi4$ and $F\ge0$, then $u(\theta)\ge\frac12\sin(2\theta)$ on $[0,L]$.
\end{lemma}
\begin{proof}
The formula is follows by using variation of parameters; the final assertion follows because the sine kernel is nonnegative for $0\le\zeta\le\theta\le L<\frac\pi4$.
\end{proof}

The next lemma gives the pointwise estimates for $b$ and $c_m$ that will be used below.
\begin{lemma}
\label{lem:b-c-ptwise}
Suppose that $m>4\sqrt{\alpha}$. Then, for $0\le \theta\le \frac{\pi}{m}$,
\begin{enumerate}[label=\normalfont(\roman*)]
    \item
    \begin{equation}
    \label{ineq:b-ptwise-new}
        \frac12\sin(2\theta)\le b(\theta)\le \frac{\alpha}{2}\sin(2\theta).
    \end{equation}
    \item
    \begin{equation}
    \label{ineq:cm-ptwise-new}
        -\frac12\sin\left(2\left(\frac{\pi}{m}-\theta\right)\right)
        \le c_m(\theta)\le
        -\frac{1}{2\alpha}\sin\left(2\left(\frac{\pi}{m}-\theta\right)\right).
    \end{equation}
\end{enumerate}
\end{lemma}

\begin{proof}
Set $L=\frac{\pi}{m}$. Since $m>4\sqrt\alpha$, we have $L<\frac{\pi}{4}$. We first record the derivative bounds needed for the comparison. Let
\[
q(\theta):=P(\theta)b'(\theta).
\]
Then $q'=-4Pb$ and
\[
q''+4q=-4P'b\ge0,
\]
since $P'\le0$ and $b>0$ on $(0,L]$ by \Cref{lemma:abd-props}. As $q(0)=P(0)$ and $q'(0)=0$, variation of constants gives
\[
q(\theta)-P(0)\cos(2\theta)
=\frac12\int_0^\theta \sin(2(\theta-\zeta))\big(-4P'(\zeta)b(\zeta)\big)\,\mathrm d\zeta\ge0.
\]
Hence $q>0$ on $[0,L]$, while $q'=-4Pb<0$ gives $q(\theta)\le P(0)$. Therefore
\begin{equation}
\label{ineq:bprime-in-proof}
0<b'(\theta)\le \frac{P(0)}{P(\theta)}\le\alpha.
\end{equation}
Using the representation of $c_m$ in \Cref{lemma:abd-props}, together with $a(L)>0$, $a'<0$, $b(L)>0$ and $b'>0$, we obtain $c_m'>0$. Since $(Pc_m')'=-4Pc_m>0$ and $c_m'(L)=1$,
\begin{equation}
\label{ineq:cmprime-in-proof}
0<c_m'(\theta)\le \frac{P(L)}{P(\theta)}\le1.
\end{equation}

We now compare with the constant-coefficient equation. Dividing the equation for $b$ by $P$,
\[
b''+4b=-\frac{P'}P b'\ge0.
\]
Thus
\begin{equation}
\label{eqn:b-voc}
b(\theta)=\frac12\sin(2\theta)
+\frac12\int_0^\theta \sin(2(\theta-\zeta))
\left(-\frac{P'(\zeta)}{P(\zeta)}b'(\zeta)\right)\,\mathrm d\zeta.
\end{equation}
The integral is nonnegative, which gives the lower bound in \eqref{ineq:b-ptwise-new}. For the upper bound, \eqref{ineq:bprime-in-proof} gives
\[
-\frac{P'}P b'\le \frac{\mathrm d}{\mathrm d\theta}\left(\frac{P(0)}P\right).
\]
Moreover, for $0\le\zeta\le\theta\le L<\frac\pi4$,
$0\le\sin(2(\theta-\zeta))\le\sin(2\theta)$. Hence \eqref{eqn:b-voc} yields
\[
b(\theta)\le \frac12\sin(2\theta)\frac{P(0)}{P(\theta)}
\le \frac\alpha2\sin(2\theta).
\]

For $c_m$, define
\[
d(x):=-c_m(L-x),\qquad Q(x):=P(L-x),\qquad 0\le x\le L.
\]
Then $d(0)=0$, $d'(0)=1$, $d,d'>0$, $Q'\ge0$, and
\[
d''+4d=-\frac{Q'}Q d'.
\]
Therefore
\begin{equation}
\label{eqn:d-voc}
d(x)=\frac12\sin(2x)-\frac12\int_0^x
\sin(2(x-\zeta))\frac{Q'(\zeta)}{Q(\zeta)}d'(\zeta)\,\mathrm d\zeta,
\end{equation}
which immediately gives $d(x)\le\frac12\sin(2x)$. From \eqref{ineq:cmprime-in-proof},
\[
d'(\zeta)\le \frac{Q(0)}{Q(\zeta)}.
\]
Using again $\sin(2(x-\zeta))\le\sin(2x)$ in \eqref{eqn:d-voc},
\[
\begin{aligned}
d(x)
&\ge \frac12\sin(2x)\left(1-\int_0^x
\frac{Q'(\zeta)Q(0)}{Q(\zeta)^2}\,\mathrm d\zeta\right)\\
&=\frac{Q(0)}{2Q(x)}\sin(2x)
\ge \frac1{2\alpha}\sin(2x).
\end{aligned}
\]
Substituting $x=L-\theta$ proves \eqref{ineq:cm-ptwise-new}.
\end{proof}

\begin{lemma}
\label{lem:b-c-prime}
For $m>4\sqrt\alpha$,
\begin{enumerate}[label=\normalfont(\roman*)]
    \item $0<b'(\theta)\le \frac{P(0)}{P(\theta)}\le\alpha$ on $[0,\frac{\pi}{m}]$;
    \item $0<c_m'(\theta)\le \frac{P(\frac{\pi}{m})}{P(\theta)}\le1$ on $[0,\frac{\pi}{m}]$.
\end{enumerate}
\end{lemma}
\begin{proof}
These are precisely \eqref{ineq:bprime-in-proof} and \eqref{ineq:cmprime-in-proof} from the proof of \Cref{lem:b-c-ptwise}.
\end{proof}

\subsubsection{Green's function construction}
By \Cref{lemma:trivial-kernel}, the homogeneous periodic boundary value problem associated to \eqref{eqn:bvp} has trivial kernel. Thus, for every $g\in C^0([ -\frac{\pi}{m},\frac{\pi}{m}])$, \eqref{eqn:bvp} has a unique solution. 

We will only need an explicit Green's function when $g$ is odd. In this case the evenness of $P$ and uniqueness imply that $G$ is odd. Indeed, if $G$ solves \eqref{eqn:bvp}, define
\begin{equation*}
    \widetilde G(\theta):=-G(-\theta).
\end{equation*}
Since $P$ is even and $g$ is odd, $\widetilde G$ satisfies the same differential equation and the same periodic boundary conditions as $G$. Hence $\widetilde G=G$ by uniqueness. Consequently $G(0)=0$, and periodicity together with oddness gives
\begin{equation*}
    G\left(\frac{\pi}{m}\right)=G\left(-\frac{\pi}{m}\right)=-G\left(\frac{\pi}{m}\right),
\end{equation*}
so that $G(\frac{\pi}{m})=0$. Therefore, on $[0,\frac{\pi}{m}]$, $G$ is the unique solution of
\begin{equation}
    \label{eqn:reduced-bvp}
    \begin{rcases}
    (P G')' + 4PG = g\\
    G(0) = 0 \\
    G(\frac{\pi}{m}) = 0.
\end{rcases}
\end{equation}
The homogeneous problem associated to \eqref{eqn:reduced-bvp} has trivial kernel by \Cref{lemma:zeros}, so the reduced boundary value problem has a unique solution. We now construct its Green's kernel.
\begin{lemma}
   Let $m > 4\sqrt{\alpha} $.  The function 
    \begin{equation}
    \label{eqn:greens:reduced}
    H_m(\theta, \zeta) = -\frac{1}{P(0)c_m(0)}\begin{cases}
        b(\theta) c_m(\zeta), \quad 0\le \theta < \zeta\\
        c_m(\theta)b(\zeta), \quad \zeta < \theta \le \frac{\pi}{m}
    \end{cases}
\end{equation}
is a Green's function for \eqref{eqn:reduced-bvp} and furthermore $H_m(\theta, \zeta) < 0$ for all $\theta \in (0, \frac{\pi}{m})$.
\end{lemma}

\begin{proof}
    Recall that $b(0) = c_m(\frac{\pi}{m})= 0$. A Green's function for \eqref{eqn:reduced-bvp} is given by
        \begin{equation*}
    H_m(\theta, \zeta) = -\frac{1}{P(\zeta)W(b, c_m)(\zeta)}\begin{cases}
        b(\theta) c_m(\zeta), \quad 0\le \theta < \zeta\\
        c_m(\theta)b(\zeta), \quad \zeta < \theta \le \frac{\pi}{m}
    \end{cases}
\end{equation*}
where $W(b, c_m)(\zeta)$ denotes the Wronskian of $b$ and $c_m$. In fact by Abel's formula,
\begin{align*}
    W(b, c_m)(\zeta) = W(b, c_m)(0) \exp\left(-\int_0^\zeta\frac{P'}{P}  \ \mathrm{d}\theta\right) = -c_m(0)\frac{P(0)}{P(\zeta)},
\end{align*}
which gives \eqref{eqn:greens:reduced}. To see that $H_m(\theta, \zeta) < 0$, note by \Cref{lemma:abd-props} that $c_m < 0$  and $b > 0$ on $(0, \frac{\pi}{m})$. 
\end{proof}

We summarize our conclusions in the following proposition.

\begin{proposition}
\label{prop:G-greens-formula}
    Let $m > 4\sqrt{\alpha}$. If $g$ is odd, then the unique solution $G$ to \eqref{eqn:bvp} is odd and, on $[0,\frac{\pi}{m}]$, is given by
    \begin{equation*}
        G(\theta) = \int_{0}^{\frac{\pi}{m}} H_m(\theta, \zeta)g(\zeta) \, \mathrm{d}\zeta.
    \end{equation*}
\end{proposition}

\subsubsection{Green's function bounds}
\label{sec:greensbounds}
We now prove some useful bounds on the kernel $H_m$ constructed above. Our bounds depend only on the quantities $P(0), P(\frac{\pi}{m})$ and $\alpha(P)$, which are preserved if $P = P(t, \cdot)$ solves a transport equation. We often employ without declaration the simple inequalities $\sin\theta \le \theta$ and $\sin\theta \ge \frac{2}{\pi}\theta$ in $[0, \frac{\pi}{2}]$, and $|\sin\theta| \le 1$. 

\begin{proposition}
\label{prop:greens:ptwise}
Assume that $m>4\sqrt\alpha$ and set $L=\frac{\pi}{m}$. Then, for $0\le \theta,\zeta\le L$,
\begin{equation}
\label{ineq:Hm-abs-bounds}
\frac{1}{2\alpha P(0)\sin(2L)}\,\mathcal S(\theta,\zeta)
\le -H_m(\theta,\zeta)
\le \frac{\alpha^2}{2P(0)\sin(2L)}\,\mathcal S(\theta,\zeta),
\end{equation}
where
\[
\mathcal S(\theta,\zeta):=
\begin{cases}
\sin(2\theta)\sin(2(L-\zeta)),&0\le\theta<\zeta,\\
\sin(2\zeta)\sin(2(L-\theta)),&0\le\zeta<\theta.
\end{cases}
\]
\end{proposition}

\begin{proof}
By \Cref{lemma:abd-props}, $|c_m(0)|=\alpha^{-1}b(L)$. Hence \Cref{lem:b-c-ptwise} gives
\[
\frac{1}{2\alpha}\sin(2L)\le |c_m(0)|\le \frac12\sin(2L).
\]
Inserting these bounds, together with \eqref{ineq:b-ptwise-new}--\eqref{ineq:cm-ptwise-new}, into \eqref{eqn:greens:reduced} gives \eqref{ineq:Hm-abs-bounds}.
\end{proof}

\begin{corollary}
\label{cor:greens-Linf}
For $m>4\sqrt\alpha$,
\[
\|H_m\|_{L^\infty}\le \frac{\alpha^2}{4P(0)}\tan\left(\frac{\pi}{m}\right)
\le \frac{\alpha^2}{4P(0)}.
\]
\end{corollary}
\begin{proof}
For $0\le\theta<\zeta\le L$, monotonicity of $\sin(2\theta)$ on $[0,L]$ and the identity
\[
\sin(2\zeta)\sin(2(L-\zeta))\le \sin^2L
\]
give $\mathcal S(\theta,\zeta)\le\sin^2L$; the case $\zeta<\theta$ is identical. The result follows from the upper bound in \eqref{ineq:Hm-abs-bounds} and $L<\frac\pi4$.
\end{proof}

\begin{lemma}
\label{prop:greens:bounds}
Assume that $m>4\sqrt\alpha$. Then
\[
\|\partial_\theta H_m\|_{L^\infty}\le \frac{\pi\alpha^2}{2P(0)}.
\]
\end{lemma}

\begin{proof}
Set $L=\frac{\pi}{m}$. By \Cref{lem:b-c-prime},
\[
0<b'\le\alpha,\qquad 0<c_m'\le1.
\]
Thus, using $c_m(L)=b(0)=0$,
\[
|c_m(\zeta)|\le L-\zeta,\qquad b(\zeta)\le\alpha\zeta.
\]
Also \Cref{lem:b-c-ptwise} gives $|c_m(0)|\ge(2\alpha)^{-1}\sin(2L)$. Differentiating \eqref{eqn:greens:reduced}, in either region $\theta<\zeta$ or $\zeta<\theta$, therefore gives
\[
|\partial_\theta H_m(\theta,\zeta)|
\le \frac{2\alpha^2L}{P(0)\sin(2L)}
\le \frac{\pi\alpha^2}{2P(0)},
\]
where we used $\sin(2L)\ge\frac{4L}{\pi}$.
\end{proof}

\subsection{\textit{a priori estimates}}
In this section, we prove some key \textit{a priori} bounds for smooth solutions to the system \eqref{eqn:1dsys}. A smooth solution of  \eqref{eqn:1dsys} with admissible initial data clearly satisfies
\begin{equation}
\label{eqn:duhamel-g}
    g(t, \theta) = g_0(\phi(0, t,\theta)) - \int_0^t \partial_\theta P(s,\phi(s, t,\theta))(4G^2 + (\partial_\theta G)^2))(s,\phi(s, t,\theta)) \ \mathrm{d}s
\end{equation}
\begin{equation*}
    P(t, \theta)=P_0(\phi(0,t,\theta))
\end{equation*}
where $\phi$ is the flow map for $t, s \in \RR$ defined via
\begin{equation}
\label{eqn:flow-map}
    \frac{\mathrm{d}}{\mathrm{d}t} \phi(t, s, \theta) = 2G(t,  \phi(t, s, \theta)), \quad \phi(s,s,\theta) = \theta,
\end{equation}
i.e. $\phi(t, s, \theta)$ is the position of the particle at time $t$ which at time $s$ is located at $\theta$.  For concision, we will sometimes denote $\phi(t,s,\theta) = \phi_{t,s}(\theta)$ and for given $f$, $f(t, \phi(t,s,\theta)) = f \circ \phi_{t,s}$. We shall use these formulas to control various norms of the solution in terms of $\|g(t, \cdot)\|_{L^1(\mathbb{S}^1)}$. Subsequently, we give some sufficient conditions for the control of $\|g(t, \cdot)\|_{L^1(\mathbb{S}^1)}$ given that $(g_0, P_0)$ satisfy hypotheses \eqref{H1} and \eqref{H2}. This will enable us to formulate a continuation criterion for \eqref{eqn:1dsys} in the next section. 

We present first some basic estimates.
\begin{proposition}[\textit{a priori} estimates]
\label{prop:apriori-ests1}
    Let $(g, P)$ be a smooth solution to \eqref{eqn:1dsys} with initial data $(g_0, P_0)$ satisfying \eqref{H1} and \eqref{H2}. Then for each $t \ge 0$,
    \begin{enumerate}[label=\normalfont(\roman*)]
        \item $P(t, 0) = \sup P(t,\theta) = \sup P_0(\theta) = P_0(0)$ and $P(t, \frac{\pi}{m}) = \inf P(t,\theta) = \inf P_0(\theta) = P_0(\frac{\pi}{m})$.
        \item \begin{equation*}
            \| G (t, \cdot) \|_{C^{1}(\mathbb{S}^1)} \le  C\|g(t, \cdot)\|_{L^1(\mathbb{S}^1)}
        \end{equation*}
        for some $C > 0$ depending only on $P_0$.
        \item \begin{equation*}\|P(t, \cdot)\|_{C^1(\mathbb{S}^1)} \le C\|P_0\|_{C^1}e^{\int_0^t \|g(s, \cdot)\|_{L^1}\mathrm{d}s}\end{equation*}
        \item \begin{equation*}
            \|g(t, \cdot)\|_{C^0(\mathbb{S}^1)} \le C(\|g_0\|_{L^\infty} + e^{\int_0^t \|g(s, \cdot)\|_{L^1}\mathrm{d}s}\int_0^t \|g(s, \cdot)\|_{L^1}^2 \ \mathrm{d}s)
        \end{equation*}
        \item \begin{equation*}
            \| G (t, \cdot) \|_{C^{2}(\mathbb{S}^1)} \le  C\Big(\|g_0\|_{L^\infty} + \|g(t, \cdot)\|_{L^1}+\int_0^t \|g(s, \cdot)\|_{L^1}^2 \ \mathrm{d}s  \Big)e^{\int_0^t \|g(\tau, \cdot)\|_{L^1}\mathrm{d}\tau}
        \end{equation*}
    \end{enumerate}
\end{proposition}
\begin{proof}
$ $\newline
    \begin{enumerate}[label=\normalfont(\roman*)]
        \item Due to $m$-fold symmetry, for each $t \ge 0$, $G(t, \cdot)$ solves the BVP \eqref{eqn:bvp} with $P = P(t, \theta)$ and $g = g(t,\theta)$. Since $g(t, \theta)$ is odd,  $G$ is also odd by \Cref{prop:G-greens-formula} and $G(t, 0) = G(t, \frac{\pi}{m}) = 0$. All the conclusions now follow from the fact that $P$  is transported by $G$. 
        \item By \Cref{prop:G-greens-formula}, on $[0, \frac{\pi}{m}]$, 
        \begin{equation*}
            G(t, \theta) = \int_0^{\frac{\pi}{m}} H_m^t(\theta,\zeta) g(t, \zeta) \ \mathrm{d}\zeta
        \end{equation*}
        where $H_m^t(\cdot, \cdot)$ denotes the Green's kernel for \eqref{eqn:reduced-bvp} when $P=P(t, \cdot)$. Hence by \Cref{cor:greens-Linf,prop:greens:bounds} (recall that $\alpha = \alpha(P)$ is independent of $t$)
        \begin{equation*}
            \| G(t, \cdot)\|_{L^\infty} \le \|H_m^t\|_{L^\infty}\|g(t,\cdot)\|_{L^1(\mathbb{S}^1)} \le \frac{\alpha^2}{4P(t,0)}\|g(t,\cdot)\|_{L^1(\mathbb{S}^1)} = \frac{\alpha^2}{4P_0(0)}\|g(t,\cdot)\|_{L^1(\mathbb{S}^1)} 
        \end{equation*}
        and 
        \begin{equation*}
            \| \partial_\theta G(t, \cdot)\|_{L^\infty} \le  \frac{\pi\alpha^2}{2P(t,0)} \|g(t,\cdot)\|_{L^1(\mathbb{S}^1)} = \frac{\pi\alpha^2}{2P_0(0)} \|g(t,\cdot)\|_{L^1(\mathbb{S}^1)}
        \end{equation*}
        \item From \eqref{eqn:flow-map}, since $G(t, \cdot) \in C^1(\mathbb{S}^1)$ it follows that $\phi(t,s,\cdot) \in C^1(\mathbb{S}^1)$. This is a consequence of smooth dependence of the ODE solution on the parameter $\theta$. Further, since
        \begin{equation*}
            \partial_\theta\phi(t,s,\theta) = 1 + 2\int_s^t \partial_\theta G(\tau, \phi(\tau,s,\theta))\partial_\theta \phi(\tau,s,\theta) \ \mathrm{d}\tau
        \end{equation*}
        we have by (ii) above, for $s \le t$ ($s \ge t$ is identical)
        \begin{align*}
            \|\partial_\theta\phi(t,s,\cdot)\|_{C^0} &\le  C(1 + \int_s^t \|G(\tau, \cdot)\|_{C^1}\|\partial_\theta\phi(\tau,s, \cdot)\|_{C^0}\mathrm{d}\tau)\\
            &\le C(1 + \int_s^t \|g(\tau, \cdot)\|_{L^1}\|\partial_\theta\phi(\tau,s, \cdot)\|_{C^0} \mathrm{d}\tau),\\
        \end{align*}
       Hence by Gr\"{o}nwall,
        \begin{equation}
        \label{ineq:pa_thetaphiC_0}
            \|\partial_\theta\phi(t,s,\cdot)\|_{C^0} \le C\exp\left(\int_s^t \|g(\tau, \cdot)\|_{L^1}\mathrm{d}\tau\right)
        \end{equation}
        Now since $P = P_0(\phi(0,t,\theta))$, we obtain 
        \begin{equation*}
            \|P(t, \cdot)\|_{C^1} \le \|P_0\|_{C^1}(1 + \|\partial_\theta \phi(t,0, \cdot)\|_{C^0}) \le  C\|P_0\|_{C^1}\exp\left(\int_0^t \|g(\tau, \cdot)\|_{L^1}\mathrm{d}\tau\right)
        \end{equation*}
        \item We have
        \begin{align}
            \|g(t, \cdot)\|_{C^0} &\le C(\|g_0\|_{L^\infty} + \int_0^t \|P(s, \cdot)\|_{C^1}\|G(s, \cdot)\|_{C^1}^2 \ \mathrm{d}s) \\
            &\le C(\|g_0\|_{L^\infty} + e^{\int_0^t \|g(\tau, \cdot)\|_{L^1}\mathrm{d}\tau}\int_0^t \|g(s, \cdot)\|_{L^1}^2 \ \mathrm{d}s) 
        \end{align}
        \item From $\partial_{\theta\theta} G =\frac{1}{P}(g - 4PG - \partial_\theta P \partial_\theta G)$, we deduce
        \begin{align}
            \|\partial_{\theta\theta} G(t, \cdot)\|_{C^0} &\le C(\|g(t, \cdot)\|_{C^0} + \|P(t, \cdot)\|_{C^0}\|G(t, \cdot)\|_{C^0} + \|P(t, \cdot)\|_{C^1}\|G(t, \cdot)\|_{C^1})\\
            &\le  C(\|g(t, \cdot)\|_{C^0} + \|g(t, \cdot)\|_{L^1} + \|P(t, \cdot)\|_{C^1}\|g(t, \cdot)\|_{L^1})\\
            &\le C\Big(\|g_0\|_{L^\infty} + e^{\int_0^t \|g(\tau, \cdot)\|_{L^1}\mathrm{d}\tau}\int_0^t \|g(s, \cdot)\|_{L^1}^2 \ \mathrm{d}s + \|g(t, \cdot)\|_{L^1}e^{\int_0^t \|g(\tau, \cdot)\|_{L^1}\mathrm{d}\tau}\Big) \\
            &\le C\Big(\|g_0\|_{L^\infty} + \|g(t, \cdot)\|_{L^1}+\int_0^t \|g(s, \cdot)\|_{L^1}^2 \ \mathrm{d}s  \Big)e^{\int_0^t \|g(\tau, \cdot)\|_{L^1}\mathrm{d}\tau}
        \end{align}
        Together with (ii), this gives the claimed estimate. 
    \end{enumerate}
\end{proof}

In view of \Cref{prop:apriori-ests1} and the formulae \eqref{eqn:duhamel-g}, \eqref{eqn:flow-map}, it is not difficult to see that control of  $\|g(s, \cdot)\|_{L^1}$ propagates $C^{1,\gamma}$ regularity of $g$ and $C^{2,\gamma}$ regularity of $P$ in time: if for each $t \ge 0$, we have  $\sup_{0\le s\le t} \|g(s, \cdot)\|_{L^1} \le M_t < \infty$, then we obtain $\|g(s, \cdot)\|_{C^{1,\gamma}} \le C$ and $ \|P(s, \cdot)\|_{C^{2,\gamma}} \le C$ for $0 \le s \le t$ and $C = C(M_t, \|g_0\|_{C^{1,\gamma}},  \|P_0\|_{C^{1,\gamma}}) < \infty$. Indeed, $g(t, \cdot) \in L^1$ implies that $G(t, \cdot) \in C^2$, which means the flow map is $C^2$ in $\theta$, which is sufficient to propagate the mentioned regularity. 

Next we prove some estimates on the crucial norm $\|g(t, \cdot)\|_{L^1(\mathbb{S}^1)}$. It turns out to be more fruitful to consider instead the quantity
\begin{equation}
\label{def:gweightedP-2L1}
    \int_0^{\frac{\pi}{m}}\frac{g}{P^2} \ \mathrm{d}\theta
\end{equation}
Since $g_-(t,\theta) \le \|g_0\|_{L^\infty}$, we have 
\begin{equation}
\label{ineq:|g|ltg}
   \begin{aligned}
        |g(t,\theta)| &= g_+(t,\theta) + g_-(t,\theta) \\
        &= g(t,\theta) + 2g_-(t,\theta) \\
        &\le g(t,\theta) + 2\|g_0\|_{L^\infty}
   \end{aligned}
\end{equation}
Furthermore  $P_0(\frac{\pi}{m}) \le P(t, \theta) \le P_0(0)$, and $g(t,\theta)$ has odd symmetry and is $m$-fold symmetric. So to estimate $\|g(t, \cdot)\|_{L^1(\mathbb{S}^1)}$, it is sufficient to estimate \eqref{def:gweightedP-2L1}. We also introduce the quantity
\begin{equation}
\label{def:Ibplus}
    \mathcal I_b^+(t) := \int_0^t \frac{1}{b(s,\frac{\pi}{m})}\int_0^{\frac{\pi}{m}} g_+(s,\theta)b(s,\theta)\,\mathrm{d}\theta\,\mathrm{d}s.
\end{equation}
The following estimate is the basis of the continuation criteria we enunciate in \Cref{sec:continuation-crit}. 

\begin{proposition}
\label{prop:apriori-gL1}
Under the assumptions of \Cref{prop:apriori-ests1}, we have for each $t \ge 0$,
\begin{equation*}
\begin{aligned}
&\int_0^{\frac{\pi}{m}}\frac{g}{P^2}(t,\theta)\,\mathrm{d}\theta\\
&\le \Bigg(\int_0^{\frac{\pi}{m}}\frac{g_0}{P_0^2}(\theta)\,\mathrm{d}\theta
+\frac{\pi\alpha^2}{mP_0^2(0)}\|g_0\|_\infty
+\frac{4\pi^2\alpha^3}{m^2P_0^3(0)}\|g_0\|_\infty^2t\Bigg)
 e^{\left(\frac{2\alpha^3}{P_0(0)}\mathcal I_b^+(t)
 +\frac{8\pi\alpha^3}{mP_0(0)}\|g_0\|_\infty t\right)}.
\end{aligned}
\end{equation*}
Consequently,
\begin{equation*}
\begin{aligned}
&\|g(t,\cdot)\|_{L^1(\mathbb{S}^1)}\\
&\le \Bigg(\alpha^2\|g_0\|_{L^1(\mathbb{S}^1)}
+2\pi\alpha^2\|g_0\|_\infty
+\frac{8\pi^2\alpha^3}{mP_0(0)}\|g_0\|_\infty^2t\Bigg)
 e^{\left(\frac{2\alpha^3}{P_0(0)}\mathcal I_b^+(t)
 +\frac{8\pi\alpha^3}{mP_0(0)}\|g_0\|_\infty t\right)}
+4\pi\alpha^2\|g_0\|_\infty.
\end{aligned}
\end{equation*}
\end{proposition}

\begin{proof}
        Consider the equation that $\frac{g}{P^2}$ satisfies:
        \begin{equation*}
            \partial_t\left(\frac{g}{P^2}\right) + 2G\partial_\theta\left(\frac{g}{P^2}\right) = -\frac{\partial_\theta P}{P^2}(4G^2 + (\partial_\theta G)^2)
        \end{equation*}
        Integrating, we obtain
        \begin{equation*}
            \frac{\mathrm{d}}{\mathrm{d}t} \int_0^{\frac{\pi}{m}}\frac{g}{P^2} \ \mathrm{d}\theta = -\int_0^{\frac{\pi}{m}}2G \partial_\theta\left(\frac{g}{P^2}\right) \ \mathrm{d}\theta - \int_0^{\frac{\pi}{m}} \frac{\partial_\theta P}{P^2}(4G^2 + (\partial_\theta G)^2) \ \mathrm{d}\theta
        \end{equation*}
        Multiply the ``Biot-Savart law" \eqref{eqn:BSlaw} by $\frac{2}{P^2}\partial_\theta G$ and integrate by parts to discover the identity

\begin{multline}
    \Big[\frac{4G^2+(\partial_\theta G)^2}{P}\Big]_0^{\frac{\pi}{m}}+\int_0^{\frac{\pi}{m}} (4G^2 + (\partial_\theta G)^2)\frac{\partial_\theta P}{P^2} \ \mathrm{d}\theta \\+  2\int_0^{\frac{\pi}{m}} (\partial_\theta G)^2\frac{\partial_\theta P}{P^2} \ \mathrm{d}\theta = -\int_0^{\frac{\pi}{m}} 2G\partial_\theta \Big(\frac{g}{P^2}\Big) \ \mathrm{d}\theta + \Big[\frac{gG}{P^2}\Big]_0^{\frac{\pi}{m}}
\end{multline}

and after using the boundary conditions on $G$,

\begin{equation*}
    \Big[\frac{(\partial_\theta G)^2}{P}\Big]_0^{\frac{\pi}{m}} +  2\int_0^{\frac{\pi}{m}} (\partial_\theta G)^2\frac{\partial_\theta P}{P^2} \ \mathrm{d}\theta = -\int_0^{\frac{\pi}{m}} 2G\partial_\theta \Big(\frac{g}{P^2}\Big) \ \mathrm{d}\theta -\int_0^{\frac{\pi}{m}} (4G^2 + (\partial_\theta G)^2)\frac{\partial_\theta P}{P^2} \ \mathrm{d}\theta
\end{equation*}
Thus
\begin{equation}
\label{ineq:d/dt-g/P2}
    \frac{\mathrm{d}}{\mathrm{d}t} \int_0^{\frac{\pi}{m}}\frac{g}{P^2} \ \mathrm{d}\theta =  \Big[\frac{(\partial_\theta G)^2}{P}\Big]_0^{\frac{\pi}{m}} +  2\int_0^{\frac{\pi}{m}} (\partial_\theta G)^2\frac{\partial_\theta P}{P^2} \ \mathrm{d}\theta \le  \Big[\frac{(\partial_\theta G)^2}{P}\Big]_0^{\frac{\pi}{m}} 
\end{equation}

Let $b = b(t, \theta)$ and $c_m = c_m(t, \theta)$ be the homogeneous ODE solutions to \eqref{eqn:b-ivp} and \eqref{eqn:cm-ivp} when $P = P(t, \theta)$ in \eqref{eqn:bvp}. Multiplying the equation $\partial_\theta (P\partial_\theta G) + 4PG = g$ by $b(t,\theta), c_m(t,\theta)$ and integrating by parts, (and using $-\alpha c_m(t, 0) = b(t, \frac{\pi}{m})$ from \Cref{lemma:abd-props}) we find

\begin{equation}
    \label{eqn:pa-th-G(0)}
    \partial_\theta G(t, 0) = \frac{\alpha}{P_0(0)b(t,\frac{\pi}{m})}\int_0^{\frac{\pi}{m}} g(t, \zeta)c_m(t, \zeta) \ \mathrm{d}\zeta 
\end{equation}
and
\begin{equation}
    \label{eqn:pa-th-G(pi/m)}
    \partial_\theta G(t, {\textstyle{\frac{\pi}{m}}}) = \frac{\alpha}{P_0(0) b(t,\frac{\pi}{m})}\int_0^{\frac{\pi}{m}} g(t, \zeta)b(t, \zeta) \ \mathrm{d}\zeta 
\end{equation}
Alternatively, one can differentiate the formula \eqref{eqn:greens:reduced}. We can substitute these formulas in  \eqref{ineq:d/dt-g/P2} to deduce
\begin{align}
\label{ineq:diff:g/P2}
        \frac{\mathrm{d}}{\mathrm{d}t} \int_0^{\frac{\pi}{m}}\frac{g}{P^2} \ \mathrm{d}\theta &\le \frac{\alpha^2}{P_0^3(0)b^2(t, \frac{\pi}{m})}\Big(\sqrt{\alpha}\int_0^{\frac{\pi}{m}} gb \ \mathrm{d}\zeta + \int_0^{\frac{\pi}{m}} gc_m \ \mathrm{d}\zeta\Big)\Big( \sqrt{\alpha}\int_0^{\frac{\pi}{m}} gb \ \mathrm{d}\zeta - \int_0^{\frac{\pi}{m}} gc_m \ \mathrm{d}\zeta\Big)
\end{align}

We now split the right members of the inequality above in the following way:
\begin{equation*}
    \begin{aligned}
        \sqrt{\alpha}\int_0^{\frac{\pi}{m}} gb \ \mathrm{d}\zeta + \int_0^{\frac{\pi}{m}} gc_m \ \mathrm{d}\zeta &= \sqrt{\alpha}\int_0^{\frac{\pi}{m}} g_+b \ \mathrm{d}\zeta + \int_0^{\frac{\pi}{m}} g_+c_m \ \mathrm{d}\zeta \\
        &-\sqrt{\alpha}\int_0^{\frac{\pi}{m}} g_-b \ \mathrm{d}\zeta -\int_0^{\frac{\pi}{m}} g_-c_m \ \mathrm{d}\zeta
    \end{aligned}
\end{equation*}
and
\begin{equation*}
    \begin{aligned}
        \sqrt{\alpha}\int_0^{\frac{\pi}{m}} gb \ \mathrm{d}\zeta - \int_0^{\frac{\pi}{m}} gc_m \ \mathrm{d}\zeta &= \sqrt{\alpha}\int_0^{\frac{\pi}{m}} g_+b \ \mathrm{d}\zeta - \int_0^{\frac{\pi}{m}} g_+c_m \ \mathrm{d}\zeta \\
        &-\sqrt{\alpha}\int_0^{\frac{\pi}{m}} g_-b \ \mathrm{d}\zeta +\int_0^{\frac{\pi}{m}} g_-c_m \ \mathrm{d}\zeta
    \end{aligned}
\end{equation*}
Denote
\begin{equation*}
    \begin{aligned}
        I^+_1 &= \sqrt{\alpha}\int_0^{\frac{\pi}{m}} g_+b \ \mathrm{d}\zeta + \int_0^{\frac{\pi}{m}} g_+c_m \ \mathrm{d}\zeta\\
        I^+_2 &= \sqrt{\alpha}\int_0^{\frac{\pi}{m}} g_+b \ \mathrm{d}\zeta -\int_0^{\frac{\pi}{m}} g_+c_m \ \mathrm{d}\zeta \\
        I^-_1 &= -\sqrt{\alpha}\int_0^{\frac{\pi}{m}} g_-b \ \mathrm{d}\zeta -\int_0^{\frac{\pi}{m}} g_-c_m \ \mathrm{d}\zeta\\
        I^-_2 &= -\sqrt{\alpha}\int_0^{\frac{\pi}{m}} g_-b \ \mathrm{d}\zeta +\int_0^{\frac{\pi}{m}} g_-c_m \ \mathrm{d}\zeta
    \end{aligned}
\end{equation*}

Observing that $I_2^+>0$, and invoking \Cref{lem:b-c-prime}, for $j=1,2$ we have
\begin{equation}
\label{ineq:g-proof-Iminus}
\begin{aligned}
|I^-_j|
&=\left|-\sqrt{\alpha}\int_0^{\frac{\pi}{m}}g_-b\,\mathrm d\zeta
\pm\int_0^{\frac{\pi}{m}}g_-c_m\,\mathrm d\zeta\right|\\
&\le \sqrt{\alpha}\int_0^{\frac{\pi}{m}}g_-b\,\mathrm d\zeta
-\int_0^{\frac{\pi}{m}}g_-c_m\,\mathrm d\zeta\\
&\le \frac{\pi}{m}\|g_0\|_\infty
\left(\sqrt{\alpha}\,b(t,{\textstyle\frac{\pi}{m}})+|c_m(t,0)|\right)\\
&\le \frac{2\pi}{m}\|g_0\|_\infty\sqrt{\alpha}\,
 b(t,{\textstyle\frac{\pi}{m}}),
\end{aligned}
\end{equation}
where we used $|c_m(t,0)|=\alpha^{-1}b(t,\frac{\pi}{m})$.

For the remainder of the proof, set
\begin{equation*}
\mathcal Y(t):=
\int_0^{\frac{\pi}{m}}\frac{g(t,\theta)}{P^2(t,\theta)}\,\mathrm d\theta
+\frac{\pi\alpha^2}{mP_0^2(0)}\|g_0\|_\infty.
\end{equation*}
Since $g_-\le \|g_0\|_\infty$ and $P(t,\theta)\ge P_0(0)/\alpha$,
\begin{equation}
\label{ineq:gplus-by-Y}
0\le \int_0^{\frac{\pi}{m}}\frac{g_+}{P^2}\,\mathrm d\theta
=\int_0^{\frac{\pi}{m}}\frac{g}{P^2}\,\mathrm d\theta
 +\int_0^{\frac{\pi}{m}}\frac{g_-}{P^2}\,\mathrm d\theta
\le \mathcal Y(t).
\end{equation}
In particular $\mathcal Y(t)\ge0$. Moreover, using that $b$ is increasing, $c_m<0$ is increasing, and
$|c_m(t,0)|=\alpha^{-1}b(t,\frac{\pi}{m})$, we obtain, for $j=1,2$,
\begin{equation}
\label{ineq:g-proof-est1}
\begin{aligned}
\frac{\alpha^2}{P_0^2(0)b(t,\frac{\pi}{m})}|I^+_j|
&\le \frac{\alpha^2}{P_0^2(0)b(t,\frac{\pi}{m})}
\left(\sqrt{\alpha}\int_0^{\frac{\pi}{m}}g_+b\,\mathrm d\zeta
+\int_0^{\frac{\pi}{m}}g_+|c_m|\,\mathrm d\zeta\right)\\
&\le (\alpha^{5/2}+\alpha)
\int_0^{\frac{\pi}{m}}\frac{g_+}{P^2}\,\mathrm d\theta\\
&\le 2\alpha^{5/2}\mathcal Y(t).
\end{aligned}
\end{equation}
Thus
\begin{equation}
\label{ineq:diff:g/P2-split}
\frac{\mathrm d}{\mathrm dt}\int_0^{\frac{\pi}{m}}\frac{g}{P^2}\,\mathrm d\theta
\le \frac{\alpha^2}{P_0^3(0)b^2(t,\frac{\pi}{m})}
\left(I_1^+I_2^+ + I_1^+I_2^-+I_1^-I_2^+ + I_1^-I_2^-\right).
\end{equation}
Since
\[
I_1^+\le \sqrt{\alpha}\int_0^{\frac{\pi}{m}}g_+b\,\mathrm d\zeta,
\]
\eqref{ineq:g-proof-est1} gives
\begin{equation}
\label{ineq:g-proof-positive-product}
\frac{\alpha^2}{P_0^3(0)b^2(t,\frac{\pi}{m})}I_1^+I_2^+
\le \frac{2\alpha^3}{P_0(0)b(t,\frac{\pi}{m})}
\left(\int_0^{\frac{\pi}{m}}g_+b\,\mathrm d\zeta\right)\mathcal Y(t).
\end{equation}
For the mixed terms, \eqref{ineq:g-proof-Iminus} and \eqref{ineq:g-proof-est1} yield
\begin{equation}
\label{ineq:g-proof-mixed-products}
\begin{aligned}
&\frac{\alpha^2}{P_0^3(0)b^2(t,\frac{\pi}{m})}
\left(|I_1^+I_2^-|+|I_1^-I_2^+|\right)\\
&\hspace{2cm}\le
\frac{8\pi\alpha^3}{mP_0(0)}\|g_0\|_\infty\,\mathcal Y(t),
\end{aligned}
\end{equation}
and the negative-negative term satisfies
\begin{equation}
\label{ineq:g-proof-negative-product}
\frac{\alpha^2}{P_0^3(0)b^2(t,\frac{\pi}{m})}|I_1^-I_2^-|
\le \frac{4\pi^2\alpha^3}{m^2P_0^3(0)}\|g_0\|_\infty^2.
\end{equation}
Since the term added in the definition of $\mathcal Y$ is independent of time, combining
\eqref{ineq:diff:g/P2-split}--\eqref{ineq:g-proof-negative-product} gives
\begin{equation}
\label{ineq:Y-differential}
\mathcal Y'(t)
\le \left(
\frac{2\alpha^3}{P_0(0)b(t,\frac{\pi}{m})}
\int_0^{\frac{\pi}{m}}g_+b\,\mathrm d\zeta
+\frac{8\pi\alpha^3}{mP_0(0)}\|g_0\|_\infty
\right)\mathcal Y(t)
+\frac{4\pi^2\alpha^3}{m^2P_0^3(0)}\|g_0\|_\infty^2.
\end{equation}
Gr\"onwall's inequality therefore gives
\begin{equation}
\label{ineq:g-proof-gronwall}
\begin{aligned}
\mathcal Y(t)
&\le \left(\mathcal Y(0)
+\frac{4\pi^2\alpha^3}{m^2P_0^3(0)}\|g_0\|_\infty^2t\right)
\exp\left(
\frac{2\alpha^3}{P_0(0)}\mathcal I_b^+(t)
+\frac{8\pi\alpha^3}{mP_0(0)}\|g_0\|_\infty t
\right).
\end{aligned}
\end{equation}
Since $\int g/P^2\le\mathcal Y$, this proves the first claimed estimate.
Finally, using $P_0(\frac{\pi}{m})\le P(t,\theta)\le P_0(0)$, $P_0(\frac{\pi}{m})=P_0(0)/\alpha$, and the odd $m$-fold symmetry of $g$, we have
\begin{equation*}
\begin{aligned}
\|g(t,\cdot)\|_{L^1(\mathbb{S}^1)}
&=2m\int_0^{\frac{\pi}{m}}|g(t,\theta)|\,\mathrm{d}\theta\\
&\le 2mP_0^2(0)\int_0^{\frac{\pi}{m}}\frac{|g(t,\theta)|}{P^2(t,\theta)}\,\mathrm{d}\theta\\
&\le 2mP_0^2(0)\int_0^{\frac{\pi}{m}}\frac{g(t,\theta)}{P^2(t,\theta)}\,\mathrm{d}\theta
+4\pi\alpha^2\|g_0\|_\infty\\
&\le 2mP_0^2(0)\mathcal Y(t)+4\pi\alpha^2\|g_0\|_\infty.
\end{aligned}
\end{equation*}
Furthermore,
\begin{equation*}
2mP_0^2(0)\mathcal Y(0)
\le \alpha^2\|g_0\|_{L^1(\mathbb{S}^1)}+2\pi\alpha^2\|g_0\|_\infty.
\end{equation*}
Combining these inequalities with \eqref{ineq:g-proof-gronwall} gives the claimed $L^1$ estimate.
\end{proof}

\subsection{Existence and uniqueness for the 1D system}
\label{sec:1dsys-exist}
We now prove local existence and uniqueness for initial data $(g_0, P_0) \in C_m^{1,\gamma}(\mathbb{S}^1) \times C^{2,\gamma}_m(\mathbb{S}^1)$ satisfying \eqref{H1}-\eqref{H2}. In fact, uniqueness of solutions to \eqref{eqn:1dsys} with data in the class $C_m^{1,\gamma}(\mathbb{S}^1) \times C^{2,\gamma}_m(\mathbb{S}^1)$ follows immediately from Theorem~\ref{thm:lwp}, since each solution of \eqref{eqn:1dsys} defines a solution of \eqref{eqn:inho-2deuler:vort} with data in $\mathring{C}_m^{1, \gamma}(\mathbb{R}^2) \times \mathring{C}_m^{2, \gamma}(\mathbb{R}^2)$. It is also routine to show uniqueness directly for
\eqref{eqn:1dsys}. Hence, we focus here on the question of existence.
\subsubsection{Local existence}
\begin{proposition}
\label{prop:1dsys:lwp}
Let $\alpha=\alpha(P_0)$ and suppose that $m>4\sqrt{\alpha}$. Suppose $(g_0,P_0)\in C_m^{1,\gamma}(\mathbb{S}^1)\times C_m^{2,\gamma}(\mathbb{S}^1)$ satisfy \eqref{H1}--\eqref{H2}, and let $M$ be such that
\[
\max\{\|P_0\|_{C^{2,\gamma}},\|g_0\|_{C^{1,\gamma}}\}\le M.
\]
Then there exists $T=T(m,\gamma,\inf P_0,M)>0$ and a unique solution of \eqref{eqn:1dsys} satisfying
\[
g\in L^\infty(0,T;C_m^{1,\gamma})\cap C([0,T];C_m^{1,\beta}),
\qquad
P\in L^\infty(0,T;C_m^{2,\gamma})\cap C([0,T];C_m^{2,\beta})
\]
for every $0<\beta<\gamma$, and attaining the data $(g_0,P_0)$.
\end{proposition}
\begin{proof}
Let $g_0=g_0(\theta)$ and $P_0=P_0(\theta)$. For $n\ge 1$, define $(g_n,P_n)$ on $[0,\frac{\pi}{m}]$ inductively by
\begin{equation}
\label{eqn:iter:Pn}
   \begin{rcases}
        \partial_t P_n + 2G_n\partial_\theta P_n = 0,\\
        P_n(0, \theta) = P_0(\theta),
   \end{rcases}
\end{equation}
and
\begin{equation}
\label{eqn:iter:gn}
   \begin{rcases}
        \partial_t g_n + 2G_{n}\partial_\theta g_n = -\partial_\theta P_{n-1}(4G_n^2 + (\partial_\theta G_n)^2),\\
        g_n(0, \theta) = g_0(\theta),
   \end{rcases}
\end{equation}
where $G_n$ is the unique solution of the boundary value problem
\begin{equation}
    \label{eqn:itr:inhom-bvp}
    \begin{rcases}
   \partial_\theta (P_{n-1} \partial_\theta G_n) + 4P_{n-1}G_n = g_{n-1},\\
    G_n(t,0) = G_n(t,\frac{\pi}{m}) = 0.
\end{rcases}
\end{equation}
Thus,
\begin{equation}
\label{eqn:duhamel-g_n}
\begin{split}
    g_n(t, \theta) ={}& g_0(\phi^n(0, t,\theta)) \\
    &- \int_0^t \partial_\theta P_{n-1}(s,\phi^n(s,t,\theta))
    \big(4G_n^2 + (\partial_\theta G_n)^2\big)(s,\phi^n(s,t,\theta))\,\mathrm{d}s,
\end{split}
\end{equation}
while
\begin{equation*}
    P_n(t, \theta)=P_0(\phi^n(0,t,\theta)),
\end{equation*}
where $\phi^n$ is the flow map defined by
\begin{equation}
\label{eqn:flow-map_n}
    \frac{\mathrm{d}}{\mathrm{d}t} \phi^n(t,s,\theta)
    =2G_n(t,\phi^n(t,s,\theta)),
    \qquad \phi^n(s,s,\theta)=\theta.
\end{equation}

Fix $0<\beta<\gamma$ and assume inductively that
\[
P_{n-1}\in L^\infty(0,T;C^{2,\gamma})\cap C([0,T];C^{2,\beta}),
\qquad
g_{n-1}\in L^\infty(0,T;C^{1,\gamma})\cap C([0,T];C^{1,\beta}),
\]
and
\begin{equation}
\label{ineq:induct-hyp}
    \sup_{0\le t\le T}\|P_{n-1}(t)\|_{C^{2,\gamma}}\le 2M,
    \qquad
    \sup_{0\le t\le T}\|g_{n-1}(t)\|_{C^{1,\gamma}}\le 2M.
\end{equation}
The formula for $P_{n-1}$ along characteristics shows that its range agrees with that of $P_0$ and that its monotonicity on $[0,\pi/m]$ is preserved. In particular,
\[
    \inf P_{n-1}(t,\cdot)=\inf P_0,
    \qquad
    \alpha(P_{n-1}(t,\cdot))=\alpha(P_0).
\]
Hence the Green kernel bounds of Section~4.1 apply uniformly in $n$. In particular,
\begin{equation*}
    G_n(t,\theta)
    =\int_0^{\pi/m}H_m[P_{n-1}](\theta,\zeta)
    g_{n-1}(t,\zeta)\,\mathrm{d}\zeta.
\end{equation*}
Using the kernel estimates together with \eqref{eqn:itr:inhom-bvp}, we obtain
\begin{equation}
\label{ineq:Gn-holder}
    \sup_{0\le t\le T}\|G_n(t)\|_{C^{2,\gamma}}\le C_M,
\end{equation}
where throughout this proof $C_M$ denotes a positive constant depending only on $m$, $\gamma$, $\inf P_0$, and $M$, whose value may change from line to line. The inductive lower-exponent continuity and continuous dependence of the boundary value problem also give $G_n\in C([0,T];C^{2,\beta})$. Indeed, the kernel bounds first control $G_n$ and $\partial_\theta G_n$ uniformly; the equation
\[
    \partial_{\theta\theta}G_n
    =\frac{1}{P_{n-1}}
    \big(g_{n-1}-\partial_\theta P_{n-1}\partial_\theta G_n-4P_{n-1}G_n\big)
\]
first gives a uniform $C^2$ bound; hence $\partial_\theta G_n$ is uniformly Lipschitz, and applying the same identity once more gives the stated $C^{2,\gamma}$ bound.

The standard H\"older transport estimate applied to \eqref{eqn:iter:Pn} and \eqref{ineq:Gn-holder} yields
\begin{equation}
\label{ineq:P_nbnd}
    \|P_n(t)\|_{C^{2,\gamma}}
    \le e^{C_Mt}\|P_0\|_{C^{2,\gamma}}
    \le M e^{C_Mt}.
\end{equation}
Similarly, if
\[
    F_n:=-\partial_\theta P_{n-1}
    \big(4G_n^2+(\partial_\theta G_n)^2\big),
\]
then \eqref{ineq:induct-hyp} and \eqref{ineq:Gn-holder} imply
\[
    \|F_n(t)\|_{C^{1,\gamma}}\le C_M.
\]
The inhomogeneous H\"older transport estimate therefore gives
\begin{equation}
\label{ineq:g_nbnd}
    \|g_n(t)\|_{C^{1,\gamma}}
    \le e^{C_Mt}\left(\|g_0\|_{C^{1,\gamma}}+C_Mt\right)
    \le e^{C_Mt}(M+C_Mt).
\end{equation}
The characteristic formulas also give $P_n\in C([0,T];C^{2,\beta})$ and $g_n\in C([0,T];C^{1,\beta})$. Choose $T=T(m,\gamma,\inf P_0,M)>0$ sufficiently small that, for $0\le t\le T$, the right-hand sides of \eqref{ineq:P_nbnd} and \eqref{ineq:g_nbnd} are at most $2M$. This closes the induction and yields, uniformly in $n$,
\begin{equation}
\label{ineq:iter-uniform-holder}
    \sup_{0\le t\le T}
    \left(\|P_n(t)\|_{C^{2,\gamma}}+\|g_n(t)\|_{C^{1,\gamma}}
    +\|G_n(t)\|_{C^{2,\gamma}}\right)\le C_M.
\end{equation}

The equations \eqref{eqn:iter:Pn} and \eqref{eqn:iter:gn} also imply
\[
    \sup_{0\le t\le T}\|\partial_tP_n(t)\|_{C^{1,\gamma}}
    +\sup_{0\le t\le T}\|\partial_tg_n(t)\|_{C^{0,\gamma}}
    \le C_M.
\]
Combining these bounds with \eqref{ineq:iter-uniform-holder}, standard interpolation gives equicontinuity of $P_n$ in $C^2$ and of $g_n$ in $C^1$. Hence, after passing to a subsequence, Arzel\`a--Ascoli gives functions $g,P$ such that
\begin{equation*}
   \begin{cases}
    g_n \to g &\text{in } C([0,T];C^1([0,\frac{\pi}{m}])),\\
    P_n \to P &\text{in } C([0,T];C^2([0,\frac{\pi}{m}])).
   \end{cases}
\end{equation*}
Let $G(t,\cdot)$ be the unique solution of
\[
    \partial_\theta(P\partial_\theta G)+4PG=g,
    \qquad G(t,0)=G(t,\pi/m)=0.
\]
Continuous dependence for this boundary value problem, together with the preceding convergence, yields
\[
    G_n\to G \qquad\text{in } C([0,T];C^2([0,\tfrac{\pi}{m}])).
\]
We may now pass to the limit in
\begin{equation*}
    \begin{cases}
    g_n(t,\theta)=g_0(\theta)-\displaystyle\int_0^t
    \left[2G_n\partial_\theta g_n
    +\partial_\theta P_{n-1}(4G_n^2+(\partial_\theta G_n)^2)\right](s,\theta)\,\mathrm{d}s,\\[1ex]
    P_n(t,\theta)=P_0(\theta)-\displaystyle\int_0^t
    2G_n\partial_\theta P_n(s,\theta)\,\mathrm{d}s,\\[1ex]
    \partial_\theta(P_{n-1}\partial_\theta G_n)+4P_{n-1}G_n=g_{n-1},
    \end{cases}
\end{equation*}
to obtain
\begin{equation*}
    \begin{cases}
    g(t,\theta)=g_0(\theta)-\displaystyle\int_0^t
    \left[2G\partial_\theta g
    +\partial_\theta P(4G^2+(\partial_\theta G)^2)\right](s,\theta)\,\mathrm{d}s,\\[1ex]
    P(t,\theta)=P_0(\theta)-\displaystyle\int_0^t
    2G\partial_\theta P(s,\theta)\,\mathrm{d}s,\\[1ex]
    \partial_\theta(P\partial_\theta G)+4PG=g.
    \end{cases}
\end{equation*}
The uniform bounds \eqref{ineq:iter-uniform-holder} pass to the limit, so $g\in L^\infty(0,T;C^{1,\gamma})$ and $P\in L^\infty(0,T;C^{2,\gamma})$. Since $g\in C([0,T];C^1)$ and $P\in C([0,T];C^2)$, interpolation with these uniform endpoint bounds yields, for every $0<\beta<\gamma$,
\[
g\in C([0,T];C^{1,\beta}),\qquad P\in C([0,T];C^{2,\beta}).
\]
The equations also imply $g\in C^1([0,T];C^{0,\beta})$ and $P\in C^1([0,T];C^{1,\beta})$ for every $0<\beta<\gamma$.

Note that $g(t,0)=g(t,\frac{\pi}{m})=0$, $\partial_\theta P(t,0)=\partial_\theta P(t,\frac{\pi}{m})=0$, and $G(t,0)=G(t,\frac{\pi}{m})=0$. Finally extend $g(t,\theta)$ and $G(t,\theta)$ oddly to $\theta\in[-\frac{\pi}{m},\frac{\pi}{m}]$ and then periodically to $\mathbb S^1$; similarly, extend $P$ evenly and periodically. This concludes the proof of local existence.
\end{proof}
\subsubsection{Continuation criterion}
\label{sec:continuation-crit}
Let $T_*$ denote the maximal time of existence of the solution given by \Cref{prop:1dsys:lwp}. If the endpoint $C^{2,\gamma}\times C^{1,\gamma}$ norms remained bounded as $t\uparrow T_*$, the lower-exponent time continuity would give a terminal limit, the uniform endpoint bounds would pass to that limit, and \Cref{prop:1dsys:lwp} would restart the solution. Hence, if $T_*<+\infty$,
\[
\limsup_{t\to T_*^-}\max\{\|P(t,\cdot)\|_{C^{2,\gamma}},\|g(t,\cdot)\|_{C^{1,\gamma}}\}=+\infty.
\] As mentioned previously, control of $\|g_+(t,\cdot)\|_{L^1(\mathbb{S}^1)}$ yields control of $\max\{\|P(t,\cdot)\|_{{C}^{2, \gamma}}, \|g(t,\cdot)\|_{{C}^{1, \gamma}}\}$, so that if $T_* < +\infty$, we can also assert that
\begin{equation*}
    \limsup_{t \to T_*^{-}} \|g_+(t, \cdot)\|_{L^1(\mathbb{S}^1)} = +\infty.
\end{equation*}
The \textit{a priori} estimate of \Cref{prop:apriori-gL1} applies to the local solution from \Cref{prop:1dsys:lwp}. Since all the other quantities in that estimate remain finite on a finite time interval, if $T_*<+\infty$ then necessarily
\begin{equation}
\label{eqn:continuation-crit1}
\limsup_{t\to T^-_*}\mathcal I_b^+(t)= \limsup_{t\to T^-_*}\int_0^t \frac{1}{b(s,\frac{\pi}{m})}\int_0^{\frac{\pi}{m}} g_+(s,\theta)b(s,\theta)\,\mathrm{d}\theta\,\mathrm{d}s=+\infty.
\end{equation}
\begin{remark}
    Despite the criterion, the non-local nature of the system \eqref{eqn:1dsys} makes it difficult to verify that \eqref{eqn:continuation-crit1} holds under just hypotheses \eqref{H1} and \eqref{H2}. In fact, the system \eqref{eqn:1dsys} hides a Riccati structure. Denote $\mathcal{L}_P[G]:= \partial_\theta(P\partial_\theta G) + 4PG$, where $G$ satisfies periodic boundary conditions. Then under the assumptions of \Cref{sec:bvp}, we can define $\mathcal{L}_P^{-1}$. It can be shown that (formally)
\begin{equation}
\label{eqn:ricatti-structure}
    \partial_t(\partial_\theta G) + 2G\partial_\theta(\partial_\theta G) = (\partial_\theta G)^2 - 4\partial_\theta \mathcal{L}_P^{-1}[3P\partial_\theta^{-1}(\partial_\theta G)^2 + G^2\partial_\theta P].
\end{equation}
Thus the Riccati term $(\partial_\theta G)^2$ poses a potential obstruction to global existence, unless it can be shown that the term $4\partial_\theta \mathcal{L}_P^{-1}[3P\partial_\theta^{-1}(\partial_\theta G)^2 + G^2\partial_\theta P]$ dominates.
\end{remark}

\subsubsection{Proof of \Cref{thm:gwp-1dsys} - Global existence for admissible data}
\label{sec:gwppf}
We next turn to the task of obtaining control on $\mathcal I_b^+(t)$ under the assumption of admissibility on the initial data (\eqref{H1}-\eqref{H3}). By \Cref{eqn:continuation-crit1}, this will imply \Cref{thm:gwp-1dsys}. We first note two key estimates.

\bigskip
\noindent
\textbf{Bound on the homogeneous arc.} Temporarily denote $\gamma(t) := \phi(t, 0, \frac{\pi}{m}-\ve)$. 
Recall that 
 \begin{equation*}
            G(t, \theta) = \int_0^{\frac{\pi}{m}} H_m^t(\theta,\zeta) g(t, \zeta) \ \mathrm{d}\zeta
\end{equation*}
where $H_m^t(\cdot, \cdot)$ denotes the Green's kernel for \eqref{eqn:reduced-bvp} when $P=P(t, \cdot)$. We decompose $G$ as
\begin{equation*}
    \begin{aligned}
         G(t, \theta) &= \int_0^{\frac{\pi}{m}} H_m^t(\theta,\zeta) g_+(t, \zeta) \ \mathrm{d}\zeta - \int_0^{\frac{\pi}{m}} H_m^t(\theta,\zeta) g_-(t, \zeta) \ \mathrm{d}\zeta \\
         &:= G_l(t,\theta) +  G_r(t,\theta) \\
         &\le G_r(t,\theta),
    \end{aligned}
\end{equation*}
noting that $G_l(t,\theta)  \le 0$ and $G_r(t,\theta) \ge 0$. This implies that
\begin{equation}
\label{ineq:gamma:diff}
    \begin{aligned}
        \gamma'(t) &=  2G(t, \gamma(t)) \\
        &\le 2G_r(t, \gamma(t))\\
        &\le \frac{\pi^2\alpha^2}{mP_0(0)}\|g_0\|_\infty\left(\frac{\pi}{m}-\gamma(t)\right)
    \end{aligned}
\end{equation}
where the last inequality is justified by estimating as in (ii) in \Cref{prop:apriori-ests1},
\begin{equation*}
    \begin{aligned}
        |G_r(t, \theta)| = |G_r(t,\theta) - G(t, \pi/m)| &\le \|\partial_\theta G_r(t, \cdot) \|_\infty \left(\frac{\pi}{m}-\theta\right)\\
        &\le \frac{\pi\alpha^2}{2P_0(0)}\|g_-(t, \cdot)\|_{L^1}\left(\frac{\pi}{m}-\theta\right) \\
        &\le \frac{\pi^2\alpha^2}{2mP_0(0)}\|g_0\|_\infty\left(\frac{\pi}{m}-\theta\right).
    \end{aligned}
\end{equation*}
We have invoked $g_- \le \|g_0\|_\infty$ as well. 
Hence from \eqref{ineq:gamma:diff},
\begin{equation}
\label{ineq:gamma}
    0 < \gamma(t) \le \frac{\pi}{m}-\varepsilon e^{-\frac{\pi^2\alpha^2}{mP_0(0)}\|g_0\|_\infty t}
\end{equation}

Fix $T < T_*$. By hypothesis \eqref{H3}, $P_0' \equiv 0$ on $[\frac{\pi}{m}-\varepsilon , \frac{\pi}{m}]$. Since the flow is order preserving and $\frac{\pi}{m}$ is a fixed characteristic, the image of this interval at time $t$ is exactly $[\gamma(t),\frac{\pi}{m}]$. Hence
\begin{equation*}
\partial_\theta P(t,\theta)=0,\qquad \gamma(t)\le \theta\le \frac{\pi}{m}.
\end{equation*}
On this material interval the equation for $g$ reduces to the homogeneous transport equation. Consequently,
\begin{equation}
\label{ineq:g:hom:arc}
|g(t,\theta)|\le \|g_0\|_\infty,\qquad \gamma(t)\le \theta\le \frac{\pi}{m}.
\end{equation}

\bigskip
\noindent
\textbf{Bound on $G$ along characteristics.} Next, we observe the following bound obtained by integrating \eqref{eqn:flow-map}
\begin{equation*}
   \phi(t, 0, \theta) = \theta + \int_0^t 2G(s, \phi(s,0,\theta)) \ \mathrm{d}s
\end{equation*}
for $\theta \in [0, \frac{\pi}{m}]$.
Since $G(t, 0) = G(t, \frac{\pi}{m}) = 0$ for all $t$, all particle trajectories are constrained to $[0, \frac{\pi}{m}]$. Hence $0 \le \phi(t, 0, \theta) \le \frac{\pi}{m}$ and
\begin{equation*}
   0\le \theta + \int_0^t 2G(s, \phi(s,0,\theta)) \ \mathrm{d}s \le \frac{\pi}{m}
\end{equation*}
Hence 
\begin{equation}
\label{ineq:apriori-G-int}
   -\frac{\pi}{m}\le \int_0^t 2G(s, \phi(s,0,\theta)) \ \mathrm{d}s \le \frac{\pi}{m}
\end{equation}

The next proposition shows that if \eqref{H1}-\eqref{H3} hold, then the unique solution of \Cref{prop:1dsys:lwp} satisfies
\begin{equation*}
    \lim\sup_{t \to T_*^{-}} \|g(t, \cdot)\|_{L^1} < \infty,
\end{equation*}
where $T_*$ is its maximal existence time. It follows that $T_* = + \infty$.
\begin{proposition}
\label{prop:global-b-control}
Let $(g,P)$ be the maximal smooth solution to \eqref{eqn:1dsys} with admissible initial data $(g_0,P_0)$, defined on $[0,T_*)$. Then $\mathcal I_b^+(t)<\infty$ for each $0\le t<T_*$. Moreover, if $T_*<\infty$, then
\begin{equation*}
\limsup_{t\to T_*^-}\mathcal I_b^+(t)<\infty.
\end{equation*}
\end{proposition}
\begin{proof}
 By \eqref{eqn:greens:reduced}, we may write
\begin{equation*}
    G(t, \gamma(t)) = \frac{1}{P_0(0)|c_m(t, 0)|} \left(c_m(t,\gamma(t))\int_0^{\gamma(t)} g(t,\zeta)b(t,\zeta)\ \mathrm{d}\zeta + b(t,\gamma(t))\int_{\gamma(t)}^{\frac{\pi}{m}} g(t,\zeta)c_m(t,\zeta)\ \mathrm{d}\zeta \right)
\end{equation*}
and split $g= g_+ - g_{-}$ to replace the right-hand side above by
\begin{equation*}
    \begin{aligned}
            &\frac{1}{P_0(0)|c_m(t, 0)|} \left(c_m(t,\gamma(t))\int_0^{\gamma(t)} g_+(t,\zeta)b(t,\zeta)\ \mathrm{d}\zeta + b(t,\gamma(t))\int_{\gamma(t)}^{\frac{\pi}{m}} g_+(t,\zeta)c_m(t,\zeta)\ \mathrm{d}\zeta \right) \\
            &-\frac{1}{P_0(0)|c_m(t, 0)|} \left(c_m(t,\gamma(t))\int_0^{\gamma(t)} g_-(t,\zeta)b(t,\zeta)\ \mathrm{d}\zeta + b(t,\gamma(t))\int_{\gamma(t)}^{\frac{\pi}{m}} g_-(t,\zeta)c_m(t,\zeta)\ \mathrm{d}\zeta \right) \\
            &=: I_+(t) - I_-(t)
    \end{aligned}
\end{equation*}
By \eqref{ineq:apriori-G-int} and $g_- \le \|g_0\|_\infty$, we have
\begin{equation*}
\begin{aligned}
       \left|\int_0^t I_+(s)\ \mathrm{d}s\right| &=  \left|\int_0^t I_-(s)\ \mathrm{d}s+\int_0^t G(s, \gamma(s))\ \mathrm{d}s\right| \\
       &\le\left|\int_0^t I_-(s)\ \mathrm{d}s\right|+\left|\int_0^t G(s, \gamma(s))\ \mathrm{d}s\right| \\
       &\le\left|\int_0^t I_-(s)\ \mathrm{d}s\right|+\frac{\pi}{m} \\
       &\le C\|g_0\|_\infty t+\frac{\pi}{m}
\end{aligned}
\end{equation*}
where $C$ depends only on $m, \alpha, P_0(0)$. 

Observe from the signs of $b$ and $c_m$ in \Cref{lemma:abd-props} that $I_+(s)\le0$. Thus

\begin{equation}
\label{eqn:abs-G}
\begin{aligned}
 \left| I_+(s)\right| &= \frac{1}{P_0(0)|c_m(s, 0)|} \Big(|c_m(s,\gamma(s))|\int_0^{\gamma(s)} g_+(s,\zeta)b(s,\zeta)\ \mathrm{d}\zeta\\
 &+ b(s,\gamma(s))\int_{\gamma(s)}^{\frac{\pi}{m}} g_+(s,\zeta)|c_m(s,\zeta)|\ \mathrm{d}\zeta \Big)
 \end{aligned}
\end{equation}

It follows in particular from \eqref{eqn:abs-G} that
\begin{equation}
\label{ineq:g-proof-est2}
    \int_0^t \frac{|c_m(s,\gamma(s))|}{P_0(0)|c_m(s,0)|}\int_0^{\gamma(s)} g_+(s,\zeta)b(s,\zeta)\ \mathrm{d}\zeta\ \mathrm{d}s \le C\|g_0\|_\infty t+\frac{\pi}{m}
\end{equation}
Now by \Cref{lem:b-c-ptwise} and \eqref{ineq:gamma},
\begin{align*}
|c_m(s,\gamma(s))|
&\ge \frac{1}{2\alpha}\sin\left(2\left(\frac{\pi}{m}-\gamma(s)\right)\right)\\
&\ge \frac{2}{\pi\alpha}\left(\frac{\pi}{m}-\gamma(s)\right)\\
&\ge \frac{2\varepsilon}{\pi\alpha}
 e^{-\frac{\pi^2\alpha^2}{mP_0(0)}\|g_0\|_\infty s}.
\end{align*}
Returning to \eqref{ineq:g-proof-est2}, we obtain
\begin{equation}
\label{ineq:g-proof-est3}
\int_0^t \frac{1}{P_0(0)|c_m(s,0)|}\int_0^{\gamma(s)} g_+(s,\zeta)b(s,\zeta)\,\mathrm{d}\zeta\,\mathrm{d}s
\le \frac{\pi\alpha}{2\varepsilon}\left(C\|g_0\|_\infty t+\frac{\pi}{m}\right)
 e^{\frac{\pi^2\alpha^2}{mP_0(0)}\|g_0\|_\infty t}.
\end{equation}
Since $\alpha|c_m(s,0)|=b(s,\frac{\pi}{m})$, this is equivalent to
\begin{equation}
\label{ineq:g-proof-est4}
\int_0^t \frac{1}{P_0(0)b(s,\frac{\pi}{m})}\int_0^{\gamma(s)} g_+(s,\zeta)b(s,\zeta)\,\mathrm{d}\zeta\,\mathrm{d}s
\le \frac{\pi}{2\varepsilon}\left(C\|g_0\|_\infty t+\frac{\pi}{m}\right)
 e^{\frac{\pi^2\alpha^2}{mP_0(0)}\|g_0\|_\infty t}.
\end{equation}
On the homogeneous arc, we use $g(t,\theta) \le \|g_0\|_{L^\infty(\mathbb{S}^1)}$ (i.e. \eqref{ineq:g:hom:arc}):
\begin{equation}
\label{ineq:g-proof-est5}
    \int_0^t \frac{1}{P_0(0)|b(s,\frac{\pi}{m})|}\int_{\gamma(s)}^{\frac{\pi}{m}} g_+(s,\zeta)b(s,\zeta)\ \mathrm{d}\zeta\ \mathrm{d}s \le \frac{t}{P_0(0)} \frac{\pi}{m} \|g_0\|_{L^\infty(\mathbb{S}^1)}
\end{equation}
Combining \eqref{ineq:g-proof-est4} and \eqref{ineq:g-proof-est5}, and multiplying by $P_0(0)$, we conclude that
\begin{equation*}
\mathcal I_b^+(t)
\le \frac{\pi}{m}\|g_0\|_\infty t
+\frac{\pi P_0(0)}{2\varepsilon}\left(C\|g_0\|_\infty t+\frac{\pi}{m}\right)
 e^{\frac{\pi^2\alpha^2}{mP_0(0)}\|g_0\|_\infty t}<\infty.
\end{equation*}
The right-hand side is bounded uniformly for $0\le t<T_*$ whenever $T_*<\infty$, which gives the claimed bound on $\limsup_{t\to T_*^-}\mathcal I_b^+(t)$. This proves the proposition.
\end{proof}

\subsection{Homogenization and average relaxation in a special case}
\label{sec:long-time}
In this section, we mention an observation on the long-time behavior of $g(t, \theta)$ when the initial data $g_0$ is  signed. We will show below that either the system homogenizes, or $g$ relaxes to zero in a weak sense. Assume hence that $g_0(\theta) \ge 0$ and $P_0' \le 0$ on $[0, \frac{\pi}{m}]$ are admissible. Then $g$ satisfies
\begin{equation*}
    \partial_t g + 2G \partial_\theta g \ge 0
\end{equation*}
for all $t$, and thus $g(t, \theta) \ge 0$. By \Cref{prop:G-greens-formula}, $G(t,\theta) \le 0$ on $(0, \frac{\pi}{m})$ for all $t \ge 0$. This implies that particle trajectories are nonincreasing in time, and consequently
\begin{equation*}
    \phi_{\mathrm{final}}(\theta):=\lim_{t\to\infty} \phi(t,0,\theta)
\end{equation*}
exists for each $\theta \in [0, \frac{\pi}{m}]$. Either for all $\theta \in (0, \frac{\pi}{m})$, $\phi_{\mathrm{final}}(\theta) = 0$ (i.e. the system homogenizes and $P(t,\theta) \to P_0(\frac{\pi}{m})$ for each $\theta \in (0, \frac{\pi}{m}]$) or there exists $\theta_0 \in (0, \frac{\pi}{m})$ with $0<\phi_{\mathrm{final}}(\theta_0) < \frac{\pi}{m}$. In the latter case, we observe the following.

First, by integrating the flow map,
\begin{equation*}
   \phi(t, 0, \theta) = \theta + \int_0^t 2G(s, \phi(s,0,\theta)) \ \mathrm{d}s
\end{equation*}
for $\theta \in [0, \frac{\pi}{m}]$.
Since $G(t, 0) = G(t, \frac{\pi}{m}) = 0$ for all $t$, all particle trajectories are constrained to $[0, \frac{\pi}{m}]$. Hence $0 \le \phi(t, 0, \theta) \le \frac{\pi}{m}$ and
\begin{equation*}
   0\le \theta + \int_0^t 2G(s, \phi(s,0,\theta)) \ \mathrm{d}s \le \frac{\pi}{m}
\end{equation*}
Hence (recall $|G| = -G$) 
\begin{equation}
\label{ineq:apriori-G-int:signed-case}
    \int_0^t 2|G(s, \phi(s,0,\theta))| \ \mathrm{d}s \le \frac{\pi}{m}
\end{equation}
Denote temporarily $\gamma(t) = \phi(t,0,\theta_0)$. Since $-G = |G|$ and
\begin{equation*}
    G(t, \gamma(t)) = \frac{1}{P_0(0)|c_m(t, 0)|} \left(c_m(t,\gamma(t))\int_0^{\gamma(t)} g(t,\zeta)b(t,\zeta)\mathrm{d}\zeta + b(t,\gamma(t))\int_{\gamma(t)}^{\frac{\pi}{m}} g(t,\zeta)c_m(t,\zeta)\mathrm{d}\zeta \right),
\end{equation*}
we have
\begin{equation}
    |G(t, \gamma(t))| = \frac{1}{P_0(0)|c_m(t, 0)|} \Big(|c_m(t,\gamma(t))|\int_0^{\gamma(t)} gb\ \mathrm{d}\zeta + b(t,\gamma(t))\int_{\gamma(t)}^{\frac{\pi}{m}} g|c_m|\ \mathrm{d}\zeta \Big)
\end{equation}
and by \eqref{ineq:apriori-G-int:signed-case}, we see in particular that the inequalities
\begin{equation*}
    \begin{aligned}
        \int_0^t\frac{|c_m(s,\gamma(s))|}{|c_m(s, 0)|}\int_0^{\gamma(s)} g(s,\zeta)b(s,\zeta)\ \mathrm{d}\zeta \ \mathrm{d}s &\le \frac{\pi P_0(0)}{2m} \\
         \int_0^t\frac{b(s,\gamma(s))}{|c_m(s, 0)|}\int_{\gamma(s)}^{\frac{\pi}{m}} g(s,\zeta)|c_m(s,\zeta)|\ \mathrm{d}\zeta&\le \frac{\pi P_0(0)}{2m}
    \end{aligned}
\end{equation*}
hold for all $t$. By \Cref{lem:b-c-ptwise}, we have $|c_m(s,\theta)| \gtrsim (\frac{\pi}{m}-\theta)$ and $b(s,\theta) \gtrsim \theta$. Furthermore, since $0<\lim_{t\to\infty}\gamma(t) < \frac{\pi}{m}$ by assumption, we have $(\frac{\pi}{m}-\gamma(s)) \gtrsim 1$ and $\gamma(s)\gtrsim 1$ for all $s$ and so,
\begin{equation*}
    \begin{aligned}
        \int_0^t\int_0^{\gamma(s)} g(s,\zeta)\zeta\ \mathrm{d}\zeta \ \mathrm{d}s &\lesssim 1 \\
         \int_0^t\int_{\gamma(s)}^{\frac{\pi}{m}} g(s,\zeta)\left(\frac{\pi}{m}-\zeta\right)\ \mathrm{d}\zeta\ \mathrm{d}s &\lesssim 1
    \end{aligned}
\end{equation*}
with the implied constants depending only on fixed parameters. Let $0 < \delta \ll 1$ be sufficiently small. Then the above inequalities imply that
\begin{equation*}
    \begin{aligned}
         \int_0^t\int_{\delta}^{\frac{\pi}{m}-\delta} g(s,\zeta)\ \mathrm{d}\zeta \ \mathrm{d}s &= \int_0^t\int_{\delta}^{\gamma(s)} g(s,\zeta)\ \mathrm{d}\zeta \ \mathrm{d}s+\int_0^t\int_{\gamma(s)}^{\frac{\pi}{m}-\delta} g(s,\zeta)\ \mathrm{d}\zeta \ \mathrm{d}s \\
         &\lesssim \delta^{-1}\left( \int_0^t\int_0^{\gamma(s)} g(s,\zeta)\zeta\ \mathrm{d}\zeta \ \mathrm{d}s+  \int_0^t\int_{\gamma(s)}^{\frac{\pi}{m}} g(s,\zeta)\left(\frac{\pi}{m}-\zeta\right)\ \mathrm{d}\zeta\ \mathrm{d}s\right) \\
         &\lesssim \delta^{-1}
    \end{aligned}
\end{equation*}
Letting $t \to \infty$, we have
\begin{equation}
\label{ineq:relaxation}
       \int_0^\infty\int_{\delta}^{\frac{\pi}{m}-\delta} g(t,\theta)\ \mathrm{d}\theta \ \mathrm{d}t \lesssim \delta^{-1} < \infty.
\end{equation}
Thus \eqref{ineq:relaxation} says that if the system does not homogenize, then away from the boundaries, $g$ tends to zero on average in the sense that its $L^1$-norm is time integrable. Note that the foregoing argument works assuming only global-existence, positivity and hypotheses \eqref{H1} and \eqref{H2}. Thus we have shown the following proposition.
\begin{proposition}
Let $(g, P)$ be a global solution to \eqref{eqn:1dsys} with initial data $(g_0, P_0)$ satisfying \eqref{H1} and \eqref{H2}. Assume that $g_0 \ge 0$ and $P_0' \le 0$, so that $g \ge 0$ for all $t \ge 0$. Either 
\begin{equation*}
    P(t,\theta) \to P_0(\pi/m) = \inf P_0 \text{ for all } \theta \in (0, \pi/m]
\end{equation*}
or for each $\delta > 0$, sufficiently small,
\begin{equation*}
       \int_0^\infty\int_{\delta}^{\frac{\pi}{m}-\delta} g(t,\theta)\ \mathrm{d}\theta \ \mathrm{d}t  < \infty.
\end{equation*}
\end{proposition}

\section*{Acknowledgments}
The author thanks Tarek Elgindi immensely for introducing him to this topic. This work is a product of numerous useful discussions with him. The relevance of the theory of quasiconformal mappings and the reference \cite{astala2008elliptic} was found in conversation with Gemini Pro. We used ChatGPT 5.6 Sol for routine editorial work on an earlier human-generated draft of this manuscript, and to generate the figures.  

\appendix
\section{Some calculus inequalities in $\mathring{C}^{k,\gamma}(\mathbb{R}^2)$}
\label{app:a}
\begin{lemma}[Product rule]
\label{lem:prod-rule}
   Let $f, g\in \mathring{C}^{0, \gamma}(\mathbb{R}^2)$, $0 < \gamma \le 1$. Then $fg \in \mathring{C}^{0,\gamma}(\mathbb{R}^2)$and 
    \begin{equation*}
        \|fg\|_{\mathring{C}^{0, \gamma}} \le C\|f\|_{\mathring{C}^{0, \gamma}}\|g\|_{\mathring{C}^{0, \gamma}}
    \end{equation*}
\end{lemma}
\begin{proof}
We treat the case $\gamma < 1$, leaving $\gamma = 1$ for the reader.
    Trivially, $\|fg\|_{L^\infty} \le \|f\|_{L^\infty}\|g\|_{L^\infty}$. Next, we assume without loss of generality that $|x| > |y|$. Then
    \begin{equation*}
       |x|^\gamma f(x)g(x) - |y|^\gamma f(y)g(y) = g(x)(|x|^\gamma f(x) - |y|^\gamma f(y)) + (|y|^\gamma -|x|^\gamma)g(x)f(y) +f(y)(|x|^\gamma g(x) - |y|^\gamma g(y))
    \end{equation*}
    Hence since $\left|\frac{|y|^\gamma -|x|^\gamma}{|x-y|^\gamma}\right| \le C_\gamma$, we have
    \begin{equation*}
        \left|\frac{|x|^\gamma f(x)g(x) - |y|^\gamma f(y)g(y)}{|x-y|^\gamma}\right| \le \|g\|_{L^\infty}[f]_{\mathring{C}^{0,\gamma}} + C_\gamma\|f\|_{L^\infty}\|g\|_{L^\infty} +  \|f\|_{L^\infty}[g]_{\mathring{C}^{0,\gamma}}
    \end{equation*}
    The conclusion is now immediate, in view of \eqref{def:angular-Holder}.
\end{proof}

Let us also note the following useful equivalence.
\begin{lemma}
    Let $\gamma \in (0, 1]$ and let $f \in \mathring{C}^{0, \gamma}(\mathbb{R}^2)$. Then
        \begin{equation}
    \label{eqn:C-ring-norm-eq}
               \frac{1}{2}\|f\|_{\mathring{C}^{0,\gamma}} \le  \|f\|_{L^\infty} + \sup_{x\ne y} \min\{|x|^\gamma, |y|^\gamma\}\frac{|f(x)-f(y)|}{|x-y|^\gamma} \le 2\|f\|_{\mathring{C}^{0,\gamma}}
    \end{equation}
    and
    \begin{equation*}
       \frac{1}{2}\|f\|_{\mathring{C}^{0,\gamma}} \le  \|f\|_{L^\infty} + \sup_{x\ne y} \max\{|x|^\gamma, |y|^\gamma\}\frac{|f(x)-f(y)|}{|x-y|^\gamma} \le 2\|f\|_{\mathring{C}^{0,\gamma}}.
    \end{equation*}
\end{lemma}
\begin{proof}
    Let $x, y \in \RR^2$, $x \ne y$ and suppose with no loss of generality that $|x| \le |y|$. Then \eqref{eqn:C-ring-norm-eq} follows from the identity
    \begin{align*}
        \frac{|x|^\gamma f(x) - |y|^\gamma f(y) }{|x-y|^\gamma} = |x|^\gamma \frac{f(x) - f(y)}{|x-y|^\gamma} + \frac{|x|^\gamma - |y|^\gamma}{|x-y|^\gamma}f(y)
    \end{align*}
    and
    \begin{equation*}
        \frac{||x|^\gamma - |y|^\gamma|}{|x-y|^\gamma} \le 1.
    \end{equation*}
    The second inequality is handled similarly.
\end{proof}

\begin{lemma}
\label{lem:composition}
    Let $\Psi:\mathbb{R}^2\to\mathbb{R}^2$ be a bi-Lipschitz homeomorphism with $\Psi(0)=0$. Then, for $0<\gamma\le1$,
    \begin{equation*}
        \|f\circ\Psi\|_{\mathring{C}^{0,\gamma}}
        \le C_\gamma\left[1+\big(\operatorname{Lip}(\Psi)\operatorname{Lip}(\Psi^{-1})\big)^\gamma\right]
        \|f\|_{\mathring{C}^{0,\gamma}}.
    \end{equation*}
\end{lemma}
\begin{proof}
    By \eqref{eqn:C-ring-norm-eq}, it suffices to estimate the equivalent seminorm involving $\min\{|x|^\gamma,|y|^\gamma\}$. Since $\Psi(0)=0$ and $\Psi$ is bi-Lipschitz,
    \begin{equation*}
        |x|\le \operatorname{Lip}(\Psi^{-1})|\Psi(x)|,
        \qquad
        |\Psi(x)-\Psi(y)|\le \operatorname{Lip}(\Psi)|x-y|.
    \end{equation*}
    Hence
    \begin{align*}
       &\min\{|x|^\gamma,|y|^\gamma\}\frac{|f(\Psi(x))-f(\Psi(y))|}{|x-y|^\gamma}\\
       &\qquad\le \big(\operatorname{Lip}(\Psi)\operatorname{Lip}(\Psi^{-1})\big)^\gamma
       \min\{|\Psi(x)|^\gamma,|\Psi(y)|^\gamma\}
       \frac{|f(\Psi(x))-f(\Psi(y))|}{|\Psi(x)-\Psi(y)|^\gamma}.
    \end{align*}
    Together with $\|f\circ\Psi\|_{L^\infty}=\|f\|_{L^\infty}$, another application of \eqref{eqn:C-ring-norm-eq} yields the claim.
\end{proof}

The same norm equivalence also shows that if we restrict the $\mathring{C}^{0,\gamma}$ norm to annuli $A_{r_1, r_2} = \{ r_1\le |x| \le r_2\}$ where $0 < r_1 < r_2$, then it is equivalent to the ${C}^{0,\gamma}$ norm. Precisely, define
\begin{equation*}
  \|f\|_{\mathring{C}^{0,\gamma}(A_{r_1, r_2})} =   \|f\|_{L^\infty(A_{r_1, r_2})} + \sup_{\substack{x\ne y\\ x, y \in A_{r_1, r_2}}} \max\{|x|^\gamma, |y|^\gamma\}\frac{|f(x) - f(y)|}{|x-y|^\gamma}
\end{equation*}
Then
\begin{equation}
\label{ineq:holder-equiv}
   \min\{1,r_1^\gamma\} \|f\|_{{C}^{0,\gamma}(A_{r_1,r_2})} \le  \|f\|_{\mathring{C}^{0,\gamma}(A_{r_1, r_2})} \le \max\{1,r_2^\gamma\} \|f\|_{{C}^{0,\gamma}(A_{r_1,r_2})}
\end{equation}

{\small
\bibliographystyle{plain}
\bibliography{refs.bib}
}

\end{document}